\documentclass[11pt,reqno]{amsart}
\usepackage{amsmath,amssymb,graphicx,mathrsfs,amsopn,stmaryrd,amsthm,color,bm,extarrows}
\usepackage[dvipsnames]{xcolor}
\usepackage[colorlinks=true,citecolor=blue,linkcolor=blue
]{hyperref}
\usepackage[final]{showlabels}
\usepackage[english]{babel}
\usepackage[mathscr]{euscript}
\usepackage{geometry}
\usepackage{collectbox}
\makeatletter
\newcommand{\mybox}{%
    \collectbox{%
        \setlength{\fboxsep}{1pt}%
        \fbox{\BOXCONTENT}%
    }%
}
\makeatother

\usepackage{dsfont}
\usepackage{tikz}    
\usetikzlibrary{arrows,matrix,decorations.markings,shapes.geometric}   
\usetikzlibrary{decorations.pathreplacing}

\allowdisplaybreaks

\numberwithin{equation}{section}
\newtheorem{thm}{Theorem}[section]
\newtheorem{prop}[thm]{Proposition}
\newtheorem{lem}[thm]{Lemma}
\newtheorem{cor}[thm]{Corollary}

\theoremstyle{definition} 

\newtheorem{dfn}[thm]{Definition}
\theoremstyle{remark}
\newtheorem{rem}[thm]{Remark}

\newcommand{\beq}{\begin{equation}}
\newcommand{\eeq}{\end{equation}}
\newcommand{\be}{\begin{equation*}}
\newcommand{\ee}{\end{equation*}}

\newcommand{\bC}{\mathbb{C}}

\newcommand{\bZ}{\mathbb{Z}}
\newcommand{\bK}{\mathbb{K}}

\newcommand{\bS}{\mathbb{S}}
\newcommand{\bB}{\mathbb{B}}

\newcommand{\bN}{\mathbb{N}}

\newcommand{\mc}{\mathcal}

\newcommand{\cR}{\mathcal{R}}

\newtheorem{innercustomthm}{{\bf Theorem}}
	\newenvironment{customthm}[1]
	{\renewcommand\theinnercustomthm{#1}\innercustomthm}
	{\endinnercustomthm}

\newcommand{\g}{\mathfrak{g}}
\newcommand{\gl}{\mathfrak{gl}}

\newcommand{\fkS}{\mathfrak{S}}

\newcommand{\rY}{\mathrm{Y}}

\newcommand{\End}{\mathrm{End}}
\newcommand{\id}{{\mathrm{id}}}   
   
\newcommand{\gr}{{\mathrm{gr}}}

\newcommand{\tl}{\tilde}
\newcommand{\wtl}{\widetilde}
\newcommand{\gge}{\geqslant}
\newcommand{\lle}{\leqslant}

\newcommand{\wt}{\widehat}    
\newcommand{\ka}{\kappa}

\newcommand{\Sym}{{\mathrm{Sym}}}
\newcommand{\Y}{{\mathscr{Y}}}
\newcommand{\scrX}{{\mathscr{X}}}

\newcommand{\SY}{{\mathscr{SY}}}

\newcommand{\rU}{{\mathrm{U}}}
\newcommand{\Yi}{{\mathbf{Y}^\imath}}
\newcommand{\bv}{{\bm{v}}}

\newcommand{\qbinom}[2]{\begin{bmatrix} #1\\#2 \end{bmatrix} }

\newcommand{\arxiv}[1]{\href{http://arxiv.org/abs/#1}{\tt arXiv:\nolinkurl{#1}}}
\def \bTH { \boldsymbol{\Theta}}

\def \bDel{ \boldsymbol{\Delta}}
\def \bB{\mathbf{B}}
\def \bh{\mathbf{h}}
\def \bb{\mathbf{b}}

\def \h{\hbar}
\def \H{h}
\def \X{b}

\def \Gr{\mathrm{Gr}}
\def \TH{\Theta}

\def \I{\mathbb{I}}
\def \N{\mathbb{N}}
\def \TH{\Theta}
\def \Sym{\mathrm{Sym}}
\def \U{\mathbf U}
\def \Ui{\mathbf{U}^\imath}
\def \tUi{\mathbf{U}^\imath}
\def \Q{\mathbb{Q}}
\def \K{\mathbb{K}}
\def \Z{\mathbb{Z}}
\def \bK{\mathbf{K}}
\def \C{\mathbb{C}}
\def \ov{\overline}
\def \a{\mathfrak{a}}

\def \A{\mathcal{A}}
\def \hg{\mathfrak{g}[t,t^{-1}]}

\begin{document}
\pagestyle{myheadings}
\setcounter{page}{1}

\title[Twisted Yangians of quasi-split type]{A Drinfeld type presentation of twisted Yangians\\
of quasi-split type}

\author{Kang Lu}
\address{Department of Mathematics, University of Virginia, 
Charlottesville, VA 22903, USA}
\email{kang.lu@virginia.edu}

\author{Weinan Zhang}
\address{Department of Mathematics and New Cornerstone
Science Laboratory, The University of Hong Kong, Hong Kong SAR, P.R.China}
\email{mathzwn@hku.hk}

		
\begin{abstract}
We formulate a family of algebras, twisted Yangians (of simply-laced quasi-split type) in Drinfeld type current generators and defining relations. These new algebras admit PBW type bases and are shown to be a deformation of twisted current algebras. For all quasi-split types excluding the even rank case in type AIII, we show that the twisted Yangians can be realized via a degeneration on the Drinfeld type presentation of affine $\imath$quantum groups. For both even and odd rank cases in type AIII, we use the Gauss decomposition method to show that these new algebras are isomorphic to Molev-Ragoucy's reflection algebras defined in the R-matrix presentation. 
\end{abstract}
	
\maketitle
\setcounter{tocdepth}{1}
\tableofcontents

\thispagestyle{empty}
\section{Introduction}
\subsection{Background}
Yangians first emerged (for $\gl_N$ in R-matrix form) in the work of Faddeev and the St. Petersburg school in the late 1970s and early 1980s concerning the quantum inverse scattering method. Later, Drinfeld introduced in \cite{Dr88} a remarkable new current presentation for Yangians. Two fundamental methods have been widely used in the study of the Drinfeld presentation for Yangians: the first method indicates that Yangians can be obtained from affine quantum groups through a suitable degeneration \cite{Dr87,GM12,GTL13}; the second one is the Gauss decomposition approach, which is used to establish an explicit isomorphism between the Drinfeld presentation and the R-matrix presentation \cite{BK05,JLM18,FT23}. 

The twisted Yangians, associated with symmetric pairs, are coideal subalgebras of Yangians. Inspired by \cite{Ch84,Sk88}, they were introduced in the R-matrix presentation by Olshanski \cite{Ol92} for type AI and AII, by Molev-Ragoucy \cite{MR02} for type AIII, and by Guay-Regelskis \cite{GR16} for other classical types. Twisted Yangians and their representation theory have been well studied in the R-matrix presentation; cf. \cite{MNO96,Na04,Mol07,GRW17}. Beyond the classical types, the R-matrix presentation of twisted Yangians remains unknown; these algebras are currently only available in their Drinfeld J-presentation form \cite{Ma02, BR17}, which applies to general symmetric pairs.

In contrast, the Drinfeld type presentation for twisted Yangians was unknown for a long time, until our recent work joint with Weiqiang Wang \cite{KLWZ23a,KLWZ23b}, where we formulated a  Drinfeld type presentation for twisted Yangians of arbitrary split type. Here, split type means that the involution for the underlying symmetric pair is conjugated to the Chevalley involution $\varpi$. For split type A (i.e., type AI), we established in \cite{KLWZ23a} an isomorphism between the twisted Yangian in Drinfeld type presentation and in R-matrix type presentation, using the Gauss decomposition approach. 

Affine $\imath$quantum groups $\Ui$ form another distinguished class of quantum algebras associated to symmetric pairs. They are coideal subalgebras of affine quantum groups, arising from the theory of quantum symmetric pairs of Kac-Moody type \cite{Ko14}. The Drinfeld type presentation for $\Ui$ was first constructed for split ADE types in \cite{LW21}, and generalized to split BCFG types in \cite{Z22}, and further generalized for a class of quasi-split types in \cite{LWZ23a,LWZ23b}. In \cite{KLWZ23b}, we showed that the Drinfeld type presentation for twisted Yangian of split types can be obtained from the Drinfeld type presentation of $\Ui$ via a degeneration approach. 

\subsection{Goal}

Let $\a$ be a Kac-Moody algebra. A symmetric pair $(\a,\a^\omega)$ (of Kac-Moody type) is of quasi-split type  if the involution $\omega$ has the form $\omega=\varpi\circ \tau$, where $\tau$ is an involution of the underlying Dynkin diagram. (In particular, $\omega$ is of split type if in addition $\tau=\id$.) Given an involution $\omega_0$ on a finite type Lie algebra $\g$, one can lift it to an involution $\omega$ on the untwisted affine Lie algebra $\widehat{\g}$ by \eqref{def:thetahat}. The quasi-split symmetric pairs of affine type associated with a nontrivial $\tau$ obtained by such lifting are classified as AIII$_{n}$, DII$_{n}$, EII$_{6}$, where ADE corresponds to the type of $\g$ and $n$ represents the rank of $\g$. 

The goal of this paper is to provide a Drinfeld current presentation for twisted Yangians of quasi-split types AIII$_{n}$, DII$_{n}$,  EII$_{6}$, generalizing our previous work with Wang for split type \cite{KLWZ23a,KLWZ23b}. Note that our scope exactly includes all quasi-split affine types with a nontrivial $\tau$ fixing the affine simple node. 

\subsection{Drinfeld type presentation for twisted Yangians}\label{intro-main}

We formulate in Definition~\ref{deftY} the Drinfeld type presentation for twisted Yangians $\Yi$ in terms of current generators $\{\H_{i,r},\X_{i,r},i\in \mathbb{I}^0,r\in \N\}$ and defining current relations \eqref{qsconj0}-\eqref{qsconj10} for type AIII$_{n}$, DII$_{n}$,  EII$_{6}$. Most defining relations \eqref{qsconj0}-\eqref{qsconj8} are natural quasi-split generalizations of their split type counterparts in \cite[(4.1)-(4.6)]{KLWZ23b}. As shown in later sections, these relations can be viewed as the degenerate version of the Drinfeld type defining relations for affine $\imath$quantum groups. 
On the contrary, the Serre relation \eqref{qsconj10} is essentially new, since it only appears when $\tau$ is nontrivial. We show that \eqref{qsconj10} can be obtained from the finite type Serre relation inductively using the other defining relations; see Section~\ref{sec:qsSerre}.

In the next two subsections, we shall elaborate on that our Drinfeld type presentation fits into two frameworks: the degeneration approach and the Gauss decomposition approach. Along the way, we shall establish a PBW type basis for $\Yi$ in terms of the current generators, which indicates that $\Yi$ is a flat deformation of twisted current algebras; see Theorem~\ref{thm:pbw} and Corollary~\ref{cor:PBW2}.

\subsection{Degeneration approach} The degeneration approach has the advantage that it works uniformly for any types provided the Drinfeld type presentation for the corresponding affine $\imath$quantum groups is available; see \cite{GM12,GTL13,KLWZ23b}. 

Drinfeld type presentations for quasi-split affine $\imath$quantum groups $\Ui$ of types AIII$_{2n-1}$, DII$_{n}$,  EII$_{6}$ have been constructed in \cite{LWZ23b} (which were denoted by AIII$_{2n-1}^{(\tau)}$, DI$_{n}^{(\tau)}$,  EI$_{6}^{(\tau)}$ in that paper). We recall this result in Proposition~\ref{thm:qsiDR}. We show that $\Yi$ for these types can be realized as a suitable degenerate version of the Drinfeld type presentation of $\Ui$.
More precisely, generalizing the construction in \cite{CG15,KLWZ23b}, we define a filtration on $\Ui$ in Section~\ref{sec:filtration} and write $\Gr_{\bK}\tUi$ for the associated graded algebra. We have the following result.
\begin{customthm} {\bf A} [Theorem~\ref{thm:main1}]
\label{th:A}
Let $\Yi$ be of type AIII$_{2n-1}$, DII$_{n}$, EII$_{6}$. Then there is an algebra isomorphism
\begin{align*}
\begin{split}
\Phi: &\, \Yi \longrightarrow \Gr_{\bK}\tUi, 
\\
&\, \X_{i,r}  \mapsto \ov{\X}_{i,r,1},\quad
\H_{i,r} \mapsto \ov{\H}_{i,r,1}, 
\end{split}
\end{align*}
for $i\in \I^0,r \in \N$, where $\X_{i,r,1}, \H_{i,r,1} \in \tUi$ are defined in \eqref{def:qsHX} and $\ov{\X}_{i,r,1},\ov{\H}_{i,r,1}$ are their images in $\Gr_{\bK}\tUi$. 
\end{customthm}

In Section~\ref{sec:qsdegen}, we verify that the defining relations of $\Yi$ are satisfied by $\ov{\X}_{i,r,1},\ov{\H}_{i,r,1}$ in $\Gr_{\bK}\tUi$,  showing that the map $\Phi$ is a homomorphism. Following similar arguments in \cite[Proof of Theorem 4.10]{KLWZ23b} (cf. also \cite{GM12}), $\Phi$ is an isomorphism and we establish a PBW basis for $\Yi$ in terms of Drinfeld type generators in Theorem~\ref{thm:pbw}.

Drinfeld type presentations for affine $\imath$quantum groups of type AIII$_{2n},n>1$ are not known so far, while the Drinfeld type presentation for the affine $\imath$quantum group of type AIII$_{2}$ has been established in \cite{LWZ23a}. Nevertheless, due to the complexity of the Serre relations {\em loc. cit.}, it is technically much more difficult to formulate a degenerate version for the Drinfeld type presentation of type AIII$_{2}$ and compare it with our definition of $\Yi$. Hence, we will not include type AIII$_{2n}$ in our degeneration approach. Note that, this case distinguishes with the others in that it contains two connecting nodes in a single $\tau$-orbit and does not have a $\tau$-fixed point except the affine node. 

\subsection{Gauss decomposition approach} 
Fortunately, we have another approach, the Gauss decomposition, which is widely used to establish the isomorphism between Drinfeld and R-matrix presentations for Yangians and affine quantum groups of classical type; see e.g. \cite{BK05,JLM18,FT23}. We shall deal with AIII$_{N-1}$ via the Gauss decomposition approach uniformly for arbitrary $N> 1$ regardless of the parity of $N$.

We start by recalling the twisted Yangian $\Y_{N}$ of type AIII$_{N-1}$ in the R-matrix presentation from \cite{MR02}. Instead of using the diagonal matrix (a solution of reflection equation) in \cite{MR02,KLWZ23a}, we use the anti-diagonal matrix $G$ as the initial matrix to define $\Y_N$; see Section~\ref{subsec:tY} for details. It turns out that, in this new setting, the Yangian of type A (of smaller rank) is naturally a subalgebra of the twisted Yangian $\Y_{N}$. The special twisted Yangian $\SY_{N}$ is naturally a subalgebra of $\Y_{N}$. Following the general Gauss decomposition approach (cf. \cite{BK05,KLWZ23a}), we are able to establish an isomorphism between twisted Yangians in the Drinfeld type presentation and in R-matrix presentation. 

\begin{customthm} {\bf B} [Theorem~\ref{main2}]
\label{th:B}
Let $N> 1$ be arbitrary and $\Yi$ be of type AIII$_{N-1}$. Then there is an algebra isomorphism,
\beq
\begin{split}
\Phi: &\,\Yi\to \SY_{N},\\
&\, h_{i,r}\mapsto h_{i,r},\quad b_{i,r}\mapsto b_{i,r},
\end{split}
\eeq
for $i\in \I^0$, $r\in\bN$. Here the elements $h_{i,r}, b_{i,r}$ in $\SY_{N}$ are defined by \eqref{beven}-\eqref{hodd}.
\end{customthm}
In particular, for the case AIII$_{2n-1}$, $\Yi$ defined via the Drinfeld type presentation fits into both degeneration picture and the Gauss decomposition approach.

In the quasi-split AIII case, our construction presents several distinctions from our previous work with Wang \cite{KLWZ23a}. Firstly, unlike the AI case, there are no natural embeddings known from twisted Yangians of lower rank into twisted Yangians of higher rank in the AIII case. We overcome this obstacle by constructing the shift homomorphisms $\psi_{m}$ in Proposition~\ref{prop:red}, following a similar approach in \cite{JLM18}. It is noteworthy that our shift homomorphisms only exist from $\Y_{M}$ to $\Y_N$ when $M,N$ have the same parity. This parity condition did not show up in the previous work \cite{KLWZ23a} for the AI case. Secondly, when $N$ is odd, the parameter shifts in our  definition \eqref{bodd}--\eqref{hodd} of $h_{i,r}$ involve a multiple of $\frac{1}{4}$, and this is novel compared to those in \cite{BK05,JLM18,FT23,KLWZ23a} or other literature. In addition, an extra factor $1\pm\frac{1}{4u}$ is needed for the Cartan currents \eqref{hodd} corresponding to the middle two nodes. This modification is crucial for making the relations simple and deriving the new Serre relations \eqref{qsconj10}.

In Section~\ref{sec:lowrk}, we verify that the elements $h_{i,r}, b_{i,r}$ in $\SY_{N}$ satisfy the defining relations of $\Yi$ for the low rank cases $N=2,3,4,5$. The general rank case follows by applying the shift homomorphism $\psi_{m}$. This implies that $\Phi$ is an algebra homomorphism, and we show that it is an isomorphism by establishing and using a PBW basis of $\SY_N$; see the proof of Theorem~\ref{main2}.

Meanwhile, we describe in Theorem~\ref{thm:center} the center of twisted Yangian $\Y_N$ in terms of Gaussian generators. It would be interesting to study the relation between the Sklyanin determinant \cite[Thm. 3.4]{MR02} and Gaussian generators, cf. \cite[Thm. 2.12.1]{Mol07}.


We expect that by carefully choosing the matrix form (K-matrix), the trigonometric counterpart of Theorem \ref{th:B} can be established via a Gauss decomposition of the corresponding twisted quantum loop algebra \cite{CGM14}. This may provide a Drinfeld type presentation for the affine $\imath$quantum groups of type AIII$_{n}$, including the type AIII$_{2n}$ which remains open for $n>1$, cf. \cite{Lu24} for the case of split type A.

\subsection{Application}
We expect the following applications for the Drinfeld type presentation of twisted Yangians.

Yangians in the Drinfeld presentation for simply-laced type have been realized geometrically in \cite{Va00} using the equivariant Borel-Moore homology. The Drinfeld type presentation of twisted Yangians for quasi-split type AIII will be crucial for their geometric realization via the equivariant Borel-Moore homology of Steinberg varieties of classical type; see \cite{SW24} for a geometric realization of affine $\imath$quantum groups. This can be further exploited to investigate the geometric Schur-Weyl duality between twisted Yangians and degenerate affine Hecke algebras of type B/C; cf. \cite{CGM14}.


Yet in another direction, the Drinfeld type presentation allows one to define the $q$-character of twisted Yangians (cf. \cite{Kni:1995,FR99}), which should be an important tool to study finite-dimensional representations and tensor products (with a Yangian module) for twisted Yangians.

\subsection{Organization}
The paper is organized as follows. In the preliminary Section \ref{sec:pre}, we review the degeneration from loop algebras to current algebras and investigate their twisted analogs for quasi-split type. 
In Section \ref{sec:tY}, we introduce twisted Yangians of quasi-split ADE type via a Drinfeld type presentation and study some of their basic properties, including the PBW theorem. In Section \ref{sec:degen}, we recall the affine $\imath$quantum groups and establish the isomorphism between twisted Yangians and a certain associated graded algebra of affine $\imath$quantum groups (except for type AIII$_{2n}$). 

Section \ref{sec:R} is devoted to the review of
twisted Yangians and extended twisted Yangians of type AIII$_n$ in the R-matrix presentation. In Section \ref{sec:GD}, we study the shift homomorphisms of extended twisted Yangians from lower rank to higher rank and introduce Drinfeld type current generators arising from the Gauss decomposition for $\Y_N$. 
This section concludes with an isomorphism between the twisted Yangians defined via the Drinfeld type presentation and the ones defined via the R-matrix presentation in Theorem~\ref{main2}, as well as a description for the center of the twisted Yangian $\Y_N$ in terms of Cartan current generators. Finally,  in Section \ref{sec:lowrk}, we verify various relations in the (extended) twisted Yangians for $N=2,3,4,5$, and this completes the proof of Theorem~\ref{main2}.

\vspace{3mm}

\noindent {\bf Acknowledgement.} We are grateful to Weiqiang Wang for valuable discussions, and to Yaolong Shen and Rui Xiong for their comments on an earlier version of this paper. KL is partially supported by Wang's NSF grant DMS--2401351. WZ is partially supported by the New Cornerstone Foundation through the
New Cornerstone Investigator grant, and by Hong Kong RGC grant 14300021, both awarded to Xuhua He.

\section{Twisted loop and current algebras}\label{sec:pre}
Throughout the paper, we use the notation $\bN:=\{0,1,2,\cdots\}$.
\subsection{From loop algebras to current algebras}
Let $\I=\{0,1,\dots,n\}$ and $\I^0=\{1,\dots,n\}$. Let $\wt{\g}$ be an untwisted affine Kac-Moody algebra of type ADE with Cartan matrix $C=(c_{ij})_{i,j\in\I}$. Let $\g$ be the corresponding simple Lie algebra with Cartan matrix $C^0=(c_{ij})_{i,j\in\I^0}$. 

Let $\cR$ be the root system of $\g$, $\{\alpha_i~|~i\in \I^0\}$ a simple system of $\cR$, and $\cR^+$ the corresponding set of positive roots. Recall that $\wt{\g}$ (with the degree operator dropped) has a loop algebra realization that fits into a short exact sequence of Lie algebras,
\beq\label{ghat}
0 \longrightarrow \C c \longrightarrow \wt{\g} \longrightarrow \mathfrak{g}[t,t^{-1}]\longrightarrow 0.
\eeq

The (enveloping algebra of) current algebra $\g[z]$ can be obtained from the (enveloping algebra of) loop algebra $\g[t,t^{-1}]$ via a degeneration as follows; see \cite{KLWZ23a} and cf. \cite{GM12} for a quantum counterpart. Write $g_k:=gt^k$ for any $g\in \g,k\in \Z$. Define $g_{r,k}\in \g[t,t^{-1}]$ for $g\in \g, k\in \Z, r\in\bN$ by
\begin{align}\label{xrk}
g_{r,k}:=\sum_{s=0}^r (-1)^{r-s} \binom{r}{s} g_{s+k}=g(t-1)^r t^k.
\end{align}

For $r\in\bN$, set $\kappa_r$ to be the ideal of $ \rU\big(\g[t,t^{-1}]\big)$ generated by elements $(g_1)_{r_1,k_1} \cdots (g_a)_{r_a,k_a}$ with $g_i\in \g,r_i\in\bN, k_i\in \Z, 1\lle i\lle a, r_1+\ldots + r_a \gge r, a\in\bN$. Then $\rU\big(\g[t,t^{-1}]\big)$ admits a decreasing filtration 
\begin{align}  \label{filtration}
\rU\big(\g[t,t^{-1}]\big) =\kappa_0 \supset \kappa_1 \supset \cdots \supset \kappa_r \supset \cdots, 
\end{align}
and we denote the associated graded algebra by 
\[
\Gr_{\kappa}\g[t,t^{-1}]:=\bigoplus_{r\gge 0} \kappa_r/ \kappa_{r+1}. 
\]
It is clear from the definition that
$g_{r,k}\in \kappa_r\setminus \kappa_{r+1}.$
Write $\ov{g_{r,k}}$ for the image of $g_{r,k}$ in  $\kappa_r / \kappa_{r+1}$. Since $g_{r+1,k}=g_{r,k+1}-g_{r,k}$, we have that 
\begin{align}\label{climit}
\ov{g_{r,k}}=\overline{g_{r,k+1}} \in  \Gr_{\kappa}\g[t,t^{-1}],\qquad  \text{ for }k\in \Z.
\end{align}

\begin{lem}[{\cite[Prop. 2.1]{KLWZ23b}}]\label{prop:iso}
There is an algebra isomorphism $\varrho: \rU(\g[z]) \stackrel{\cong}{\longrightarrow} \Gr_{\kappa}\g[t,t^{-1}]$, sending $g z^r\mapsto \ov{g_{r,1}}$, for $g\in \g,r\in\bN$. 
\end{lem}

\subsection{Twisted loop and current algebras}\label{sec:twloop}
Let $\{e_i,f_i,h_i\}_{i\in\I^0}$ be the Chevalley generators of $\g$. Let $\tau$ be a Dynkin diagram (associated with $\g$) involution, i.e. $c_{ij}=c_{\tau i,\tau j}$ and $\tau^2=\mathrm{id}$. We consider the involution 
\[
\omega_0:\g \to \g,\qquad  e_{i}\mapsto -f_{\tau i},~ f_i\mapsto -e_{\tau i},~h_i\mapsto -h_{\tau i},
\]
for $i\in \I^0$. Denote by $\g^{\omega_0}$ the $\omega_0$-fixed point subalgebra of $\g$. Then $(\g,\g^{\omega_0})$ forms a quasi-split symmetric pair (of finite type).

We extend the involution $\omega_0$ to an involution $\omega$ on the loop algebra $\hg$ by the rule:
\beq
\label{def:thetahat}
\omega:\hg\to \hg,\qquad gt^k\mapsto \omega_0(g)t^{-k},\qquad g\in \g,k\in\Z.
\eeq
The involution $\omega$ extends to an involution, which we denote again by $\omega$, on $\wt{\g}$ in \eqref{ghat} by mapping $c\mapsto -c$. The involution $\tau$ lifts to an involution (which is also denoted by $\tau$) on the affine Dynkin diagram which fixes the affine node. Then $\omega$ can be identified with the composition of the Chevalley involution on $\wt{\g}$ and the affine Dynkin diagram involution $\tau$.

\begin{lem}
We have the following standard results.
\begin{enumerate}
    \item The fixed point subalgebra $\hg^\omega$ of $\hg$ is generated by
    \beq\label{hggen}
    \theta_{i,k}:=h_{\tau i}t^k-h_it^{-k},\qquad \flat_{i,k}:=f_it^{-k}-e_{\tau i}t^k,\qquad i\in \I^0,k\in\Z.
    \eeq
    \item The fixed point subalgebra $\hg^\omega$ coincides with the fixed point subalgebra $\wt{\g}^\omega$ of $\wt{\g}$.
\end{enumerate}
\end{lem}
\begin{proof}
The proof is parallel to the proof of \cite[Lem. 2.2]{KLWZ23b}.
\end{proof}

We call the algebra $\hg^\omega\cong \wt{\g}^\omega$ a \emph{twisted loop algebra}.

We extend the involution $\omega_0$ on $\g$ to an involution $\check{\omega}$ on the current algebra $\g[z]$ by the rule:
\beq\label{omeche}
\check\omega:\g[z]\to \g[z],\qquad g z^r\mapsto \omega_0(g)(-z)^r,\qquad g\in\g, r\in\bN.
\eeq
We call the fixed point subalgebra $\g[z]^{\check\omega}$ of $\g[z]$ a \emph{twisted current algebra}.

As in \cite[\S2]{KLWZ23b}, we introduce a filtration and the corresponding associated graded for $\rU(\hg^\omega)$ that will be used to exhibit a degeneration from $\rU(\hg^\omega)$ to $\rU(\g[z]^{\check\omega})$.

For each $r\in\bN$, denote by $\ka^\imath_r$ the ideal $\ka_r \cap \rU(\hg^\omega)$ of $\rU(\hg^\omega)$. Then $\rU(\hg^\omega)$ admits a descending filtration,
\beq  \label{filtration2}
\rU(\g[t,t^{-1}]^\omega) =\kappa_0^\imath \supset \kappa_1^\imath \supset \cdots \supset \kappa_r^\imath \supset \cdots. 
\eeq
Note that the two filtrations \eqref{filtration} and \eqref{filtration2} are compatible, i.e., $\ka^\imath_r=\ka_r \cap \rU(\hg^\omega)$, therefore the associated graded of \eqref{filtration2}, 
$$
\Gr_{\ka^\imath}\hg^\omega:=\bigoplus_{r\gge 0} \ka_r^\imath / \ka_{r+1}^\imath,
$$ 
can be regarded naturally as a subalgebra of $\Gr_{\ka}\hg$:
\beq  \label{GrGr}
\Gr_{\ka^\imath}\hg^\omega  \subset \Gr_{\ka}\hg.
\eeq

For $r\in\bN$ and $k\in\Z$, define
\beq\label{theta}
\theta_{i,r,k}:=\sum_{s=0}^{r}(-1)^{r-s}{r \choose s}\theta_{i,s+k},\qquad \flat_{i,r,k}:=\sum_{s=0}^r (-1)^{r-s}{r \choose s}\flat_{i,s+k}
\eeq
Then we have the following
\beq\label{theta-exp}
\begin{split}
\theta_{i,r,k}&=h_{\tau i}(t-1)^rt^k-h_{i}(1-t)^{r}t^{-r-k},\\
\flat_{i,r,k}&=f_i(1-t)^rt^{-r-k}-e_{\tau i}(t-1)^rt^k.
\end{split}
\eeq
Let $\bar{\theta}_{i,r,k},\bar{\flat}_{i,r,k}$ be their images in the $r$-th component of $\Gr_{\ka^\imath}\hg^\omega$, respectively.

Recall the isomorphism $\varrho:\rU(\g[z])\to \Gr_\ka\hg$ from Lemma \ref{prop:iso} and note \eqref{GrGr}.
\begin{lem}
The following statements hold:
\begin{enumerate}
    \item For $r\in \N$, we have  $\theta_{i,r,k}\in \ka_{r}^\imath$ and $\varrho^{-1}(\bar\theta_{i,r,k})=(h_{\tau i}-(-1)^rh_{i})z^r$.
    \item We have $\flat_{i,r,k}\in \ka_k^\imath\setminus \ka_{k+1}^\imath$ and $\varrho^{-1}(\bar\flat_{i,r,k})=\big((-1)^rf_{i}-e_{\tau i}\big)z^r$.
\end{enumerate}
\end{lem}

\begin{proof}
This lemma follows from \eqref{theta-exp}.
\end{proof}

Note that $\theta_{i,r+1,k}=\theta_{i,r,k+1}-\theta_{i,r,k}$ and $\flat_{i,r+1,k}=\flat_{i,r,k+1}-\flat_{i,r,k}$.
\begin{cor}
The following identities hold in $\Gr_{\ka^\imath}\hg^\omega$:
\[
\bar\theta_{i,r,k}=\bar\theta_{i,r,k+1},\qquad \bar\flat_{i,r,k}=\bar\flat_{i,r,k+1},
\]
for $k\in\Z$ and $r\in\N$.
\end{cor}
 
\subsection{Presentations of $\rU(\g[z]^{\check\omega})$}


In this subsection, we formulate two presentations for $\mathrm{U}(\g[z]^{\check\omega})$.
Recall the Cartan matrix $C^0=(c_{ij})_{i,j\in\I^0}$ for $\g$. Throughout the paper, we use $\Sym_{k_1,k_2}$ to denote the symmetrization with respect to indices $k_1, k_2$ in the sense
\[
\Sym_{k_1,k_2} f(k_1, k_2) =f(k_1, k_2) +f(k_2, k_1).
\]
\begin{prop}\label{prop:qscrel}
The algebra $\rU(\g[z]^{\check\omega})$ is generated by $\xi_{i,r},x_{i,r}$ for $i\in \I^0,r\in\bN$, subject to the relations \eqref{eq:qsclassical1}--\eqref{eq:qsclassical0}, for $i,j\in\I^0$, $r,s\in \bN$:
\begin{align}
\label{eq:qsclassical1}
&[\xi_{i,r},\xi_{j,s}]=0,\qquad \xi_{i,r} =(-1)^{r+1} \xi_{\tau i,r},
\\
\label{eq:qsclassical2}
&[\xi_{i,r},x_{j,s}]=\big(c_{ij}-(-1)^r c_{\tau i,j} \big) x_{j,r+s},
\\
\label{eq:qsclassical3}
&[x_{i,r+1},x_{j,s}]-[x_{i,r},x_{j,s+1}]=-2\delta_{j,\tau i} (-1)^r \xi_{\tau i,r+s+1}, 
\\
&[x_{i,r},x_{\tau i,s}]=(-1)^r \xi_{\tau i,r+s},\qquad\qquad \qquad  c_{i,\tau i}=0,\label{eq:qsclassical0}
\end{align}
and the Serre relations \eqref{eq:qsclassical4}-\eqref{eq:qsclassical5}: for $c_{i,\tau i}=-1$, (the left-hand side is symmetric with respect to $k_1,k_2$)
\begin{align}
\big[x_{i,k_1},[x_{i,k_2},x_{\tau i,r}]\big]  
= \big((-1)^{k_1} +(-1)^{k_2} +2(-1)^r\big) x_{i,k_1+k_2+r},
\label{eq:qsclassical4}
\end{align}
and for $j\neq i,\tau i,$
\begin{align}
[x_{i,k},x_{j,r}]&=0,\hskip5.13cm c_{ij}=0,\label{eq:qsclassical6}
\\
\Sym_{k_1,k_2}\big[x_{i,k_1},[x_{i,k_2},x_{j,r}]\big] & =0,\hskip5.13cm c_{ij}=-1 \text{ and }i\ne \tau i,\\
\Sym_{k_1,k_2}\big[x_{i,k_1},[x_{i,k_2},x_{j,r}]\big] & =-((-1)^{k_1}+(-1)^{k_2})x_{j,k_1+k_2+r},\quad c_{ij}=-1 \text{ and }i=\tau i.\label{eq:qsclassical5}
\end{align}
\end{prop} 
\begin{proof}
Let $\mathscr L$ be the algebra with generators $\xi_{i,r},x_{i,r}$ and defining relations \eqref{eq:qsclassical1}--\eqref{eq:qsclassical5}. Then it is straightforward to check that the map
\begin{align*}
\rho:\mathscr L\to \rU(\g[z]^{\check \omega}),\quad 
\xi_{i,r}\mapsto (h_{\tau i}-(-1)^rh_i)z^r,\quad x_{i,r}\mapsto ((-1)^rf_i-e_{\tau i})z^r,
\end{align*}
defines an algebra homomorphism. 

The rest of the proof is similar to that of \cite[Prop.~2.5]{KLWZ23b}.
\end{proof}

The following lemma indicates the left-hand side of \eqref{eq:qsclassical4} is symmetric with respect to $k_1,k_2$.
\begin{lem}\label{lem:com-1}
If $c_{i,\tau i}=-1$, then $[x_{i,r},x_{i,s}]=0$.
\end{lem}
\begin{proof}
It is immediate from \eqref{eq:qsclassical3} as $c_{i,\tau i}=-1$ implies $i\ne \tau i$.
\end{proof}

We also have the reduced presentation for $\rU(\g[z]^{\check\omega})$ which is analogous to \cite[Prop.~2.8]{KLWZ23b}.
\begin{prop}\label{prop:redu}
The algebra $\mathrm U(\g[z]^{\check\omega})$ is generated by $\xi_{i,r},x_{i,r}$ for $i\in \I^0,r\in\bN$, subject to the relations \eqref{eq:qsclassical1}--\eqref{eq:qsclassical0} and the finite type Serre relations \eqref{finserre1}--\eqref{finserre4}: for $c_{i,\tau i}=-1$,
\beq\label{finserre1}
\big[x_{i,0},[x_{i,0},x_{\tau i,0}]\big]  
= 4x_{i,0},
\eeq
and for $j\neq i,\tau i,$
\begin{align}
[x_{i,0},x_{j,0}]&=0,\hskip 2.2cm c_{ij}=0,\label{finserre2}
\\
\big[x_{i,0},[x_{i,0},x_{j,0}]\big]& =0,\hskip 2.2cm c_{ij}=-1 \text{ and }i\ne \tau i, \label{finserre3}\\
\big[x_{i,0},[x_{i,0},x_{j,0}]\big]& =-x_{j,0},\hskip 1.5cm c_{ij}=-1 \text{ and }i=\tau i.\label{finserre4}
\end{align}
\end{prop}
\begin{proof}
By Proposition \ref{prop:qscrel}, it suffices to show that the relations \eqref{eq:qsclassical4}--\eqref{eq:qsclassical5} can be deduced from the relations \eqref{eq:qsclassical1}--\eqref{eq:qsclassical0} together with \eqref{finserre1}--\eqref{finserre4}. The deduction of relations \eqref{eq:qsclassical6}--\eqref{eq:qsclassical5} is similar to that of \cite[Prop.~2.8]{KLWZ23b}. Here we only provide the details for  \eqref{eq:qsclassical4}.

Now we assume that $c_{i,\tau i}=-1$. We first show that
\beq\label{00r}
\big[x_{i,0},[x_{i,0},x_{\tau i,r}]\big]=\big(2+2(-1)^r\big)x_{i,r}.
\eeq
Bracketing \eqref{finserre1} by $\xi_{i,r}$, we obtain that 
\beq\label{00r1}
2\big[x_{i,0},[x_{i,r},x_{\tau i,0}]\big]-(-1)^r\big[x_{i,0},[x_{i,0},x_{\tau i,r}]\big]=4x_{i,r},
\eeq
where we also used Lemma \ref{lem:com-1}. Note that by \eqref{eq:qsclassical3}, we have
\beq\label{00r2}
[x_{i,r},x_{\tau i,s}]=[x_{i,0},x_{\tau i,r+s}]+((-1)^r-1)\xi_{\tau i,r+s}.
\eeq
Plugging \eqref{00r2} (with $s=0$) into \eqref{00r1} and using \eqref{eq:qsclassical2}, one proves \eqref{00r}. 

Now let us consider the general case and we shall use Lemma \ref{lem:com-1}. Then we have
\begin{align*}
&\ \big[x_{i,k_1},[x_{i,k_2},x_{\tau i,r}]\big]\\
\stackrel{\eqref{00r2}}{=}&\ \big[x_{i,k_1},[x_{i,0},x_{\tau i,k_2+r}]\big]+\big((-1)^{k_2}-1\big)\big[x_{i,k_1},\xi_{\tau i,k_2+r}\big]\\
=\hskip 0.24cm& \ \big[x_{i,0},[x_{i,k_1},x_{\tau i,k_2+r}]\big]+\big((-1)^{k_2}-1\big)\big[x_{i,k_1},\xi_{\tau i,k_2+r}\big]\\
\stackrel{\eqref{eq:qsclassical2}}{=} & \ \big[x_{i,0},[x_{i,k_1},x_{\tau i,k_2+r}]\big]+\big((-1)^{k_2}-1\big)\big(1+2(-1)^{k_2+r}\big)x_{i,k_1+k_2+r}\\
\stackrel{\eqref{00r2}}{=}&\ \big[x_{i,0},[x_{i,0},x_{\tau i,k_1+k_2+r}]\big]+\big((-1)^{k_1}-1\big)\big[x_{i,0},\xi_{\tau i,k_1+k_2+r}\big]\\
&\hskip 3.95cm +\big((-1)^{k_2}-1\big)\big(1+2(-1)^{k_2+r}\big)x_{i,k_1+k_2+r}\\
\stackrel{\eqref{eq:qsclassical2}}{=}&\ \big[x_{i,0},[x_{i,0},x_{\tau i,k_1+k_2+r}]\big]+\big((-1)^{k_1}-1\big)\big(1+2(-1)^{k_1+k_2+r}\big)x_{i,k_1+k_2+r}\\
&\hskip 3.95cm +\big((-1)^{k_2}-1\big)\big(1+2(-1)^{k_2+r}\big)x_{i,k_1+k_2+r}.
\end{align*}
Now the relation \eqref{eq:qsclassical4} is obtained from the above equation and the special case \eqref{00r}.
\end{proof}

\section{Twisted Yangians in Drinfeld type presentations}\label{sec:tY}
In this section, we formulate in Definition~\ref{deftY} a Drinfeld type presentation for twisted Yangians $\Yi$ of quasi-split AIII$_{n}$, DII$_{n}$, EII$_{6}$ types. In Section~\ref{sec:pbw}, we formulate in Theorem~\ref{thm:pbw} a PBW basis for $\Yi$ in terms of current root vectors. We provide the generating function formulations for the defining current relations in Section~\ref{sec:geniQG}, and show that the general Serre relation can be deduced from the finite type Serre relation with the help of other defining current relations in Section~\ref{sec:qsSerre}.

\subsection{Twisted Yangians of quasi-split type}
Recall the Cartan matrix $(c_{ij})_{i,j\in\I^0}$ for $\g$ and write $\{x,y\}=xy+yx$.
\begin{dfn}\label{deftY}
The \emph{twisted Yangian of quasi-split type} $(\I,\tau)$, denoted by $\Yi$, is the $\mathbb{C}[\h]$-algebra generated by $\H_{i,r},\X_{i,r},i\in \mathbb{I}^0,r\in \N$, subject to
\begin{align}\label{qsconj0}
&\quad [\H_{i,r},\H_{j,s}]=0,\qquad\qquad  h_{\tau i,0}=-h_{i,0},
\\\label{qsconj5'}
&\quad [\H_{i,0},\X_{j,r}]=  (c_{ij}-c_{\tau i,j}) \X_{j,r},
\\\label{qsconj5}
&\quad [\H_{i,1},\X_{j,r}]=  (c_{ij}+c_{\tau i,j}) \X_{j,r+1}+\frac{\h(c_{ij}-c_{\tau i,j})}{2}\{\H_{i,0},\X_{j,r}\},
\\\notag
&\quad [\H_{i,r+2},\X_{j,s}] - [\H_{i,r},\X_{j,s+2}]
\\\label{qsconj2}
&\hskip 1cm =\frac{c_{ij}-c_{\tau i,j}}{2}\h\{\H_{i,r+1},\X_{j,s}\}+\frac{c_{ij}+c_{\tau i,j}}{2}\h\{\H_{i,r},\X_{j,s+1}\}+\frac{c_{ij}c_{\tau i,j}}{4}\h^2[\H_{i,r}, \X_{j,s}],
\\\label{qsconj3}
&\quad [\X_{i,r+1 },\X_{j,s}]  - [\X_{i,r },\X_{j,s+1 }] 
 =\frac{c_{ij}}{2}\h\{\X_{i,r },\X_{j,s }\}-2 \delta_{\tau i,j}(-1)^{r}\H_{j,r+s+1 },
\end{align}
and the Serre relations: for $c_{ij}=0,$
\begin{align}\label{qsconj4}
&[\X_{i,r},\X_{j,s}]= \delta_{\tau i,j}(-1)^{r}h_{j,r+s},
\end{align}
and for $c_{ij}=-1,j\neq \tau i\neq i$,
\begin{align}\label{qsconj9}
&\mathrm{Sym}_{k_1,k_2}\big[b_{i,k_1},[b_{i,k_2},b_{j,r}] \big] =0,
\end{align}
and for $c_{ij}=-1,\tau i=i$,
\beq\label{qsconj8}
\begin{split}
\mathrm{Sym}_{k_1,k_2}&\big[b_{i,k_1},[b_{i,k_2},b_{j,r}] \big] \\
=\,&(-1)^{k_1}\sum_{p\gge 0}2^{-2p}\h^{2p}\big([h_{i,k_1+k_2-2p-1},b_{j,r+1}]-\h\{h_{i,k_1+k_2-2p-1},b_{j,r}\}\big),
\end{split}
\eeq
and for $c_{i,\tau i}=-1$,
\begin{align}\label{qsconj10}
&\mathrm{Sym}_{k_1,k_2}\big[b_{i,k_1},[b_{i,k_2},b_{\tau i,r}] \big] =\frac{4}{3}\,\mathrm{Sym}_{k_1,k_2}(-1)^{k_1}\sum_{p=0}^{k_1+r}3^{-p}[b_{i,k_2+p},h_{\tau i,k_1+r-p}],
\end{align}
where $\H_{i,s}=0$ if $s< 0$.
\end{dfn}



Note that it follows from \eqref{qsconj0} and \eqref{qsconj3} that for $r\in\bN$,
\beq\label{hrel}
\H_{\tau i,r}=(-1)^{r+1} \H_{i,r}.
\eeq

\begin{rem}\label{rem:newSerre}
The relation \eqref{qsconj10} only appears for type AIII$_{2n}$. 
\end{rem}

\begin{rem}\label{rem:h=1}
To indicate the dependence on $\h$, we write $\Yi$ as $\mathbf Y^\imath_\h$. The algebras $\mathbf Y^\imath_\h$ with $\h\ne 0$ are all isomorphic to each other, with an explicit isomorphism given by
\[
\mathbf Y^\imath_\h\to \mathbf Y^\imath_1,\qquad h_{i,r}\mapsto \h^r h_{i,r},\qquad b_{i,r}\mapsto \h^r b_{i,r}.
\]
\end{rem}
If $c_{\tau i,j}=0$, then the relation \eqref{qsconj2} has a more familiar equivalent form, which corresponds to the current relations for the ordinary Yangians; see also \S\ref{sec:hb} and \S\ref{secA}. 
\begin{lem}\label{typeArel}
If $c_{\tau i,j}=0$, then the relations \eqref{qsconj5'}--\eqref{qsconj2} are equivalent to 
\beq\label{alt}
[h_{i,0},b_{j,s}]=c_{ij}b_{j,s},\quad [h_{i,r+1},b_{j,s} ]-[h_{i,r},b_{j,s+1} ]= \frac{c_{ij}}{2}\h\{h_{i,r},b_{j,s}\}.
\eeq    
\end{lem}
\begin{proof}
If $c_{\tau i,j}=0$, then the relation \eqref{qsconj5'} is exactly the first relation of \eqref{alt}. Let 
\[
\mathscr Q_{ij}(r,s)=[h_{i,r+1},b_{j,s}]-[h_{i,r},b_{j,s+1}]-\frac{c_{ij}}{2} \h\{h_{i,r},b_{j,s}\}.
\]
If $c_{\tau i, j}=0$, then \eqref{qsconj2} is equivalent to
\[
\mathscr Q_{ij}(r+1,s)+\mathscr Q_{ij}(r,s+1)=0.
\]
Note that $\mathscr Q_{ij}(0,s)=0$ by \eqref{qsconj5'} and \eqref{qsconj5}. Thus the relations \eqref{qsconj5'}--\eqref{qsconj2} imply \eqref{alt}. 

The other direction is obvious.
\end{proof}

\subsection{PBW type bases for $\Yi$}\label{sec:pbw}
For each $\alpha\in\cR^+$, fix $i_1,\dots,i_k\in\I^0$ such that the element $f_{\alpha}:=\big[f_{i_1},[f_{i_2},\cdots[f_{i_{k-1}},f_{i_k}]\cdots]\big]$ is a nonzero root vector in $\g_{-\alpha}$. Define
\beq\label{bjir1}
b_{\alpha,r}:=\Big[b_{i_1,0},\big[b_{i_2,0},\cdots[b_{i_{k-1},0},b_{i_k,r}]\cdots\big]\Big].
\eeq
Extend the action of $\tau$ to $\I$ by setting $\tau (0)=0$. Let 
\beq\label{I-tau}
\I_\tau=\{\text{the chosen representatives of $\tau$-orbits in } \I\},\qquad \I_\tau^0=\I_\tau\setminus \{0\}.
\eeq 
Define 
\beq\label{eq:I+-}
\I_=^0 =\{\text{fixed points of $\tau$ in $\I^0$}\},\qquad  \I_{\ne}^0=\I_\tau^0\setminus \I_=^0.
\eeq

\begin{prop}\label{prop:span}
The ordered monomials of 
\beq\label{eq:pbw-elem}
\{b_{\alpha,r},h_{i,r},h_{j,2r+1}~|~\alpha\in\cR^+,i\in\I^0_{\ne},j\in\I_=^0,r\in\N\}
\eeq
(with respect to any fixed total ordering) form a spanning set of $\Yi$.
\end{prop}

\begin{proof}
The proof is the same as \cite[Prop. 4.5]{KLWZ23b} by defining a filtration on the algebra $\Yi$ with $\deg b_{i,r}=\deg h_{i,r}=r+1$. The choice of the subset of elements of the form $h_{i,r},h_{j,2r+1}$ in \eqref{eq:pbw-elem} is due to the relation \eqref{hrel}.
\end{proof}

\begin{thm}\label{thm:pbw}
The ordered monomials of 
$$
\{b_{\alpha,r},h_{i,r},h_{j,2r+1}~|~\alpha\in\cR^+,i\in\I^0_{\ne},j\in\I_{=}^0,r\in\N\}
$$ 
(with respect to any fixed total ordering) form a basis of $\Yi$. 
\end{thm}

\begin{proof}
By Proposition \ref{prop:span}, these monomials form a spanning set of $\Yi$. If $\Yi$ is not of type AIII$_{2n}$, then they are also linearly independent by the proof of Theorem \ref{thm:main1}; see the proof \cite[Thm. 4.10]{KLWZ23b} for detail. If $\Yi$ is of type AIII$_{2n}$, the statement is a corollary of the proof of Theorem \ref{main2} via Gauss decomposition.
\end{proof}

\begin{rem}
Setting $\h=0$ in the defining relation of $\Yi$, one recovers a presentation of $\rU(\g[z]^{\check\omega})$. Moreover, it follows by comparing the PBW basis of $\Yi$ in Theorem \ref{thm:pbw} and the corresponding PBW basis of  $\rU(\g[z]^{\check\omega})$ that $\Yi$ is a flat deformation of $\rU(\g[z]^{\check\omega})$.
\end{rem}

We define a filtration on $\Yi$ by setting $\deg \, \H_{i,r}=\deg \, \X_{i,r}=r$, and denote by ${\gr}\,\Yi$ the associated graded algebra of $\Yi$. Let $\bar{\H}_{i,r}, \bar{\X}_{i,r}$ denote the images of $\H_{i,r},\X_{i,r}$ in  ${\mathrm{gr}}\,\Yi$, respectively. Recall the presentation of $\rU\big(\g[z]^{\check\omega}\big)$ from Proposition \ref{prop:qscrel}.

\begin{cor}
\label{cor:PBW2}
 Let $\Yi$ be a twisted Yangian of quasi-split type. There is an algebra isomorphism  
\[
\iota: \rU\big(\g[z]^{\check\omega}\big) \stackrel{\cong}{\longrightarrow} {\mathrm{gr}}\,\Yi
\]
which sends $\xi_{i,r}\mapsto \bar{\H}_{i,r}, x_{i,r}\mapsto  \bar{\X}_{i,r}$. 
\end{cor}
\begin{proof}
The proof is similar to that of \cite[Coro. 4.13]{KLWZ23b} by using Proposition \ref{prop:qscrel} and Theorem \ref{thm:pbw}.
\end{proof}

We also formulate an alternative presentation of the twisted Yangian $\Yi$ which is simpler than the one in Definition \ref{deftY}. The new presentation is easier to verify as it avoids the full set of complicated Serre relations.

\begin{thm}\label{thm:redu}
The twisted Yangian $\Yi$ has the following presentation as a $\C[\h]$-algebra with generators $h_{i,r},b_{i,r}$, $i\in\I^0,r\in\N$, subject to the relations \eqref{qsconj0}--\eqref{qsconj3} and the finite type Serre relations \eqref{finserre1-ty}--\eqref{finserre4-ty}: for $c_{i,\tau i}=-1$,
\beq\label{finserre1-ty}
\big[b_{i,0},[b_{i,0},b_{\tau i,0}]\big]  
= 4b_{i,0},
\eeq
and for $j\neq i,\tau i,$
\begin{align}
[b_{i,0},b_{j,0}]&=0,\hskip 2.2cm c_{ij}=0,\label{finserre2-ty}
\\
\big[b_{i,0},[b_{i,0},b_{j,0}]\big]& =0,\hskip 2.2cm c_{ij}=-1 \text{ and }i\ne \tau i, \label{finserre3-ty}\\
\big[b_{i,0},[b_{i,0},b_{j,0}]\big]& =-b_{j,0},\hskip 1.6cm c_{ij}=-1 \text{ and }i=\tau i.\label{finserre4-ty}
\end{align}
\end{thm}
\begin{proof}
The proof is completely analogous to that of \cite[Thm. 4.14]{KLWZ23b} with the help of Proposition \ref{prop:redu} and Corollary \ref{cor:PBW2}.
\end{proof}

\subsection{Presentation via generating functions}
\label{sec:geniQG}
Define
\[
b_{i}(u):=\h \sum_{r\gge 0}b_{i,r}u^{-r-1},\qquad h_{i}(u):=1+\h \sum_{r\gge 0}h_{i,r}u^{-r-1},\qquad h_i^{\circ}(u):=h_i(u)-1.
\]
Defining relations \eqref{qsconj0}--\eqref{qsconj8} (without \eqref{qsconj10}) for $\Yi$ can be reformulated in generating function form as below:
\beq
[h_{i}(u),h_j(v)]=0,\hskip2cm  h_{\tau i}(u)=h_i(-u),\qquad \hskip1cm
\eeq
\beq
\begin{split}
(u-v)[b_i(u),b_j(v)]&=\frac{c_{ij}}{2}\h\{b_i(u),b_j(v)\}
+\h\big([b_{i,0},b_j(v)]-[b_i(u),b_{j,0}]\big)
\\&\quad  -\delta_{\tau i,j}  \h \Big(\frac{2u}{u+v} h_i^\circ(u)+\frac{2v}{u+v}h_j^\circ(v)\Big),  \qquad\qquad (c_{i,\tau i}\ne 0),
\end{split}
\eeq

\beq
\begin{split}
(u^2-v^2)[h_i(u),b_j(v)]=&\,\frac{c_{ij}-c_{\tau i,j}}{2}\h u\{h_i(u),b_j(v)\}+\frac{c_{ij}+c_{\tau i,j}}{2}\h v\{h_i(u),b_j(v)\}\\
&+\frac{c_{ij}c_{\tau i,j}}{4}\h^2[h_i(u),b_j(v)]-\h[h_i(u),b_{j,1}]\\
&-\h v[h_i(u),b_{j,0}]-\frac{c_{ij}+c_{\tau i,j}}{2}\h^2\{h_i(u),b_{j,0}\},
\end{split}
\eeq

\beq
(u+v)[b_{i}(u),b_{j}(v)]=\delta_{\tau i,j}(h_{j}(v)-h_i(u)), \hskip 3cm (c_{ij}=0),
\eeq

\beq\label{gconj5}
\Sym_{u_1,u_2}\big[b_{i}(u_1),[b_{i}(u_2), b_{j}(v) ]\big]=0, \hskip 2.5cm (c_{ij}=-1,j\ne\tau i\ne i),
\eeq

\beq\label{gconj6} 
\begin{split}
&\Sym_{u_1,u_2}\big[b_{i}(u_1),[b_{i}(u_2), b_{j}(v) ]\big]
\\
&=\frac{4\h }{u_1 + u_2}\Sym_{u_1,u_2}  \frac{u_2(v-\h )\, h_i(u_2) b_j(v)- u_2(v+\h  )\, b_j(v)h_i(u_2) }{4u_2^2- \h^2}
\quad (c_{ij}=-1,\tau i=i).
\end{split}
\eeq

\subsection{Deducing Serre relation \eqref{qsconj10}}
\label{sec:qsSerre}

In this subsection, we prove that assume relations \eqref{qsconj0}--\eqref{qsconj3}, then the Serre relation \eqref{qsconj10} for arbitrary $k_1,k_2,r$ can be deduced from \eqref{qsconj10} for the special case $k_1=k_2=r=0$, following the approach of \cite{Le93}.

Define a new set of imaginary root vectors $\zeta_{i,r}$ by
\beq\label{newrv}
\zeta(u):=\h\sum_{r\gge 0} \zeta_{i,r}u^{-r-1}=\ln h_i(u).
\eeq
For each $i\in \I^0$ and $r\in\bN$, $\zeta_{i,r}$ is a polynomial in $\{h_{i,s}\}_{0\lle s\lle r}$.
\begin{lem}[{cf. \cite{Le93,GTL13}}]\label{rvnew}
We have
\beq\label{htl0}
\begin{split}
[\zeta_{i,r},b_{j,s}]=\sum_{\ell=0}^{\lfloor \frac{r}{2}\rfloor} {r \choose 2\ell} \frac{2^{-2\ell}}{2\ell+1}\Big(c_{ij}  \big(c_{ij}\h\big)^{2\ell} +(-1)^{r+1}c_{\tau i,j}  \big(c_{\tau i,j}\h\big)^{2\ell}\Big)b_{j,r+s-2\ell}.
\end{split}
\eeq
for  $i,j\in\I^0$ and $r,s\in\bN$.
\end{lem}
\begin{proof}
Let $\mathscr L$ be the $\C[\h]$-algebra generated by $h_{i,r},b_{i,r}$ for $i\in \I^0,r\in\bN$ subject to the relations \eqref{qsconj0}--\eqref{qsconj2}. Then for each $j$, there exists an algebra homomorphism $\sigma_j$ for $\mathscr L$ defined by the assignment
\[
h_{i,r}\mapsto h_{i,r},\qquad b_{i,r}\mapsto b_{i,r+\delta_{ij}}.
\]
Then, the relations \eqref{qsconj5'}--\eqref{qsconj2} imply that
\beq\label{ppf6}
h_i(u)b_{j,s}=\frac{(u-\sigma_j+\frac{c_{ij}\h}{2})(u+\sigma_j-\frac{c_{\tau i,j}\h}{2})}{(u-\sigma_j-\frac{c_{ij}\h}{2})(u+\sigma_j+\frac{c_{\tau i,j}\h}{2})}b_{j,s}h_i(u).
\eeq
Define $\zeta_i(u)$ and $\zeta_{i,r}\in\mathscr L$ again by \eqref{newrv}. It follows from \eqref{ppf6} that
\[
[\zeta_i(u),b_{j,s}]=\ln\bigg(\frac{(1-(\sigma_j-\frac{c_{ij}\h}{2})u^{-1})(1+(\sigma_j-\frac{c_{\tau i,j}\h}{2})u^{-1})}{(1-(\sigma_j+\frac{c_{ij}\h}{2})u^{-1})(1+(\sigma_j+\frac{c_{\tau i,j}\h}{2})u^{-1})}\bigg)b_{j,s}.
\]
Then, \eqref{htl0} follows from above by expanding the RHS as a power series in $u^{-1}$.
\end{proof}

Now we fix $i$ and suppose $c_{i,\tau i}=-1$ (which also implies that $i\ne \tau i$). It follows from Lemma \ref{rvnew} that there exist $\tl h_{i,r}$ as polynomials in $\{h_{i,p}\}_{p\in\bN}$ such that
\beq\label{htl}
[\tl h_{i,r},b_{i,s}]=\big(2+(-1)^r\big)b_{i,r+s},\quad [\tl h_{i,r},b_{\tau i,s}]=-\big(1+2(-1)^r\big)b_{\tau i,r+s},
\eeq
for all $s\in\bN$.

Set
\[
(k_1,k_2\,|\,r):=\Sym_{k_1,k_2}\big[b_{i,k_1},[b_{i,k_2},b_{\tau i,r}]\big].
\]
We have the following lemma.
\begin{lem} For $k_1,k_2,r\in\bN$, we have
\beq\label{ppf1}
\begin{split}
(k_1+1,k_2\,|\,r)+(k_1,k_2+1\,|\,r)&-2(k_1,k_2\,|\,r+1)\\
=&-4\mathrm{Sym}_{k_1,k_2}(-1)^{k_1}[b_{i,k_2},h_{\tau i,k_1+r+1}].
\end{split}
\eeq
\end{lem}
\begin{proof}
By Jacobi identity, we have
\begin{align*}
&(k_1+1,k_2\,|\,r)+(k_1,k_2+1\,|\,r)\\
=\hskip0.2cm & \big[[b_{i,k_1+1},b_{i,k_2}],b_{\tau i,r}\big]+\big[[b_{i,k_2+1},b_{i,k_1}],b_{\tau i,r}\big]\\
&+2\big[b_{i,k_2},[b_{i,k_1+1},b_{\tau i,r}]\big]+2\big[b_{i,k_1},[b_{i,k_2+1},b_{\tau i,r}]\big]\\
\stackrel{\eqref{qsconj3}}{=}\,  & \h\big[\{b_{i,k_1},b_{i,k_2}\},b_{\tau i,r}\big]\\
&+2\big[b_{i,k_2},[b_{i,k_1},b_{\tau i,r+1}]\big]-\h\big[b_{i,k_2},\{b_{i,k_1},b_{\tau i,r}\}\big]-4(-1)^{k_1}[b_{i,k_2},h_{\tau i,k_1+r+1}]\\&+2\big[b_{i,k_1},[b_{i,k_2},b_{\tau i,r+1}]\big]-\h\big[b_{i,k_1},\{b_{i,k_2},b_{\tau i,r}\}\big]-4(-1)^{k_2}[b_{i,k_1},h_{\tau i,k_2+r+1}]\\
=\hskip0.2cm & 2(k_1,k_2\,|\,r+1)-4(-1)^{k_1}[b_{i,k_2},h_{\tau i,k_1+r+1}]-4(-1)^{k_2}[b_{i,k_1},h_{\tau i,k_2+r+1}].\qedhere
\end{align*}
\end{proof}

\begin{prop}\label{prop-serre}
Suppose the relations \eqref{qsconj0}--\eqref{qsconj3} hold. Assuming further that $c_{i,\tau i}=-1$ and 
\beq\label{ppf0}
\big[b_{i,0},[b_{i,0},b_{\tau i,0}]\big]=4b_{i,0},
\eeq
then we have
\beq\label{todo-}
(k_1,k_2\,|\,r)=\frac{4}{3}\,\mathrm{Sym}_{k_1,k_2}(-1)^{k_1}\sum_{p=0}^{k_1+r}3^{-p}[b_{i,k_2+p},h_{\tau i,k_1+r-p}].
\eeq
\end{prop}
\begin{proof}
Note that the proofs of \eqref{htl} and \eqref{ppf1} only use the relations \eqref{qsconj0}--\eqref{qsconj3}. Therefore, under our assumption, \eqref{htl} and \eqref{ppf1} hold. We shall work inductively on the number of nonzero numbers in $k_1,k_2,r$.

We first show \eqref{todo-} for the case $(0,0\,|\,r)$ by induction on $r$. By \eqref{qsconj5'}, we have $[h_{\tau i,0},b_{i,s}]=-3b_{i,s}$. Thus the base case follows from our assumption. Suppose \eqref{todo-} holds for $(0,0\,|\,r)$. Then bracketing it with $\tl h_{i,1}$, we find that
\beq\label{ppf2}
2(1,0\,|\,r)+(0,0\,|\,r+1)=\frac{8}{3}\,\sum_{p=0}^{r}3^{-p}[b_{i,1+p},h_{\tau i,r-p}].
\eeq
On the other hand, by \eqref{ppf1}, we have
\beq\label{ppf3}
2(1,0\,|\,r)-2(0,0\,|\,r+1)=-8[b_{i,0},h_{\tau i,r+1}].
\eeq
Solving for $(0,0\,|\,r+1)$ from \eqref{ppf2} and \eqref{ppf3}, we conclude that \eqref{todo-} also holds for $(0,0\,|\,r+1)$. Therefore, we have proved \eqref{todo-} for $(0,0\,|\,r)$ with any $r\in\bN$.

Bracketing \eqref{ppf0} with $\tl h_{i,r}$, it follows from \eqref{htl} that
\[
2(r,0\,|\,0)-(-1)^r(0,0\,|\,r)=8b_{i,r}.
\]
Using \eqref{todo-} with $(0,0\,|\,r)$ and $[h_{\tau i,0},b_{i,s}]=-3b_{i,s}$, we also obtain \eqref{todo-} with $(r,0\,|\,0)$.

Bracketing $(0,0\,|\,r)$ with $\tl h_{i,k_1}$, we have
\[
2(k_1,0\,|\,r)-(-1)^r(0,0\,|\,k_1+r)=\frac{8}{3}\sum_{p=0}^{r}3^{-p}[b_{i,k_1+p},h_{\tau i,r-p}].
\]
By \eqref{todo-} with $(0,0\,|\,k_1+r)$, we obtain \eqref{todo-} with $(k_1,0\,|\,r)$.

Bracketing $(k_1,0\,|\,0)$ with $\tl h_{i,k_2}$, we get
\begin{align*}
(k_1+k_2,0\,|\,0)+(k_1,k_2\,|\,0)&-(-1)^{k_2}(k_1,0\,|\,k_2)\\
&=\frac{4}{3}\,[b_{i,k_1+k_2},h_{\tau i,0}]+\frac{4}{3}\,(-1)^{k_1}\sum_{p=0}^{k_1}3^{-p}[b_{i,k_2+p},h_{\tau i,k_1-p}].
\end{align*}
Since we already obtained \eqref{todo-} with $(k_1+k_2,0\,|\,0)$ and $(k_1,0\,|\,k_2)$, we easily deduce \eqref{todo-} with $(k_1,k_2\,|\,0)$ from the above equation.

Finally, bracketing $(k_1,k_2\,|\,0)$ with $\tl h_{i,r}$, we have
\begin{align*}
(k_1+r,k_2\,|\,0)&+(k_1,k_2+r\,|\,0)-(-1)^{r}(k_1,k_2\,|\,r)\\
&=\frac{4}{3}(-1)^{k_1}\sum_{p=0}^{k_1}3^{-p}[b_{i,k_2+r+p},h_{\tau i,k_1-p}]+\frac{4}{3}\,(-1)^{k_2}\sum_{p=0}^{k_2}3^{-p}[b_{i,k_1+r+p},h_{\tau i,k_2-p}].
\end{align*}
Repeating the same procedure, we verify \eqref{todo-} with $(k_1,k_2\,|\,r)$.
\end{proof}

\section{A degeneration approach}
\label{sec:degen}
In this section, we first recall from \cite{LWZ23b} the Drinfeld type presentation for the quasi-split affine $\imath$quantum group $\tUi$ of type AIII$_{2n-1}$, DII$_n$, and EII$_{6}$ (which were denoted by AIII$_{2n-1}^{(\tau)}$, DI$_{n}^{(\tau)}$,  EI$_{6}^{(\tau)}$ in \cite{LWZ23b}). We establish in Theorem~\ref{thm:main1} that the twisted Yangians $\Yi$ can be realized via a degeneration of affine $\imath$quantum groups.


\subsection{Quasi-split affine $\imath$quantum groups}
Let $\bv$ be an indeterminate and define for $m\in\N$
\[
[m]=\frac{\bv^m-\bv^{-m}}{\bv-\bv^{-1}},\qquad [m]!=[1][2]\cdots[m],\qquad \qbinom{m}{s}=\frac{[m]!}{[s]![m-s]!}.
\]
We write $[x,y]=xy-yx$ and $[x,y]_a=xy-ayx$.

Let $\U:=\rU_\bv(\wt{\g})$ be the Drinfeld-Jimbo quantum group, i.e., $\U$ is the $\Q(\bv)$-algebra generated by $E_i,F_i,K_i^{\pm 1}$, $i\in \I$, subject to the relations:
\begin{align*}
[E_i,F_j]=\delta_{ij}\frac{K_i-K_i^{-1}}{\bv-\bv^{-1}},
&\qquad [K_i,K_j]=0,
\\
K_i E_j=\bv^{c_{ij}}E_jK_i,&\qquad K_i F_j=\bv^{-c_{ij}}F_jK_i,
\end{align*}
and the quantum Serre relations which we skip. 

Recall $\I_\tau$ from \eqref{I-tau}. Let $\tUi$ be the $\Q(\bm v)$-subalgebra of $\U$ generated by $B_i,\K_j:=K_jK_{\tau j}^{-1}$ for $i\in \I, j \in \I_{\tau}$, where 
\[
B_i:=\begin{cases}
F_i-\bv^{-2}E_{i} K_i^{-1},& \tau i=i,
\\
F_i+E_{\tau i} K_i^{-1}, & \tau i\neq i.
\end{cases}
\]

We recall the Drinfeld presentation for $\tUi$ from \cite{LWZ23b}. Denote, for $c_{ij}=-1$,
	\begin{align*}
			\bS_{i,j}(k_1,k_2|l) : = \Sym_{k_1,k_2}\Big(B_{i,k_1}B_{i,k_2}B_{j,l}-[2]B_{i,k_1} B_{j,l} B_{i,k_2} + B_{j,l} B_{i,k_1}B_{i,k_2}\Big).
	\end{align*}

\begin{prop}[\text{cf. \cite[Thm. 4.9]{LWZ23b}}]
\label{thm:qsiDR}
The affine $\imath$quantum group $\tUi$ is isomorphic to the $\Q(\bv)$-algebra generated by the elements $B_{i,l}$, $H_{i,m}$, $\K_j$ where $i,j\in \I^0$, $\tau j\neq j$, $l\in\Z$ and $m \in \Z_{\gge 1}$, subject to the following relations: for $m, n \gge 1, k, l \in \Z$, and $i, j \in \I^0$,
\begin{align}
\K_i\K_{\tau i}=1, \quad [\K_i,\K_j]&  =  
[\K_i,H_{j,n}]=0,\quad \K_i B_{j,l}=\bv^{c_{\tau i,j}-c_{ij}} B_{j,l} \K_i,
\label{qsiDR0}
\\
\label{qsiDR1}
 [H_{i,m},H_{j,n}] &=0,
\\
\label{qsiDR2}
[H_{i,m},B_{j,l}] &=\frac{[mc_{ij}]}{m} B_{j,l+m}-\frac{[mc_{\tau i,j}]}{m} B_{j,l-m},
\\
\label{qsiDR4}
[B_{i,k},B_{\tau i,l}]
&= \K_{\tau i}  \TH_{i,k-l}- \K_{i} \Theta_{\tau i,l-k},
 \quad \text{ if }c_{i,\tau i}=0,
\\
\notag [B_{i,k},B_{i,l+1}]_{\bv^{-2}}-\bv^{-2}[B_{i,k+1},B_{i,l}]_{\bv^{2}}
=\,&\bv^{-2} \Theta_{i,l-k+1}-  \bv^{-2} \Theta_{i,l-k-1}
\\
+\,&\bv^{-2} \Theta_{i,k-l+1}  -\bv^{-2} \Theta_{i,k-l-1},  \qquad\text{ if } \tau i =i,
\label{qsiDR5}
\\ 
\label{qsiDR3}
[B_{i,k},B_{j,l+1}]_{\bv^{-c_{ij}}} -&\bv^{-c_{ij}}[B_{i,k+1},B_{j,l}]_{\bv^{c_{ij}}} =0, \qquad\hskip0.32cm \text{ if } j\neq \tau i,
\end{align}
and the Serre relations
\begin{align}
	&\bS_{i,j}(k_1,k_2|l)
	= 0, \qquad\qquad\qquad \text{ if }c_{ij}=-1, j \neq \tau i, \label{qsiDR9}
 \\
 &[B_{i,k},B_{j,l}]
	= 0, \qquad\qquad\qquad \text{ if }c_{ij}=0, j \neq \tau i\neq i, \label{qsiDR9'}
  \\\notag
	\bS_{i,j}(k_1,k_2|l)  
		 =&-\sum_{p=0}^{\bigl\lfloor\frac{k_2-k_1-1}{2}\bigr\rfloor } (\bv^{2p+1}+\bv^{-2p-1}) [\Theta_{i,k_2-k_1-2p-1},B_{j,l-1}]_{\bv^{-2}} \\
&-\sum_{p=1}^{\bigl\lfloor\frac{k_2-k_1 }{2}\bigr\rfloor} (\bv^{2p}+\bv^{-2p})[B_{j,l},\Theta_{i,k_2-k_1-2p}]_{\bv^{-2}}-[B_{j,l},\Theta_{i,k_2-k_1}]_{\bv^{-2}}
			   \label{qsiDR8}
\\\notag
&  +\{k_1 \leftrightarrow k_2\},\qquad\qquad \text{ if }c_{ij}=-1 \text{ and }\tau i=i.
\end{align}
Here
\beq\label{exp}
1+(\bv-\bv^{-1})\sum_{k\gge 1}\Theta_{i,k}u^k=\exp\big((\bv-\bv^{-1})\sum_{k\gge 1}H_{i,k}u^k \big),
\eeq
$\Theta_{i,0}=(\bv-\bv^{-1})^{-1},\Theta_{i,m}=0$ if $m<0$, and $\{k_1\leftrightarrow k_2\}$ denotes a summand obtained from the prior one with $k_1, k_2$ switched.
\end{prop}

\begin{rem}
We are using quasi-split analogs of Baseilhac-Kolb's root vectors \cite{BK20} in Proposition~\ref{thm:qsiDR}; see \cite[\S 4.4]{LWZ23b}. The presentation {\em loc. cit.} is formulated for the universal $\imath$quantum group, which involves certain central elements $\K_i, \K_j\K_{\tau j}, C, i,j\in \I^0,\tau i=i,\tau j\neq j$. The affine $\imath$quantum group $\tUi$ in Proposition~\ref{thm:qsiDR} can be recovered from the universal $\imath$quantum group by setting all these central elements to be $1$. 
\end{rem}

It would be convenient to rewrite certain relations of $\tUi$ in Proposition~\ref{thm:qsiDR} for later use.

\begin{lem}We have the following.
\begin{enumerate}
    \item The relation \eqref{qsiDR2} is equivalent to the following relation \emph{(}cf. \cite[Lem. 5.4]{LWZ23a}\emph{)}
\begin{align}
\label{qsiA1DR3'}
[\Theta_{i,m},B_{j,k}]&+\bv^{c_{ij}-c_{\tau i,j}}[\Theta_{i,m-2},B_{j,k}]_{\bv^{2(c_{\tau i,j}-c_{ij})}}
\\ \notag
&-\bv^{c_{ij}}[\Theta_{i,m-1},B_{j,k+1}]_{\bv^{-2c_{ij}}}- \bv^{-c_{\tau i, j}}[\Theta_{i,m-1},B_{j,k-1}]_{\bv^{2c_{\tau i,j}}}
=0.
\end{align}
\item The relation \eqref{qsiDR4} for $k>l$ admits the following equivalent form
\begin{align}\label{qsiDR10}
[B_{i,l+m},B_{\tau i,l}]
			&= \K_{\tau i}  \TH_{i,m},\qquad l\in\Z,m>0. 
\end{align}
\end{enumerate}
\end{lem}

\subsection{Classical limit}
The classical limit at $\bv\mapsto 1$ of $\tUi$ is identified with $\rU(\wt{\g}^\omega)$ via the correspondence between generators:
\beq
\TH_{i,m} \rightsquigarrow \theta_{i,m}, \qquad B_{i,k} \rightsquigarrow \flat_{i,k},\qquad i\in\I^0,m\gge 1, k\in\Z,
\eeq
where $\theta_{i,m}$ and $\flat_{i,k}$ are defined in \eqref{hggen}; cf. \cite{Ko14}.

For each $\alpha\in\cR^+$, fix $i_1,\dots,i_a\in\I^0$ such that $f_{\alpha}:=\big[f_{i_1},[f_{i_2},\cdots[f_{i_{a-1}},f_{i_a}]\cdots]\big]$ is a nonzero root vector in $\g_{-\alpha}$. For $\alpha\in\mc R^+$ and $k\in\Z$, define
\beq\label{Bjir1}
B_{\alpha,k}:=\Big[B_{i_1,0},\big[B_{i_2,0},\cdots[B_{i_{a-1},0},B_{i_a,k}]\cdots\big]\Big].
\eeq

Let $\A$ be the localization of $\C[\bv,\bv^{-1}]$ with respect to the ideal $(\bv-1)$ and $\tUi_{\A}$ the $\A$-subalgebra of $\tUi$ generated by $B_{i,k},\TH_{i,m}$, for $i\in \I^0,k\in \Z, m\gge 1$. Let $\U_\A$ be the $\A$-subalgebra of $\U$ generated by $E_i,F_i,K_i^{\pm 1}, \frac{K_i-1}{\bv-1}$ for $i\in \I$. It is known (cf. \cite[Coro. 10.3]{Ko14}) that $\tUi_\A\subset \U_\A$. The following proposition is an analog of \cite[Prop. 2.1]{GM12}. 

\begin{prop}[cf. \cite{Ko14}]
\label{prop:GM}
We have an algebra isomorphism $$\tUi_\A / (\bv-1) \tUi_\A \stackrel{\cong~}{\longrightarrow} \rU(\wt{\g}^\omega).$$ Moreover, $\tUi_\A$ is a free $\A$-module isomorphic to $\rU(\wt{\g}^\omega) \otimes_{\C} \A$. 
\end{prop}
\begin{proof}
The first statement was established in \cite[Thm. 10.8]{Ko14}. The second statement follows by similar arguments as \cite[Prop.~3.3]{KLWZ23b}.
\end{proof}


\subsection{From affine $\imath$quantum groups to twisted Yangians}
\label{sec:filtration}

We define a filtration on $\tUi_\A $ analogous to the one in \cite[\S2]{CG15} on quantum loop algebras. 
Let $\psi$ denote the following composition of algebra homomorphisms (see Proposition \ref{prop:GM}),
\begin{align*}
\Psi:\tUi_{\A}\twoheadrightarrow \tUi_{\A}/(\bv-1)\tUi_\A
\stackrel{\cong~}{\longrightarrow} \mathrm{U}(\wt{\g}^{\omega}).
\end{align*}
For $r\in\bN$, denote by $\mathrm{K}_r$ the Lie ideal of $\wt{\g}^{\omega}$ spanned by
\[
g t^s(t-1)^r+\omega(g) t^{-s-r}(1-t)^r,\qquad s\in \Z,g\in \g.
\]
Denote
\beq\label{W}
\mathcal{W}:=\C\text{-span}
\{B_{\alpha,k},\TH_{i,l}~|~\alpha\in \cR^+, i\in \I^0,k\in \Z,l> 0\}\subset \U_\A^\imath.
\eeq
Set $\mathcal{K}_r$ to be the two-sided ideals of $\tUi_{\A}$ generated by $\Psi^{-1}(\mathrm{K}_r)\cap \mc W$. Define $\mathbf{K}_r$ to be the $\C$-sums of ideals $(\bv-\bv^{-1})^{r_0} \mathcal{K}_{r_1}\cdots \mathcal{K}_{r_a}$ such that $r_0+r_1+\cdots +r_a\gge r$.

Define $\Gr_{\mathbf K} \tUi$ to be the $\C$-algebra 
\begin{align}
\Gr_{\mathbf K} \tUi:=\bigoplus_{r\gge 0} \mathbf{K}_r/\mathbf{K}_{r+1}.
\end{align}
Then $\Gr_{\mathbf K} \tUi$ can be viewed as a $\C[\h]$-algebra by setting $\h:=\ov{\bv-\bv^{-1}}\in \bK_1/\bK_2$.

For $l\gge 1, k\in\Z$, define 
\beq\label{def:qsHX}
\H_{i,r,l}=\sum_{s=0}^r (-1)^{r-s}\binom{r}{s}\Theta_{i,s+l},\qquad \X_{i,r,k}=\sum_{s=0}^r (-1)^{r-s}\binom{r}{s}B_{i,s+k}.
\eeq
Write $\ov{\X}_{i,r,k},\ov{\H}_{i,r,k}$ for the images of  $\X_{i,r,k},\H_{i,r,k}$ in $\bK_r/\bK_{r+1}$.

\begin{lem}\label{qslemma1}
We have
\begin{enumerate}
\item $\X_{i,r+1,k}=\X_{i,r,k+1}-\X_{i,r,k}$ and $ \H_{i,r+1,k}=\H_{i,r,k+1}-\H_{i,r,k}$;
\item $\X_{i,r,k},\H_{i,r,k}\in \bK_r$ and $\H_{i,r,k}+(-1)^r\H_{\tau i,r,k}\in \bK_{r+1}$.
\end{enumerate}
In particular, $\bar\X_{i,r,k+1}=\bar\X_{i,r,k}$ and $\bar\H_{i,r,k}=(-1)^{r+1}\bar \H_{\tau i,r,k}$.
\end{lem}

\begin{proof}
Part (1) follows from the definition~\eqref{def:qsHX}. It is also clear from \eqref{def:qsHX} that $\X_{i,r,k},\H_{i,r,k}\in \mathcal W$ and the $\Psi$-images of $\X_{i,r,k},\H_{i,r,k}$ are $\flat_{i,r,k},\theta_{i,r,k}$, respectively. By \eqref{theta-exp}, we have $\flat_{i,r,k},\theta_{i,r,k}\in \mathrm{K}_r$ and $\theta_{i,r,k}+(-1)^r\theta_{\tau i,r,k}\in \mathrm K_{r+1}$. Then the desired statement follows from the definition of $\bK_r$.
\end{proof}


\begin{thm}\label{thm:main1} 
Let $\Yi$ be the algebra defined in Definition~\ref{deftY} of type $\mathrm{AIII}_{2n-1}$, $\mathrm{DII}_n$, $\mathrm{EII}_6$. Then there is an algebra isomorphism
\beq\label{isophi}
\begin{split}
\Phi:&\,\Yi \rightarrow \Gr_{\bK}\tUi, \\
&\, \X_{i,r} \mapsto \ov{\X}_{i,r,1},\quad
\H_{i,s} \mapsto \ov{\H}_{i,s,1},
\end{split}
\eeq
for $i\in \I^0,r,s\in \N$.
\end{thm}

\begin{proof}
In the subsequent Section \ref{sec:qsdegen}, we shall show that the defining relations \eqref{qsconj0}--\eqref{qsconj8} for $\Yi$ are satisfied by $\ov{\X}_{i,r,1}$, $\ov{\H}_{i,s,1}$ in $\Gr_{\bK}\tUi$. Note that the relation \eqref{qsconj10} does not show up for these types, since we did not include type AIII$_{2n}$ (see also Remark~\ref{rem:newSerre}). Therefore $\Phi$ is an algebra homomorphism. Recall that the ordered monomials in Proposition~\ref{prop:span} form a spanning set of $\Yi$. Using similar arguments as in \cite[Thm. 4.10]{KLWZ23b}, one can show that the $\Phi$-images of these monomials form a basis for $\Gr_{\bK}\tUi$. This implies that $\Phi$ is an algebra isomorphism. 
\end{proof}

\subsection{Verifications of the relations}\label{sec:qsdegen}
In this subsection, we verify that the degenerate current relations \eqref{qsconj0}-\eqref{qsconj8} are satisfied in $\Gr_{\mathbf{K}}\tUi$. We will often write $a=b+\mc O(\h^r)$, which means that $a-b$ is an element in $\bK_{r}$. 

\subsubsection{Relations \eqref{qsconj5'}-\eqref{qsconj5}}
 \begin{prop}
 We have the following relations in $\Gr_{\mathbf{K}}\tUi$, 
 \begin{align}
 \label{qssimple0}
[\ov{\H}_{i,0,1},\ov{\X}_{j,r,1}]&=(c_{ij}-c_{\tau i,j})\ov{\X}_{j,r,1},
\\\label{qssimple0'}
[\ov{\H}_{i,1,1},\ov{\X}_{j,r,1}]&=(c_{ij}+c_{\tau i,j}) \ov{\X}_{j,r+1,1}+\frac{\h(c_{ij}-c_{\tau i,j})}{2}\{\ov{\H}_{i,0,1},\ov{\X}_{j,r,1}\}.
\end{align}
 \end{prop}

 \begin{proof}
 Note that, by definition, $\Theta_{i,1}=H_{i,1}$. Setting $m=1$ in \eqref{qsiDR2}, we have
\begin{equation}\label{qssimlpe1}
[\Theta_{i,1},B_{j,l}]=[c_{ij}] B_{j,l+1}-[c_{\tau i,j}] B_{j,l-1}.
\end{equation}
Taking the summation with respect to $l$ in \eqref{qssimlpe1}, we obtain
\begin{equation}\label{qssimlpe2}
[\H_{i,0,1},\X_{j,r,l}]=[c_{ij}] \X_{j,r,l+1}-[c_{\tau i,j}]\X_{j,r,l-1}  .
\end{equation}
Hence, by Lemma~\ref{qslemma1}, we have verified the relation \eqref{qssimple0}.

Setting  $m=2$ in \eqref{qsiDR2}, we have
\begin{equation} \label{qssimlpe3}
[H_{i,2},B_{j,l}]=\frac{[2c_{ij}]}{2} B_{j,l+2}-\frac{[2c_{\tau i,j}]}{2} B_{j,l-2}.
\end{equation}
By definition, we have $\Theta_{i,2}=H_{i,2}+\frac{\bv-\bv^{-1}}{2!}\Theta_{i,1}^2$ and recall that $\H_{i,1,1}=\Theta_{i,2}-\Theta_{i,1}$. Using \eqref{qssimlpe2} and \eqref{qssimlpe3}, we obtain
\begin{align*}\notag
& [\H_{i,1,1},\X_{j,r,l}]=[\Theta_{i,2}-\Theta_{i,1},\X_{j,r,l}]
\\\notag
=&\, \frac{[2c_{ij}]}{2} \X_{j,r,l+2}-\frac{[2c_{\tau i,j}]}{2}\X_{j,r,l-2}+\frac{\bv-\bv^{-1}}{2}\big\{\Theta_{i,1},[\Theta_{i,1},\X_{j,r,l}]\big\}-[\Theta_{i,1},\X_{j,r,l}]
\\
=&\, c_{ij}\X_{j,r+1,l+1}+c_{\tau i,j} \X_{j,r+1,l-2}+\frac{\h(c_{ij}-c_{\tau i,j})}{2}\{\H_{i,0,1},\X_{j,r,l}\}+\mc O(\h^{r+2}).
\end{align*}
Then, by Lemma~\ref{qslemma1}, the desired relation \eqref{qssimple0'} follows.
\end{proof}

\subsubsection{Relation \eqref{qsconj2}} We verify  \eqref{qsconj2} in this subsection.

\begin{prop}
We have the following relation in $\Gr_{\mathbf{K}}\tUi$
\begin{align*}
&\quad [\ov{\H}_{i,s+2,1},\ov{\X}_{j,r,1}] - [\ov{\H}_{i,s,1},\ov{\X}_{j,r+2,1}]
\\
&=\frac{c_{ij}-c_{\tau i,j}}{2}\h\{\ov{\H}_{i,s+1,1},\ov{\X}_{j,r,1}\}+\frac{c_{ij}+c_{\tau i,j}}{2}\h\{\ov{\H}_{i,s,1},\ov{\X}_{j,r+1,1}\}+\frac{c_{ij}c_{\tau i,j}}{4}\h^2[\ov{\H}_{i,s,1}, \ov{\X}_{j,r,1}].
\end{align*}
\end{prop}

\begin{proof}
Recall that \eqref{qsiDR3} has an equivalent form \eqref{qsiA1DR3'}. Taking a summation with respect to $m,k$ in \eqref{qsiA1DR3'}, we have
\begin{align*}
			[\H_{i,s,m},\X_{j,r,k}]&+\bv^{c_{i,j}-c_{\tau i,j}}[ \H_{i,s,m-2},\X_{j,r,k}]_{\bv^{2(c_{\tau i,j}-c_{i,j})}}
			\\\notag
			&-\bv^{c_{i,j}}[\H_{i,s,m-1},\X_{j,r,k+1}]_{\bv^{-2c_{ i,j}}}- \bv^{-c_{\tau i, j}}[\H_{i,s,m-1},\X_{j,r,k-1}]_{\bv^{2c_{\tau i,j}}}
			=0,
		\end{align*}
We rewrite the above identity as
\begin{align*}
&[\H_{i,s+1,m-1},\X_{j,r,k}]-\bv^{c_{i,j}}[\H_{i,s,m-1},\X_{j,r+1,k}]_{\bv^{-2c_{ i,j}}}
\\
&+\bv^{c_{i,j}-c_{\tau i,j}}[ \H_{i,s,m-2},\X_{j,r+1,k-1}]_{\bv^{2(c_{\tau i,j}-c_{i,j})}}- \bv^{-c_{\tau i, j}}[\H_{i,s+1,m-2},\X_{j,r,k-1}]_{\bv^{2c_{\tau i,j}}}
\\
=&\ (\bv^{c_{ij}}-1)(\H_{i,s,m-1} \X_{j,r,k}+\bv^{-c_{ij}}\X_{j,r,k}\H_{i,s,m-1}
\\
&\qquad \qquad \qquad-\bv^{-c_{\tau i,j}}\H_{i,s,m-2} \X_{j,r,k-1}-\bv^{c_{\tau i,j}-c_{ij}}\X_{j,r,k-1}\H_{i,s,m-2}).
\end{align*}
Then we have
\begin{align*}
[\H_{i,s+2,m-2},&\,\X_{j,r,k-1}] - [\H_{i,s,m-2},\X_{j,r+2,k-1}]
\\
=&\ (\bv^{c_{ij}}-1)(\H_{i,s+1,m-2} \X_{j,r,k-1}+\H_{i,s,m-2} \X_{j,r+1,k-1} \\
&\hskip1.8cm +\X_{j,r+1,k-1} \H_{i,s,m-2}+\X_{j,r,k-1}\H_{i,s+1,m-2})
\\
&+ (\bv^{c_{ij}}-1)(\H_{i,s,m-1} \X_{j,r+1,k-1}  +\X_{j,r+1,k-1} \H_{i,s,m-1})
\\
&+(\bv^{-c_{\tau i,j}}-1)(\H_{i,s+1,m-2} \X_{j,r,k-1}+ \X_{j,r,k-1}\H_{i,s+1,m-2})
\\
&- (\bv^{c_{ij}-c_{\tau i,j}}-1)(\H_{i,s,m-2} \X_{j,r+1,k-1}  +\X_{j,r+1,k-1} \H_{i,s,m-2})
\\
&+(\bv^{c_{ij}}-1)(\bv^{c_{\tau i,j}}-1)[\H_{i,s,m-2}, \X_{j,r,k-1}]+\mc O(\h^{r+s+3}).
\end{align*}
Using Lemma~\ref{qslemma1}, We obtain the desired identity. 
\end{proof}

\subsubsection{Relation \eqref{qsconj4} for $j=\tau i$}  We verify the degenerate relation for \eqref{qsconj4} for $j=\tau i$ in this subsection. 

\begin{prop}
We have the following relation in $\Gr_{\mathbf{K}}\tUi$
\begin{align*}
[\ov{\X}_{i,r,1},\ov{\X}_{\tau i,s,1}]=(-1)^s\ov{\H}_{i,r+s, 1} .
\end{align*}
\end{prop}
\begin{proof}
Recall the relation \eqref{qsiDR4} admits an equivalent form \eqref{qsiDR10}. Taking a summation with respect to $m$ in \eqref{qsiDR10}, we obtain
\begin{align}\label{eq:qsdegen6}
[\X_{i,r,k+1},\X_{\tau i,0,k}]&  = \H_{i,r, 1}.
\end{align}
We calculate $\eqref{eq:qsdegen6}_{k}-\eqref{eq:qsdegen6}_{k-1}$ and then obtain the following relation
\begin{align*}
\begin{split}
&[\X_{i,r+1,k},\X_{\tau i,0,k-1}] +[\X_{i,r,k},\X_{\tau i,1,k-1}]=\mc O(\h^{r+2}).
\end{split}
\end{align*}
We rewrite the left-hand side using \eqref{eq:qsdegen6}
\begin{align*}
[\X_{i,r,k},\X_{\tau i,1,k-1}]=-\H_{i,r+1, 1}+\mc O(\h^{r+2}).
\end{align*}
Inductively, we obtain the following relation
\begin{align*}
[\X_{i,r,k+1-s},\X_{\tau i,s,k-s}]=(-1)^s\H_{i,r+s, 1}+\mc O(\h^{r+s+1}).
\end{align*}
Hence, by Lemma~\ref{qslemma1}, we obtain the desired relation.  
\end{proof}

\subsubsection{Other relations}
Note that formulations of defining relations \eqref{qsiDR1}, \eqref{qsiDR5}-\eqref{qsiDR3}, \eqref{qsiDR9'}-\eqref{qsiDR8} for $\tUi$ are the same as the split type. Hence, their corresponding degenerate relations \eqref{qsconj0}, \eqref{qsconj3} for $j\neq \tau i$ or $j=\tau i=i$, \eqref{qsconj4} for $j\neq \tau i$, and \eqref{qsconj8} can be verified in the same way as split type; see \cite[\S 5]{KLWZ23b}. The relation \eqref{qsconj3} for $j=\tau i\ne i$ is a corollary of the relation \eqref{qsconj4} for $j=\tau i$ while the second relation in \eqref{qsconj0} (or more generally \eqref{hrel}) follows from Lemma \ref{qslemma1}.

\section{Twisted Yangians in R-matrix presentation}\label{sec:R}
The main goal for Sections~\ref{sec:R}-\ref{sec:lowrk} is to show that the twisted Yangians introduced in Definition \ref{deftY} are isomorphic to the twisted Yangians (under the name \textit{reflection algebra}) introduced by Molev and Ragoucy \cite{MR02}. In the rest of the paper, we shall set $\h=1$ for simplicity; see Remark \ref{rem:h=1}. 

In this section, we recall the basics for the twisted Yangians $\Y_N$ of quasi-split type AIII defined in the R-matrix presentation from \cite{MR02}.

\subsection{Yangians}
\label{subsec:Y}
We start with recalling the basic theory of Yangian $\rY(\gl_N)$ from \cite{MNO96,Mol07}.
\begin{dfn}
The \textit{Yangian} $\rY(\gl_N)$ corresponding to the Lie algebra $\gl_N$ is a unital associative algebra with generators $t_{ij}^{(r)}$, where $1\lle i,j\lle N$ and $r\in\bZ_{>0}$, and the defining relations written in terms of the generating series
\[
t_{ij}(u)=  \delta_{ij}+t_{ij}^{(1)}u^{-1}+t_{ij}^{(2)}u^{-2}+\cdots
\]
by the relations,
\beq\label{Trel}
(u-v)[t_{ij}(u),t_{kl}(v)]=t_{kj}(u)t_{il}(v)-t_{kj}(v)t_{il}(u).
\eeq
\end{dfn}

The Yangian $\rY(\gl_N)$ has the following R-matrix presentation. Let $R(u)$ be the Yang R-matrix
\begin{align} \label{Ru}
R(u)=1-\frac{P}{u}\in \End(\bC^N\otimes \bC^N)[u^{-1}],\quad \text{where} \quad P=\sum_{i,j=1}^N E_{ij}\otimes E_{ji},
\end{align}
and
\[
T(u)=\sum_{i,j=1}^N E_{ij}\otimes t_{ij}(u)\in \End(\bC^N)\otimes\rY(\gl_N)[[u^{-1}]]. 
\]
Then the defining relations of $\rY(\gl_N)$ can be written as
\be
R(u-v)T_1(u)T_2(v)=T_2(v)T_1(u)R(u-v).
\ee

Note that the Yang R-matrix satisfies the Yang-Baxter equation
\beq\label{ybeq}
R_{12}(u-v)R_{13}(u)R_{23}(v)=R_{23}(v)R_{13}(u)R_{12}(u-v).
\eeq

Let $g(u)$ be any formal power series in $u^{-1}$ with leading term $1$,
\[
g(u)=1+g_1u^{-1}+g_2u^{-2}+\cdots\in \bC[[u^{-1}]].
\]
There is an automorphism of $\rY(\gl_N)$ defined by
\beq\label{eq:mu_f-A}
T(u)\to g(u)T(u).
\eeq

The Yangian for $\mathfrak{sl}_N$ is the subalgebra $\rY(\mathfrak{sl}_N)$ of $\rY(\gl_N)$ which consists of all elements stable under all the automorphisms of the form \eqref{eq:mu_f-A}.

Consider the filtration on $\rY(\gl_N)$ obtained by setting
\begin{align}  \label{filter:Y}
\deg t_{ij}^{(r)}=r-1
\end{align}
for every $r\gge 1$. Denote by $\mathrm{gr}\rY(\gl_N)$ the associated graded algebra. We write $\bar t_{ij}^{(r)}$ the image of $t_{ij}^{(r)}$ in $\mathrm{gr}\rY(\gl_N)$. Then the map
\beq\label{eq:cl-limitA}
\mathrm{U}(\gl_N[z])\to \mathrm{gr}\rY(\gl_N), \qquad e_{ij} z^{r}\mapsto \bar t_{ij}^{(r+1)},
\eeq
induces an Hopf algebra isomorphism.  

\subsection{Twisted Yangians}\label{subsec:tY}
Fix $N$ and define $i'=N+1-i$ for $1\lle i\lle N$. Let $G=(g_{ij})$ be the $N\times N$ matrix defined by $g_{ij}=\delta_{ij'}$. For any $N\times N$ matrix $M=(m_{ij})$, define
\beq\label{mat'}
M'=G M G^{-1}=(m_{i'j'}).
\eeq
In particular, we have the modified R-matrix,
\beq\label{twistedR}
R'(u)=G_1R(u)G_1=G_2R(u)G_2.
\eeq
In this case, $\tau$ is the permutation of $\{1,2,\cdots,N-1\}$ defined by $\tau i=i'-1$.

The following twisted Yangians were specific reflection algebras \cite{Sk88} introduced in \cite{MR02}.

\begin{dfn}
The  \textit{twisted Yangian} $\Y_N$ of quasi-split type AIII is a unital associative algebra with generators $s_{ij}^{(r)}$, where $1\lle i,j\lle N$ and $r\in\bZ_{>0}$, and the defining relations written in terms of the generating series
\beq\label{siju}
s_{ij}(u)=  \delta_{ij}+s_{ij}^{(1)}u^{-1}+s_{ij}^{(2)}u^{-2}+\cdots
\eeq
by the \textit{quaternary} relations,

\beq\label{bcom}
\begin{split}
(u^2-v^2)[s_{ij}(u),s_{kl}(v)]&=  (u+v)(s_{kj}(u)s_{il}(v)-s_{kj}(v)s_{il}(u))\\
&  + (u-v)\Big(\delta_{kj'}\sum_{a=1}^N s_{ia'}(u)s_{al}(v)-\delta_{il'}\sum_{a=1}^N s_{ka'}(v)s_{aj}(u)\Big)\\
& - \delta_{ij'}\Big(\sum_{a=1}^N s_{ka'}(u)s_{al}(v)-\sum_{a=1}^N s_{ka'}(v)s_{al}(u)\Big)
\end{split}
\eeq
and the \textit{unitary} relation
\beq\label{bunit}
\sum_{a=1}^N s_{ia'}(u)s_{aj}(-u)=\delta_{ij'}.
\eeq
\end{dfn}
Define the operator $S(u)\in  \End(\bC^N)\otimes\Y_N[[u^{-1}]]$,
\[
S(u)=\sum_{i,j=1}^N E_{ij}\otimes s_{ij}(u),
\]
and we treat it as a matrix $S(u)=(s_{ij}(u))$. Then the defining relations of $\Y_N$ are given by
\beq\label{quamat}
R(u-v)S_1(u)R'(u+v)S_2(v)=S_2(v)R'(u+v)S_1(u)R(u-v),
\eeq
\beq\label{unimat}
S(u)S'(-u)=\mathbf{1}_N,
\eeq
where $\mathbf{1}_N$ is the identity matrix.
\begin{rem}
The presentation of $\Y_N$ is different from the presentation of the reflection algebra $\mathcal B(N,\lfloor\frac{N}{2}\rfloor)$ in \cite[\S 2]{MR02}; however, it is not hard to see that they are isomorphic, using similar arguments as in \cite[\S2.15]{Mol07}. Indeed, let $B(u)$ be the matrix of generators for $\mathcal B(N,\lfloor\frac{N}{2}\rfloor)$ defined in \cite[(2.13)]{MR02} with the initial matrix $G$ exactly given by $G=(\delta_{ij'})$; then the constant component of $B(u)$ is $G$. The matrix $S(u)$ here corresponds to $B(u)G$; see also the discussion at the end of \cite[\S2]{MR02}.
\end{rem}

It is convenient to work on the extended twisted Yangians defined below instead of twisted Yangians. By abuse of notations, we shall keep using the same notations for various elements, such as $s_{ij}(u)$ and $S(u)$, in the twisted Yangian and the extended twisted Yangian.
\begin{dfn}\label{eradef}
The \textit{extended twisted Yangian} $\scrX_N$ of quasi-split type AIII is the unital associative algebra  with generators $s_{ij}^{(r)}$, where $1\lle i,j\lle N$ and $r\in\bZ_{>0}$ satisfying the quaternary relations \eqref{quamat}, where $s_{ij}(u)$ is again given by \eqref{siju}.
\end{dfn}

\subsection{Basic properties}\label{sec:basics}
Let $\wtl S(u)=\big(S(u)\big)^{-1}=(\tl s_{ij}(u))$.
\begin{prop}[{\cite[Prop. 2.1]{MR02}}]\label{Binv}
In the algebra $\scrX_N$, the product $S(u)S'(-u)$ is a scalar matrix,
\beq\label{cudef}
S(u)S'(-u)=S'(-u)S(u)=c(u)\mathbf{1}_N,
\eeq
where $c(u)$ is an even series in $u^{-1}$ whose coefficients are central in $\scrX_N$. In particular, we have $s_{i'j'}(-u)=c(u)\tl s_{ij}(u)$.
\end{prop}

Before further discussing the properties of $\Y_N$, we collect a few facts about the inverse of $T(u)$ of $\rY(\gl_N)$. Define $\wtl T(u):=\big(T(u)\big)^{-1}$ and let $\tl t_{ij}(u)$ be its matrix elements, 
$$
\tl t_{ij}(u)=\delta_{ij}+\sum_{r>0} \tl {t}_{ij}^{(r)}u^{-r}.
$$
Then
\beq\label{T'ij}
\tl t_{ij}(u)=\delta_{ij}+\sum_{k>0} (-1)^k\sum_{a_1,\cdots,a_{k-1}=1}^N t_{ia_1}^\circ(u)t_{a_1a_2}^\circ(u)\cdots t_{a_{k-1}j}^\circ(u),
\eeq
where $t_{ij}^\circ(u)=t_{ij}(u)-\delta_{ij}$. In particular, by taking the coefficient of $u^{-r}$, for $r\gge 1$, one obtains 
\beq\label{t'ijr}
\tl t_{ij}^{(r)}=\sum_{k=1}^r (-1)^k\sum_{a_1,\cdots,a_{k-1}=1}^N\sum_{r_1+\cdots+r_k=r}t_{ia_1}^{(r_1)}t_{a_1a_2}^{(r_2)}\cdots t_{a_{k-1}j}^{(r_k)},
\eeq
where $r_i$ for $1\lle i\lle k$ are positive integers.

It is well known that $\Y_N$ can be identified as a subalgebra of $\rY(\gl_N)$, see \cite[Thm. 3.1]{MR02}. Specifically, the map
\[
S(u)\mapsto T(u)\wtl T'(-u)
\]
defines an algebra embedding $\Y_N\hookrightarrow \rY(\gl_N)$.
Moreover, there is a filtration on $\Y_N$ inherited from the one \eqref{filter:Y} on $\rY(\gl_N)$ such that $\deg s_{ij}^{(r)}=r-1$. Let $\mathcal F_s(\Y_N)$ be the subspace of $\Y_N$ spanned by elements of degree $\lle s$ such that
\begin{align}
\label{filter:B}
    \mathcal F_0(\Y_N) \subset \mathcal F_1(\Y_N) \subset \mathcal F_2(\Y_N) \subset \ldots, 
    \qquad\qquad \Y_N =\bigcup_{s\gge 0}  \mathcal F_s(\Y_N).
\end{align}
Denote by $\gr\, \Y_N$ the associated graded algebra. Let $\bar s_{ij}^{(r)}$ be the image of $s_{ij}^{(r)}$ in the $(r-1)$-st component of $\gr\,\Y_N$. Then
\beq\label{image-quo}
\bar s_{ij}^{(r)}=\bar t_{ij}^{(r)}-(-1)^r\bar t_{i'j'}^{(r)}.
\eeq

\begin{prop}\label{pbwprop1}
Given any linear order on the set of generators
\begin{enumerate}
    \item $s_{ij}^{(r)}$, $s_{kk}^{(r)}$, if $N=2n$;
    \item $s_{ij}^{(r)}$, $s_{kk}^{(r)}$, $s_{n+1,n+1}^{(2r-1)}$, if $N=2n+1$;
\end{enumerate}
where $1\lle i<j\lle N$, $r\in\bZ_{>0}$, $1\lle k\lle n$, then the order monomials in these generators form a basis of $\Y_N$.
\end{prop}
\begin{proof}
The proof is parallel to that of \cite[Thm. 3.1 \& Coro. 3.2]{MR02}.
\end{proof}

Let $\vartheta$ be the involution\footnote{We use $\vartheta$ instead of $\omega_0$ from Section~\ref{sec:twloop} as they are slightly different; however, one can show that $\vartheta$ is conjugated to $\omega_0$ by rescaling automorphisms.} of $\gl_N$ defined by
\beq\label{varth}
\vartheta: \gl_N\longrightarrow \gl_N,\quad e_{ij}\mapsto e_{i'j'}.
\eeq
Extend this involution to $\gl_N[z]$ by sending $g\otimes z^r$ to $\vartheta(g)\otimes(-z)^r$ for $g\in \gl_N$ and $r\in\bN$. Let $\gl_N[z]^\vartheta$ be the fixed point subalgebra of $\gl_N[z]$ under the involution $\vartheta$. Then it is well known that the map
\beq\label{isogr}
\mathrm{U}(\gl_N[z]^\vartheta)\longrightarrow \gr\,\Y_N,\qquad \big(e_{ij}+(-1)^re_{i'j'}\big) z^{r} \mapsto \bar s_{ij}^{(r+1)}
\eeq
induces an algebra isomorphism, cf. \eqref{eq:cl-limitA} and \eqref{image-quo}. By restriction, we can also define $\mathfrak{sl}_N[z]^\vartheta$ and $\mathrm U(\mathfrak{sl}_N[z]^\vartheta)$.

\section{Gauss decomposition approach}\label{sec:GD}

In this section, we formulate and study the Gauss decomposition for twisted Yangians. Using the Gaussian generators, we establish in Theorem~\ref{main2} an isomorphism between $\Yi$ introduced in Definition \ref{deftY} and the special twisted Yangian $\SY_N$.

\subsection{Quasi-determinants and Gauss decomposition}
We shall also need the quasi-determinant presentation, see \cite{GGRW:2005}, of Drinfeld current generating series in terms of R-matrix generating series. Let $X$ be a square matrix over a ring with identity such that its inverse matrix $X^{-1}$ exists, 
and such that its $(j,i)$-th entry is an invertible element of the ring.  Then the $(i,j)$-th
\emph{quasi-determinant} of $X$ is defined by the first formula below and denoted graphically by the boxed notation (cf. \cite[\S1.10]{Mol07}):
\begin{equation*}
\vert X\vert _{ij} \stackrel{\text{def}}{=} \left((X^{-1})_{ji}\right)^{-1} = \left\vert  \begin{array}{ccccc} x_{11} & \cdots & x_{1j} & \cdots & x_{1n}\\
&\cdots & & \cdots&\\
x_{i1} &\cdots &\boxed{x_{ij}} & \cdots & x_{in}\\
& \cdots& &\cdots & \\
x_{n1} & \cdots & x_{nj}& \cdots & x_{nn}
\end{array} \right\vert .
\end{equation*}

By \cite[Thm. 4.96]{GGRW:2005}, the matrix $S(u)$, for both $\scrX_N$ and $\Y_N$, has the following Gauss decomposition:
$$
S(u) = F(u) D(u) E(u)
$$
for unique matrices of the form
\begin{equation*}
D(u) = \left[ \begin{array}{cccc} d_1 (u) & &\cdots & 0\\
& d_2 (u) &  &\vdots\ \\
\vdots & &\ddots &\\
0 &\cdots &  &d_{N} (u)
\end{array} \right],
\end{equation*}
\begin{equation*}
E(u)=\left[ \begin{array}{cccc} \!\!1 &e_{12}(u) &\cdots & e_{1N}(u)\!\!  \\
&\ddots & &e_{2N}(u) \!\! \\
& &\ddots & \vdots\\
0 & & &1 
\end{array} \right],\qquad 
F(u) = \left[ \begin{array}{cccc} \!\!1 & &\cdots &0\!\!\\
f_{21}(u) &\ddots & &\vdots\\
\vdots & & \ddots& \\
\!\!f_{N1}(u) & f_{N2}(u) &\cdots &1\!
\end{array} \right],
\end{equation*}
where the matrix entries are defined in terms of quasi-determinants:
\begin{eqnarray}
d_i (u) &=& \left\vert  \begin{array}{cccc} s_{11}(u) &\cdots &s_{1,i-1}(u) &s_{1i}(u) \\
\vdots &\ddots & &\vdots \\
s_{i1}(u) &\cdots &s_{i,i-1}(u) &\mybox{$s_{ii}(u)$}
\end{array} \right\vert, 
\qquad \tl d_i(u)=d_i(u)^{-1}, \label{gd1}
\\
e_{ij}(u) &=&\tl d_i (u) \left\vert  \begin{array}{cccc} s_{11}(u) &\cdots & s_{1,i-1}(u) & s_{1j}(u) \\
\vdots &\ddots &\vdots & \vdots \\
s_{i-1,1}(u) &\cdots &s_{i-1,i-1}(u) & s_{i-1,j}(u)\\
s_{i1}(u) &\cdots &s_{i,i-1}(u) &\mybox{$s_{ij}(u)$}
\end{array} \right\vert,\label{gd2}
\\
f_{ji}(u) &=& \left\vert  \begin{array}{cccc} s_{11}(u) &\cdots &s_{1, i-1}(u) & s_{1i}(u) \\
\vdots &\ddots &\vdots &\vdots \\
s_{i-1,1}(u) &\cdots &s_{i-1,i-1}(u) &s_{i-1,i}(u)\\
s_{j1}(u) &\cdots &s_{j, i-1}(u) &\mybox{$s_{ji}(u)$} 
\end{array} \right\vert\, 
\tl d_{i}(u).\label{gd3}
\end{eqnarray}
The Gauss decomposition can also be written component-wise as, for $i<j$, 
\begin{align}
s_{ii}(u)&=d_i(u)+\sum_{k<i}f_{ik}(u)d_k(u)e_{ki}(u),\nonumber\\
s_{ij}(u)&=d_i(u)e_{ij}(u)+\sum_{k<i}f_{ik}(u)d_k(u)e_{kj}(u),\label{eq:sij-Gauss}\\
s_{ji}(u)&=f_{ji}(u)d_i(u)+\sum_{k<i}f_{jk}(u)d_k(u)e_{ki}(u).\nonumber
\end{align}

We further denote
\begin{align}
    \label{gauss-gen}
e_{ij}(u) &=\sum_{r\gge 1}e_{ij}^{(r)}u^{-r},\quad f_{ji}(u)=\sum_{r\gge 1}f_{ji}^{(r)}u^{-r},\quad d_k(u)=1+\sum_{r\gge 1}d_{k}^{(r)}u^{-r},
\\
e_i(u) &=\sum_{r\gge 1}e_{i}^{(r)}u^{-r}=e_{i,i+1}(u),\quad f_i(u)=\sum_{r\gge 1}f_{i}^{(r)}u^{-r}=f_{i+1,i}(u),\quad 1\lle i<N.
\end{align}

Set 
\beq\label{edfinv}
\begin{split}
&\wtl D(u)=D(u)^{-1}=\sum_{1\lle i\lle N} E_{ii}\otimes \tl d_{i}(u),
\\
&\wtl E(u)=E(u)^{-1}=1+\sum_{1\lle i<j\lle N}E_{ij}\otimes \tl e_{ij}(u),\\
&\wtl F(u)=F(u)^{-1}=1+\sum_{1\lle i<j\lle N}E_{ji}\otimes \tl f_{ji}(u).
\end{split}
\eeq
Then we have 
\beq\label{eq:def-tilde-e-f}
\begin{split}
&\tl e_{ij}(u)=\sum_{i=i_0<i_1<\cdots<i_s=j}(-1)^s e_{i_0i_1}(u)e_{i_1i_2}(u)\cdots e_{i_{s-1}i_s}(u),\\
&\tl f_{ji}(u)=\sum_{i=i_0<i_1<\cdots<i_s=j}(-1)^s f_{i_{s}i_{s-1}}(u)\cdots f_{i_2i_1}(u) f_{i_1i_0}(u).
\end{split}
\eeq

\subsection{A homomorphism $\scrX_{N-2m}\to \scrX_{N}$}

Unlike the case of type AI in \cite{KLWZ23a}, the commutator relations \eqref{bcom} for type AIII involve summations. Consequently, there is no obvious embedding from $\scrX_{N-m}$ to $\scrX_{N}$ for $m>0$. From the viewpoint of Satake diagrams, one still expects a homomorphism from $\scrX_{N-2m}$ to $\scrX_{N}$. The main goal of this section is to construct such a homomorphism, following a similar strategy of \cite[\S3]{JLM18}.

For $2\lle  i\lle 2' (=N-1)$, by \eqref{bcom}, we have 
\beq\label{emrel1}
s_{11}(u+1)s_{i 1}(u)=s_{i1}(u+1)s_{11}(u).
\eeq
Therefore, we can rewrite the following quasi-determinant
\begin{align*}
\begin{vmatrix}s_{11}(u) & s_{1j}(u)\\
s_{i1}(u) & \mybox{$s_{ij}(u)$}\end{vmatrix}&=s_{ij}(u)-s_{i1}(u)s_{11}(u)^{-1}s_{1j}(u)\\
&=s_{11}(u+1)^{-1}\big(s_{11}(u+1)s_{ij}(u)-s_{i1}(u+1)s_{1j}(u)\big).
\end{align*}
Set 
\beq\label{newt}
\mc T_{ij}(u)=s_{11}(u+1)s_{ij}(u)-s_{i1}(u+1)s_{1j}(u)=s_{11}(u+1)\begin{vmatrix}s_{11}(u) & s_{1j}(u)\\
s_{i1}(u) & \mybox{$s_{ij}(u)$}\end{vmatrix},
\eeq
and introduce
\beq\label{gamma}
\begin{split}
\Gamma(u):=&\sum_{a_i,b_i}E_{a_1b_1}\otimes E_{a_2b_2}\otimes \Gamma_{b_1b_2}^{a_1a_2}(u)\\
=&\,R_{12}(1)S_1(u+1){R}_{12}'(2u+1)S_2(u)
=S_2(u){R}_{12}'(2u+1)S_1(u+1)R_{12}(1),
\end{split}
\eeq
where $\Gamma_{b_1b_2}^{a_1a_2}(u)$ are matrix entries for $\Gamma(u)$ and the last equality follows from \eqref{quamat}.
\begin{lem}\label{Blem}
We have
\begin{enumerate}
    \item  $[s_{11}(u),\mc T_{ij}(v)]=0$, $2\lle i,j\lle 2'$;
\item $\Gamma_{1j}^{1i}(u)=\mc T_{ij}(u)$, $2\lle i,j\lle 2'$;
\item $\Gamma_{b_1b_2}^{a_1a_2}(u)=-\Gamma_{b_1b_2}^{a_2a_1}(u)=-\Gamma_{b_2b_1}^{a_1a_2}(u)$.
\end{enumerate}
\end{lem}
\begin{proof}
(1) Note that by \eqref{bcom} we have $[s_{11}(u),s_{11}(v)]=0$. It follows from \eqref{bcom} that
\begin{align*}
&[s_{11}(u),\mc T_{ij}(v)]=[s_{11}(u),s_{11}(v+1)s_{ij}(v)-s_{i1}(v+1)s_{1j}(v)]\\
=&\,\frac{1}{u-v}s_{11}(v+1)\Big(s_{i1}(u)s_{1j}(v)- s_{i1}(v)s_{1j}(u) \Big)-\frac{1}{u-v-1}\Big(s_{i1}(u)s_{11}(v+1)\\
&\quad\quad -s_{i1}(v+1)s_{11}(u)\Big)s_{1j}(v)-\frac{1}{u-v}s_{i1}(v+1)\Big(s_{11}(u)s_{1j}(v)- s_{11}(v)s_{1j}(u) \Big).
\end{align*}
Due to \eqref{emrel1}, it suffices to show that
\begin{align*}
\frac{1}{u-v}s_{11}(v+1)s_{i1}(u)&-\frac{1}{u-v-1}s_{i1}(u)s_{11}(v+1) \\&+\frac{1}{(u-v)(u-v-1)}s_{i1}(v+1)s_{11}(u)=0,
\end{align*}
which is equivalent to
\[
(v+1-u)[s_{11}(v+1),s_{i1}(u)]=s_{i1}(v+1)s_{11}(u)-s_{i1}(u)s_{11}(v+1).
\]
This follows directly from \eqref{bcom}.

(2) Computing $\Gamma_{1j}^{1i}(u)$, $2\lle i,j\lle 2'$, using the definition \eqref{gamma}, one finds that it is given by $$s_{11}(u+1)s_{ij}(u)-s_{i1}(u+1)s_{1j}(u)$$ which coincides with $\mc T_{ij}(u)$ in \eqref{newt}.

(3) Note that $(1-P_{12}) R_{12}(1)=2R_{12}(1)= R_{12}(1)(1-P_{12})$. Thus $R_{12}(1)$ remains unchanged when multiplying by $(1-P_{12})/2$ from the left or the right. Then multiplying by $(1-P_{12})/2$ from the left to \eqref{gamma}, one derives $\Gamma_{b_1b_2}^{a_1a_2}(u)=-\Gamma_{b_1b_2}^{a_2a_1}(u)$. The other identity $\Gamma_{b_1b_2}^{a_1a_2}(u)=-\Gamma_{b_2b_1}^{a_1a_2}(u)$ is obtained by multiplying by $(1-P_{12})/2$ from the right to \eqref{gamma}.
\end{proof}

We will need the following simplified expression of \eqref{ybeq} when $v=u-1$.
\begin{lem}[\text{cf. \cite[Coro. 1.7.3]{Mol07}}]
\label{Rsim}
We have the following relations,
\be
R_{12}(1)R_{13}(u)R_{23}(u-1)=R_{12}(1)\Big(1-\frac{P_{13}+P_{23}}{u-1}\Big),
\ee
\be
R_{23}(u-1)R_{13}(u)R_{12}(1)=\Big(1-\frac{P_{13}+P_{23}}{u-1}\Big)R_{12}(1).
\ee
\end{lem} 


\begin{prop}\label{embedB}
The map $s_{ij}(u)\mapsto \mc T_{ij}(u)$, $2\lle i,j\lle 2'$, defines a homomorphism $\scrX_{N-2}\to \scrX_{N}$.
\end{prop}
\begin{proof}We first introduce some shorthand notations. 
Let $u$ and $v$ be parameters. Set $a=u-v$, $\tl a=u+v$. We have the following equality in the algebra $\End(\bC^N)^{\otimes 4}\otimes \scrX_N$, 
\beq\label{pf1}
\begin{split}
&R_{23}(a-1)R_{13}(a)R_{24}(a)R_{14}(a+1)\Gamma_{12}(u)R'_{14}(\tl a+1)R'_{24}(\tl a)R'_{13}(\tl a+2)R'_{23}(\tl a+1)\Gamma_{34}(v)\\
&=\Gamma_{34}(v)R'_{23}(\tl a+1)R'_{13}(\tl a+2)R'_{24}(\tl a)R'_{14}(\tl a+1)\Gamma_{12}(u)R_{14}(a+1)R_{24}(a)R_{13}(a)R_{23}(a-1).
\end{split}
\eeq

This follows from the Yang-Baxter equation \eqref{ybeq}, and the relations \eqref{twistedR}, \eqref{quamat}. We shall rewrite both sides of \eqref{pf1} by Lemma \ref{Rsim} and then equate certain matrix elements.

Consider the right hand side of \eqref{pf1}. Applying \eqref{gamma} and Lemma \ref{Rsim}, we have
\begin{align*}
\Gamma_{34}(v)&R'_{23}(\tl a+1)R'_{13}(\tl a+2)R'_{24}(\tl a)R'_{14}(\tl a+1)\Gamma_{12}(u)R_{14}(a+1)R_{24}(a)R_{13}(a)R_{23}(a-1)\\
&=\Gamma_{34}(v)\Big(1-\frac{P_{13}'+P_{23}'}{\tl a+1}\Big)\Big(1-\frac{P_{14}'+P_{24}'}{\tl a}\Big)\Gamma_{12}(u)\Big(1-\frac{P_{14}+P_{24}}{a}\Big)\Big(1-\frac{P_{13}+P_{23}}{a-1}\Big).
\end{align*}
Then we apply the operator above to a basis vector of the form $e_1\otimes e_j\otimes e_1\otimes e_l$ for $j,l\in\{2,\cdots,2'\}$. First, an application of the factor
\[
\Big(1-\frac{P_{14}+P_{24}}{a}\Big)\Big(1-\frac{P_{13}+P_{23}}{a-1}\Big)
\]
on $e_1\otimes e_j\otimes e_1\otimes e_l$ provides
\beq\label{pf2}
\begin{split}
\frac{a-2}{a-1}&\Big(e_1\otimes e_j\otimes e_1\otimes e_l  -\frac{1}{a}e_1\otimes e_l\otimes e_1\otimes e_j\Big)- \frac{1}{a-1}(e_{1}\otimes e_1\otimes e_j\otimes e_l) \\
&+ \frac{1}{a(a-1)}(e_l\otimes e_1\otimes e_j\otimes e_1+e_1\otimes e_l\otimes e_j\otimes e_1) - \frac{a-2}{a(a-1)}(e_l\otimes e_j\otimes e_1\otimes e_1).
\end{split}
\eeq
It follows from Lemma \ref{Blem} (3) that terms inside the second and third brackets in \eqref{pf2} are annihilated by $\Gamma_{12}(u)$. We consider the further application of the rest factors acting on the last term in \eqref{pf2}. The application of $\Gamma_{12}(u)$ on $e_l\otimes e_j\otimes e_1\otimes e_1$ gives vectors of the form $e_a\otimes e_b\otimes e_1\otimes e_1$. A further application of the factor
\beq\label{pf3}
\Big(1-\frac{P_{13}'+P_{23}'}{\tl a+1}\Big)\Big(1-\frac{P_{14}'+P_{24}'}{\tl a}\Big)
\eeq
on $e_a\otimes e_b\otimes e_1\otimes e_1$ results in
\begin{align*}
e_a\otimes e_b\otimes e_1\otimes e_1&-\frac{1}{\tl a+1}(e_{1'}\otimes e_b\otimes e_1\otimes e_{a'}+e_{1'}\otimes e_b\otimes e_{a'}\otimes e_{1})\\
&-\frac{1}{\tl a+1}(e_a\otimes e_{1'}\otimes e_{b'}\otimes e_{1}+e_a\otimes e_{1'}\otimes e_1\otimes e_{b'})\\
&+\frac{1}{\tl a(\tl a+1)}(e_{1'}\otimes e_{1'}\otimes e_{a'}\otimes e_{b'}+e_{1'}\otimes e_{1'}\otimes e_{b'}\otimes e_{a'}).
\end{align*}
Again by Lemma \ref{Blem} (3), $\Gamma_{34}(v)$ annihilates the above vectors. Thus it suffices to consider the action of the operator 
\be
\Gamma_{34}(v)\Big(1-\frac{P_{13}'+P_{23}'}{\tl a+1}\Big)\Big(1-\frac{P_{14}'+P_{24}'}{\tl a}\Big)\Gamma_{12}(u)
\ee
on the terms inside the first bracket in \eqref{pf2}. Note that the action of $\Gamma_{12}(u)$ gives vectors of the form $e_{a}\otimes e_b\otimes e_1\otimes e_{c}$ where $2\lle c\lle 2'$. A similar calculation as above shows that the restriction of the image of $e_{a}\otimes e_b\otimes e_1\otimes e_{c}$ under the operator
$$
\Gamma_{34}(v)\Big(1-\frac{P_{13}'+P_{23}'}{\tl a+1}\Big)\Big(1-\frac{P_{14}'+P_{24}'}{\tl a}\Big)
$$
to the subspace spanned by the vectors of the form $e_{1}\otimes e_i\otimes e_1\otimes e_k$ with $2\lle i,k\lle 2'$ is nonzero only if $a=1$ and $2\lle b\lle 2'$. Moreover,
\begin{align*}
\Big(1-\frac{P_{13}'+P_{23}'}{\tl a+1}\Big)\Big(1-\frac{P_{14}'+P_{24}'}{\tl a}\Big)e_{1}\otimes e_b\otimes e_1\otimes e_{c}\equiv  \Big(1-\frac{P_{24}'}{\tl a}\Big)e_{1}\otimes e_b\otimes e_1\otimes e_{c},
\end{align*}
when the symbol $\equiv$ means we only keep the basis vectors which can give a nonzero contribution to the coefficients of $e_1\otimes e_i\otimes e_1\otimes e_k$ after the subsequent application of the operator $\Gamma_{34}(v)$.

To sum up, we have proved that the restriction of the operator on the right hand side of \eqref{pf1} to the subspace spanned by the basis vectors of the form $e_1\otimes e_j\otimes e_1\otimes e_l$ with $2\lle j,l\lle 2'$ coincides with the operator
\beq\label{pf4}
\frac{a-2}{a-1}\,\Gamma_{34}(v)\Big(1-\frac{P_{24}'}{\tl a}\Big)\Gamma_{12}(u)\Big(1-\frac{P_{24}}{a}\Big)=\frac{a-2}{a-1}\, \mc T_{4}(v) R_{24}'(u+v) \mc T_{2}(u)R_{24}(u-v),
\eeq
where $\mc T(u)= \big(\mc T_{ij}(u)\big)_{2\lle i,j\lle 2'}$ and the subscripts in $\mc T_{4}(v),\mc T_{2}(u)$ indicate which component they act on. The equality in \eqref{pf4} follows from Lemma~\ref{Blem} (2).


For the left hand side of \eqref{pf1}, we again apply it to the basis vectors of the form $e_1\otimes e_j\otimes e_1\otimes e_l$ with $2\lle j,l\lle 2'$ and look at the coefficients of the basis vectors of the same form in the image. Then the same argument as for the right hand side (with the reversed factors in the operators) implies that the coefficients of such basis vectors coincide with those of the operator
\beq\label{pf5}
\frac{a-2}{a-1}\,\Big(1-\frac{P_{24}}{a}\Big)\Gamma_{12}(u)\Big(1-\frac{P_{24}'}{\tl a}\Big)\Gamma_{34}(v)=\frac{a-2}{a-1}\,R_{24}(u-v) \mc T_{2}(u) R_{24}'(u+v) \mc T_{4}(v),
\eeq
where the equality follows again by Lemma~\ref{Blem} (2).

Note that $R_{24}(u-v)$ and $R_{24}'(u+v)$ are the R-matrices used to define the extended twisted Yangian $\scrX_{N-2}$. Therefore, by equating the matrix elements of the operators \eqref{pf4} and \eqref{pf5}, we get the R-matrix form of the defining relations for the algebra $\scrX_{N-2}$ is satisfied by the series $\mathcal T_{ij}(u)=\Gamma_{1j}^{1i}(u)$, see Lemma~\ref{Blem} (2), as required.
\end{proof}

We also need the following generalization of Proposition \ref{embedB}. Fix a positive integer $m$ such that
$m\leqslant n$ if $N=2n+1$ and $m\leqslant n-1$ if $N=2n$.
Suppose that the generators $s_{ij}^{(r)}$ of the algebra $\scrX_{N-2m}$ are
labelled by the indices
$m+1\leqslant i,j\leqslant (m+1)'$ and $r>0$.

\begin{prop}\label{prop:red}
The mapping
\beq\label{redu}
\psi_m:s_{ij}(u)\mapsto \left|\begin{matrix}
s_{11}(u)&\dots&s_{1m}(u)&s_{1j}(u)\\
\dots&\dots&\dots&\dots\\
s_{m1}(u)&\dots&s_{mm}(u)&s_{mj}(u)\\
s_{i1}(u)&\dots&s_{im}(u)&\mybox{$s_{ij}(u)$}
\end{matrix}\right|,\qquad m+1\leqslant i,j\leqslant (m+1)',
\eeq
defines an algebra homomorphism $\scrX_{N-2m}\to \scrX_{N}$.
\end{prop}
\begin{proof}
The $m=1$ case is Proposition~\ref{embedB}. The general case follows from an induction on $m$ parallel to the one in \cite[Proof of Prop. 3.7]{JLM18}, by using the Sylvester theorem for quasi-determinants and Proposition~\ref{embedB}.
\end{proof}

The homomorphisms $\psi_m$ have the following consistence property. We write $\psi_m=\psi_m^{(N)}$ to indicate the dependence of $N$. For $l\in\bN$, we have the corresponding homomorphism
\[
\psi_m^{(N-2l)}:\scrX_{N-2m-2l}\to \scrX_{N-2l}
\]
given by \eqref{redu}.

\begin{prop}\label{cons}
We have the equality of algebra homomorphisms,
\[
\psi_l^{(N)}\circ \psi_{m}^{(N-2l)}=\psi_{l+m}^{(N)}.
\]
\end{prop}
\begin{proof}
Follows from the same argument as in \cite[Prop. 3.8]{JLM18}.
\end{proof}

For subsets $\{a_1,\dots,a_k\}$ and $\{b_1,\dots,b_k\}$ of $\{1,\dots,N\}$,
introduce quantum minors by the formula
\be
{\mathcal S\,}^{a_1\dots\,a_k}_{b_1\dots\, b_k}(u)=
\sum_{ \sigma\in \fkS_k} \mathrm{sign}~\sigma\cdot s_{a_{ \sigma(1)}b_1}(u)\dots
s_{a_{ \sigma(k)}b_k}(u-k+1).
\ee
These are formal series in $u^{-1}$ with coefficients in $\scrX_{N}$. 

\begin{prop}\label{prop:quasiqua}
For all $m+1\leqslant i,j\leqslant (m+1)'$, we have the identity
\be
\left|\begin{matrix}
s_{11}(u)&\dots&s_{1m}(u)&s_{1j}(u)\\
\dots&\dots&\dots&\dots\\
s_{m1}(u)&\dots&s_{mm}(u)&s_{mj}(u)\\
s_{i1}(u)&\dots&s_{im}(u)&\mybox{$s_{ij}(u)$}
\end{matrix}\right|={\mathcal S}^{1\dots\, m}_{1\dots\, m}(u+m)^{-1}\cdot
{\mathcal S}^{1\dots\, m\, i}_{1\dots\, m\, j}(u+m).
\ee
\end{prop}
\begin{proof}
It follows from e.g. \cite[\S1.11]{Mol07} that the identity holds for the Yangian $\rY(\gl_N)$. Note that the commutator relations between the generating series $s_{ab}(u)$ of $\scrX_N$ appearing in this identity are the same as for the Yangian $\rY(\gl_N)$, except for the commutators $[*_1,*_2]$ where $*_1$ is from the last row and $*_2$ is from the last column. However, we shall always keep the series $s_{aj}(u)$ for $1\lle a\lle m$ or $a=i$ to the right-most. Moreover, the quantum minors depends linearly on such generating series. Thus the identity does not depends on such commutators and hence it holds in $\scrX_N$ as well. 
\end{proof}

The following is a counterpart of the corresponding property of the Yangian
$\rY(\gl_N)$; see, e.g., \cite{BK05}.

\begin{cor}\label{cor:commu}
We have the relations
\be
\big[s_{ab}(u),\psi_m(s_{ij}(v))\big]=0
\ee
for all $1\leqslant a,b\leqslant m$ and $m+1\leqslant i,j\leqslant (m+1)'$. 
\end{cor}
\begin{proof}
By Proposition \ref{prop:quasiqua}, it suffices to check that $s_{ab}(u)$ commutes with the quantum minors. This follows essentially the same argument as in the proof of Proposition \ref{prop:quasiqua}.
\end{proof}

\subsection{Gaussian generators and their properties}
Recall the definition of the central series $c(u)$ from \eqref{cudef} and note that the unitary relation in $\Y_N$ \eqref{unimat} is equivalent to $c(u)=1$.

\begin{lem}\label{efaiii}
In the algebra $\scrX_N$, we have $c(u)\tl d_{i'}(u)=d_{i}(-u)$,
\beq\label{efd}
\tl d_i(u)d_{i+1}(u)=\tl d_{i'-1}(-u)d_{i'}(-u),\quad  \tl e_{ij}(u)=f_{i'j'}(-u), \quad  \tl f_{ji}(u)=e_{j'i'}(-u).
\eeq
In particular, we have $e_{i}(u)=-f_{\tau i}(-u)$.
\end{lem}

Equivalently, we have $\wtl E(u)=F(-u)'$ in the matrix form; see \eqref{mat'} and \eqref{edfinv}. 

\begin{proof}
By \eqref{eq:def-tilde-e-f} we have $e_{i}(u)=-\tl e_{i,i+1}(u)$. Thus \eqref{efd} implies that $e_{i}(u)=-\tl e_{i,i+1}(u)=-f_{\tau i}(-u)$.

It remains to show \eqref{efd}. Consider the Gauss decompositions of $S(u)$ and $\wtl S(u)$,
\[
S(u)=F(u)D(u)E(u),\qquad \wtl S(u)=\wtl E(u) \wtl D(u)\wtl F(u),
\]
where $\wtl E(u), \wtl D(u), \wtl F(u)$ are defined in \eqref{edfinv}. Then for $1\lle i<j\lle N$, we have
\begin{align*}
\tl s_{i'i'}(u)&=\tl d_{i'}(u)+\sum_{k<i}\tl e_{i'k'}(u)\tl d_{k'}(u)\tl f_{k'i'}(u),\nonumber\\
\tl s_{i'j'}(u)&=\tl d_{i'}(u)f_{i'j'}(u)+\sum_{k<i}\tl e_{i'k'}(u)\tl d_{k'}(u)\tl f_{k'j'}(u),\\
\tl s_{j'i'}(u)&=\tl e_{j'i'}(u)\tl d_{i'}(u)+\sum_{k<i}\tl e_{j'k'}(u)\tl d_{k'}(u)\tl f_{k'i'}(u).\nonumber
\end{align*}
Comparing it with \eqref{eq:sij-Gauss} and using $s_{ij}(-u)=c(u)\tl s_{i'j'}(u)$ from Proposition \ref{Binv}, we find that
\[
c(u)\tl d_{i'}(u)=d_{i}(-u),\quad \tl e_{j'i'}(u)=f_{ji}(-u),\quad \tl f_{i'j'}(u)=e_{ij}(-u),
\]
as the Gauss decomposition is unique. Hence the lemma follows.
\end{proof}

\begin{lem}\label{shiftlem}
Suppose $m\lle \lfloor \tfrac{N-1}{2}\rfloor$, then the map $\psi_m:\scrX_{N-2m}\to \scrX_N$ sends
\[
d_i(u)\to d_{m+i}(u),\quad e_{ij}(u)\to e_{m+i,m+j}(u),\quad f_{ji}(u)\to f_{m+j,m+i}(u).
\]
\end{lem}
\begin{proof}
The lemma follows from \eqref{gd1}--\eqref{gd3} and Propositions \ref{embedB}, \ref{cons}; cf. the proof of \cite[Coro. 3.2]{KLWZ23a}.
\end{proof}

Due to Lemma \ref{shiftlem}, we call $\psi_m$ a \emph{shift homomorphism}.

\begin{lem}\label{lemnew1}
Suppose $m\lle \lfloor \tfrac{N-1}{2}\rfloor$. We have $[d_i(u),d_j(v)]=0$ for $1\lle i,j\lle N$ and 
\[
[d_{i}(u),e_j(v)]=[d_{i}(u),f_j(v)]=0,
\]
for (1) $1\lle i\lle m, m<j<(m+1)'$ and (2) $1\lle j< m, m+1\lle i\lle (m+1)'$.
\end{lem}
\begin{proof}
This is a corollary of Corollary \ref{cor:commu}, Lemma \ref{efaiii}, and Lemma \ref{shiftlem}. Note that $[d_1(u),d_1(v)]=0$ follows from \eqref{bcom} and this implies by Lemma \ref{shiftlem} that $[d_i(u),d_i(v)]=0$ for $1\lle i\lle N$.
\end{proof}

\begin{lem}\label{lem:eta}
There is an anti-automorphism $\eta$ for $\mathscr X_N$ (and for $\Y_N$) defined by
\begin{align}\label{eta}
\eta:S(u)\longrightarrow S^t(u),\qquad s_{ij}(u)\mapsto s_{ji}(u).
\end{align}
Moreover, for $1\lle i<j\lle N$ and $\ 1\lle k\lle N$, we have
\[
\eta\big(e_{ij}(u)\big)=f_{ji}(u),\qquad \eta\big(f_{ji}(u)\big)=e_{ij}(u),\qquad \eta\big(d_{k}(u)\big)=d_{k}(u).
\]
\end{lem}
\begin{proof}
It is straightforward to prove that $\eta$ defines an anti-automorphism for $\mathscr X_N$ and $\Y_N$. Applying $\eta$ to \eqref{eq:sij-Gauss}, the second statement follows from the uniqueness of Gauss decomposition.
\end{proof}

\begin{lem}\label{lem:de-gen}
The algebra $\scrX_N$ is generated by the coefficients of $d_i(u)$ and $e_j(u)$, where $1\lle i\lle N$ and $1\lle j<N$.
\end{lem}
\begin{proof}
We say that a series in $u^{-1}$ can be generated if its coefficients can be generated by the coefficients of $d_i(u)$ and $e_j(u)$, where $1\lle i\lle N$ and $1\lle j<N$.

By Lemma \ref{efaiii} and \eqref{eq:def-tilde-e-f}, it suffices to show that $e_{kl}(u)$ with $k<l$ can be generated. We prove it by induction on $N$. The base case $N=2$ is trivial. Now assume $N\gge 3$. Then it follows from Lemma \ref{shiftlem} and the induction hypothesis that $d_{k}(u), e_{ij}(u), f_{ji}(u)$ for $2\lle k\lle 2'$ and $2\lle i<j\lle 2'$ can be generated.

Then we prove that $e_{1j}(u)$ for $j\gge 3$. Note that $s_{1j}(u)=d_1(u)e_{1j}(u)$, it suffices to show $s_{1j}(u)$ can be generated by another induction on $j$. Now let $2\lle j<N$ and suppose that $s_{1k}(u)$ for $1\lle k\lle j$ can be generated. By \eqref{bcom}, we have
\begin{align*}
(u^2-v^2)[s_{1j}(u),s_{j,j+1}(v)]=&\,(u+v)\big(s_{jj}(u)s_{1,j+1}(v)-s_{jj}(v)s_{1,j+1}(u)\big)\\
+& \, (u-v)\Big(\delta_{jj'}\sum_{a=1}^N s_{1a'}(u)s_{a,j+1}(v)-\delta_{1(j+1)'}\sum_{a=1}^N s_{ja'}(v)s_{aj}(u)\Big).
\end{align*}
Taking the coefficients of $v$, we have
\beq\label{0000-}
s_{1,j+1}(u)=[s_{1j}(u),s_{j,j+1}^{(1)}]-\delta_{jj'}s_{1,j'-1}(u)+\delta_{1,j'-1}s_{j'j}(u)
\eeq
Clearly, $[s_{1j}(u),s_{j,j+1}^{(1)}]$ can be generated. If $j=j'$, then $s_{1,j'-1}(u)=s_{1,j-1}(u)=d_1(u)e_{1,j-1}(u)$ can be generated by induction hypothesis. If $j'=2$, then $s_{j'j}(u)=s_{2j}(u)$ can be expressed by
\[
s_{j'j}(u)=\begin{cases}d_2(u)e_{2j}(u)+f_{1}(u)d_1(u)e_{1j}(u),&\text{ if }j> 2.\\
d_2(u)+f_{1}(u)d_1(u)e_{1}(u),  &\text{ if }j=2.
\end{cases}
\]
can also be generated by induction hypothesis. Thus it follows from \eqref{0000-} that $s_{1,j+1}(u)$ can be generated.

Similarly, one proves that $f_{j1}(u)$ can also be generated. Finally, using \eqref{eq:def-tilde-e-f}, we find that $e_{jN}(u)=\tl f_{j'1}(-u)$ can be generated. Similarly, $f_{Nj}(u)$ can also be generated. The proof is complete.
\end{proof}

\begin{prop}
Fix a linear ordering on the elements $f_{ji}^{(r)}$, $1\lle i<j\lle N$, $d_{kk}^{(r)}$, $1\lle k \lle \lfloor \tfrac{N}{2}\rfloor$, and $d_{n+1,n+1}^{(2r-1)}$ if $N=2n+1$, $r>0$. Then the ordered monomials of these elements form a basis of $\Y_N$.
\end{prop}
\begin{proof}
Recall the filtration from \eqref{filter:B} and the Gauss decomposition from \eqref{eq:sij-Gauss}. Then we have $f_{ji}^{(r)}\equiv s_{ji}^{(r)}$ and $d_{ii}^{(r)}\equiv s_{ii}^{(r)}$, modulo terms of degree at most $r-2$. Thus the statement follows from Proposition \ref{pbwprop1}.
\end{proof}

\subsection{Special twisted Yangians}\label{sec:hb}
For later use, we define the following generating series, for $0\lle i\lle N$ and $1\lle j< N$,
\begin{enumerate}
    \item if $N=2n$ is even, then we set
\begin{align}
& b_j(u)= \sqrt{-1}f_j(u+\tfrac{n-j}{2}),\label{beven}\\
& h_i(u)=\tl d_{i}(u+\tfrac{n-i}{2})d_{i+1}(u+\tfrac{n-i}{2})\label{heven};
\end{align}
\item if $N=2n+1$ is odd, then we set
\begin{align}
&b_j(u)= \sqrt{-1}f_j(u+\tfrac{n-j}{2}+\tfrac{1}{4}),\label{bodd}\\
&h_i(u)=
\begin{cases}
\big(1+\tfrac{1}{4u}\big)\tl d_{i}(u+\tfrac{1}{4})d_{i+1}(u+\tfrac{1}{4}), &\text{ if } i=n,\\
\big(1-\tfrac{1}{4u}\big)\tl d_{i}(u-\tfrac{1}{4})d_{i+1}(u-\tfrac{1}{4}), &\text{ if } i=n+1,\\
\tl d_{i}(u+\tfrac{n-i}{2}+\tfrac{1}{4})d_{i+1}(u+\tfrac{n-i}{2}+\tfrac{1}{4}), &\text{ if } i\ne n,n+1.
\end{cases}\label{hodd}
\end{align}
\end{enumerate}
Here we set $d_0(u)=d_{N+1}(u)=1$. Clearly, we have
\beq\label{prod}
\prod_{i=0}^N h_i\big(u-\tfrac{N-2i}{4}\big)=1-\delta_{N,\mathrm{odd}}\frac{1}{16u^2},
\eeq
where $\delta_{N,\mathrm{odd}}=1$ if $N$ is odd and $\delta_{N,\mathrm{odd}}=0$ otherwise.
\begin{rem}
Note that our special shifts satisfy
\begin{align*}
& \frac{n-i}{2}+\frac{n-\tau i}{2}=0,\hskip1.87cm\text{ if }N=2n,\\
& \frac{n-i}{2}+\frac{1}{4}+\frac{n-\tau i}{2}+\frac{1}{4}=0,\quad \text{ if }N=2n+1.
\end{align*}
The purpose to multiply $\sqrt{-1}$ is to change a sign so that it matches with the relations in Definition \ref{deftY}. The other modifications become apparent later in the calculation of relations of low rank cases.
\end{rem}

It follows from Lemma \ref{efaiii} that
\begin{align}
\label{i=tauih}
&h_{\tau i}(u)=h_{i}(-u),\\  
&b_{i}(u)=\begin{cases}
-\sqrt{-1}\,e_{\tau i}(-u+\tfrac{n- \tau i}{2}),&\text{ if }N=2n.\\
-\sqrt{-1}\,e_{\tau i}(-u+\tfrac{n- \tau i}{2}+\frac{1}{4}), &\text{ if }N=2n+1.
\end{cases}.\label{fi=tauie}
\end{align}
In particular, if $N=2n$ is even, then $h_n(u)$ is an even series in $u^{-1}$. Moreover, under the anti-automorphism $\eta$, we have
\beq\label{etahb}
\eta (h_i(u))=h_i(u),\qquad \eta(b_{i}(u))=-b_{\tau i}(-u).
\eeq

Introduce $h_{i,r}$ and $b_{j,r}$ for $0\lle i\lle N$, $1\lle j<N$, and $r\in\bN$ as follows,
\begin{align}
h_i(u)=1+\sum_{r\gge 0}h_{i,r}u^{-r-1},\qquad b_j(u)=\sum_{r\gge 0}b_{j,r}u^{-r-1},\label{bhcom}
\end{align}
namely they are coefficients of $h_i(u)$ and $b_j(u)$. It is also convenient to set $h_{i,-1}=1$ and $h_{i,r}=0$ if $r<-1$.
\begin{dfn}\label{defSY}
The \emph{special twisted Yangian $\SY_N$} is the subalgebra of $\Y_N$ generated by $b_{i,r}$ and $h_{i,r}$ for $1\lle i<N$ and $r\in\bN$.    
\end{dfn}

Define the root vectors $b_{\alpha,r}=b_{ji;r}$ for $\alpha=\alpha_i+\cdots+\alpha_{j-1}$ with $1\lle i<j\lle N$ and $r\in\bN$ recursively as follows, cf. \eqref{bjir1}, 
\beq\label{bjir2}
b_{i+1,i;r}=b_{i,r},\qquad b_{\alpha,r}=b_{ji;r}=[b_{j-1,0},b_{j-1,i;r}].
\eeq
Recall $\I_\tau^0$, $\I_{\ne}^0$, and $\I_{=}^0$ from \eqref{I-tau} and \eqref{eq:I+-}. Then we have
\beq\label{I+-A}
\begin{split}
&\I_{\ne}^0=\{1,\dots,n-1\},\qquad \I_{=}^0=\{n\},\qquad \text{ if }N=2n;\\
&\I_{\ne}^0=\{1,\dots,n\},\hskip 1.45cm \I_{=}^0=\varnothing,\hskip 1.1cm \text{ if }N=2n+1.
\end{split}
\eeq
\begin{prop}\label{prop:pwb}
The ordered monomials of 
\beq\label{dd34}
\{b_{\alpha,r},h_{0,2r},h_{i,r},h_{j,2r+1}~|~\alpha\in\cR^+,i\in\I^0_{\ne},j\in\I_{=}^0,r\in\N\}
\eeq
(with respect to any fixed total ordering) are linearly independent in  $\Y_N$. In particular, the ordered monomials of 
$$
\{b_{\alpha,r},h_{i,r},h_{j,2r+1}~|~\alpha\in\cR^+,i\in\I_{\ne}^0,j\in\I_{=}^0,r\in\N\}
$$ 
(with respect to any fixed total ordering) are linearly independent in  $\SY_N$.
\end{prop}

\begin{proof}
Let $\bar b_{ji;r}$, $\bar h_{i,r}$ be the images of $b_{ji;r}$, $h_{i,r}$ in the associated graded algebra $\gr\,\Y_N$, respectively. Then by \eqref{isogr}, we have 
\beq\label{helper1}
\begin{split}
&\bar b_{i+1,i;r}=\sqrt{-1}\big(e_{i+1,i}+(-1)^re_{i'-1,i'}\big)z^r,\\
&\bar h_{0,r}=\big(e_{11}+(-1)^re_{NN})\big)z^r,\\
&\bar h_{i,r}=\big(e_{i+1,i+1}-e_{ii}+(-1)^r(e_{i'-1,i'-1}-e_{i'i'})\big)z^r.    
\end{split}
\eeq
If $1\lle i<j\lle N$, we set $\alpha=\alpha_i+\cdots+\alpha_{j-1}$. Recall $\vartheta$ from \eqref{varth}. Then it follows from \eqref{bjir2} that
\beq\label{helper2}
\begin{split}
\bar b_{\alpha,r}\in (\sqrt{-1})^{j-i}\big(f_{\alpha}&+(-1)^r\vartheta(f_{\alpha})\big)z^r+\sum_{1\lle i< N} \bC\bar h_{i,r}\\
&+\sum_{\mu:\mathrm{ht}(\mu) <\mathrm{ht}(\alpha)}\C\big(f_{\mu}+(-1)^r\vartheta(f_{\mu})\big)z^r,
\end{split}
\eeq
where $\mathrm{ht}(\mu)$ denotes the height of the root $\mu$. It follows from \eqref{helper1}, \eqref{helper2} and the PBW theorem that the images of the ordered monomials of these elements  are linearly independent in the associated graded algebra $\gr\,\Y_N$, completing the proof.
\end{proof}

We will see soon that they actually form a basis for the corresponding algebras. 

\subsection{Explicit isomorphism}\label{sec:iso-GD}
In this subsection, we discuss the explicit isomorphism between the twisted Yangian $\Yi$ defined in Drinfeld type presentation and the (special) twisted Yangians constructed via the R-matrix presentation.

\begin{thm}\label{main2}
Let $N>1$ and $\Yi$ be the algebra defined in Definition~\ref{deftY} of type $\mathrm{AIII}_{N-1}$. Then there is an algebra isomorphism,
\beq\label{isoexp}
\begin{split}
\Phi: &\,\Yi\to \SY_{N},\\
&\, h_{i,r}\mapsto h_{i,r},\quad b_{i,r}\mapsto b_{i,r},
\end{split}
\eeq
where $h_{i,r},b_{i,r}$ on the right-hand side are defined in \eqref{beven}-\eqref{hodd} for $i\in \I^0$, $r\in\bN$.
\end{thm}
\begin{proof}
In the next section, we shall prove that the defining relations for $\Yi$ are satisfied in $\scrX_{N}$ for small rank cases $N=2,3,4,5$, and then the case for general $N$ follows by applying the shift homomorphism $\psi_m$ defined in Proposition~\ref{prop:red} and Lemma \ref{shiftlem}. Since $\SY_N$ is a subquotient of $\scrX_N$, the defining relations for $\Yi$ are also satisfied in $\SY_{N}$. Thus $\Phi$ is an algebra homomorphism. 

It follows by Definition \ref{defSY} that $\Phi$ is surjective. Therefore, it suffices to show that $\Phi$ is injective. For that, we only need to prove that a spanning set of $\Yi$ is sent to a set of linearly independent vectors of $\SY_{N}$. 

By Proposition \ref{prop:span}, the ordered monomials of the elements in the set \eqref{eq:pbw-elem} with the index sets given by \eqref{I+-A} in $\Yi$ form a spanning set of $\Yi$. It is clear from \eqref{bjir1} and \eqref{bjir2} that $\Phi$ sends $b_{\alpha,r}$ to $b_{\alpha,r}$ for $\alpha\in\mc R^+$ and $r\in\bN$. Thus, these ordered monomials are sent via $\Phi$ to the ordered monomials of the elements in the set \eqref{dd34} in $\SY_N$ which are linearly independent in $\SY_N$ by Proposition \ref{prop:pwb}.
\end{proof}

We can also obtain a Drinfeld type presentation for the twisted Yangian $\Y_N$. To that end, we use the following notation, which is inconsistent with the notation used before.  

Set $\tau(0)=N$, $\tau(N)=0$, $c_{0i}=-\delta_{1i}$ and $c_{Ni}=-\delta_{N-1,i}$ for $1\lle i<N$.

\begin{thm}\label{main3}
The twisted Yangian $\Y_N$ is isomorphic to the algebra generated by $h_{i,r},b_{j,r}$, $0\lle i\lle N$, $1\lle j<N$, $r\in\bN$, subject to the relations \eqref{qsconj0}--\eqref{qsconj10} together with the relation \eqref{prod} and $h_0(u)=h_{N}(-u)$. Moreover, the ordered monomials of 
\beq\label{ggen}
\{b_{\alpha,r},h_{0,2r},h_{i,r},h_{j,2r+1}~|~\alpha\in\cR^+,i\in\I^0_{\ne},j\in\I_{=}^0,r\in\N\}
\eeq
(with respect to any fixed total ordering) form a basis of $\Y_N$.
\end{thm}

\begin{proof} 
Let $\mathfrak Y^\imath$ be the algebra generated by $h_{i,r},b_{j,r}$, $0\lle i\lle N$, $1\lle j<N$, $r\in\bN$ subject to the relations \eqref{qsconj0}--\eqref{qsconj10}, \eqref{prod}, and $h_0(u)=h_{N}(-u)$. Similar to the proof of Theorem \ref{main2}, there is an algebra homomorphism from 
$$
\Xi:\mathfrak Y^\imath\to\Y_N,\quad h_{i,r}\mapsto h_{i,r},\quad b_{j,r}\to b_{j,r}.
$$
Indeed, in the next section, \eqref{qsconj0}--\eqref{qsconj10} hold in the extended twisted Yangian $\scrX_N$, and hence holds in $\Y_N$. By Proposition \ref{Binv} and Lemma \ref{efaiii}, the unitary relation \eqref{bunit} or \eqref{unimat} implies that $d_1(u)=\tl d_N(-u)$ in $\Y_N$, i.e., $h_0(u)=h_{N}(-u)$.

It remains to prove that $\Xi$ is an isomorphism. It is clear from Lemma \ref{lem:de-gen} and \eqref{fi=tauie} that $\Xi$ is surjective. Now we establish the injectivity. Note that by the relation \eqref{prod} and $h_0(u)=h_{N}(-u)$, we have
\[
h_0\big(u-\tfrac{N}4\big)h_0\big(\hskip-0.1cm-u-\tfrac{N}4\big)=\Big(1-\delta_{N,\mathrm{odd}}\frac{1}{16u^2}\Big)
\prod_{1\lle i<N} h_i\big(u-\tfrac{N-2i}{4}\big)^{-1}.
\]
Thus $h_{0,2s+1}$ for $s\in\bN$ can be expressed as polynomials in $h_{0,2r}$ and $h_{i,r}$ for $1\lle i<N$ and $r\in\bN$. With this observation, we can argue in the same way as Proposition \ref{prop:span} that the ordered monomials of elements in \eqref{ggen} span the algebra $\mathfrak Y^\imath$. By Proposition \ref{prop:pwb}, the $\Xi$-images of these ordered monomials are linearly independent and hence the ordered monomials of \eqref{ggen} form a basis for $\mathfrak Y^\imath$. This shows that $\Xi$ is injective.
\end{proof}

\subsection{Center of twisted Yangians}
As an application, we take the chance to discuss a set of algebraically independent generators of the center $\Y_N$ in terms of the generating series $d_i(u)$ for $1\lle i\lle N$. 

Define
\beq\label{cu}
\mathcal C(u)=\begin{cases}
d_1(u+n-\tfrac12)d_2(u+n-\tfrac32)\cdots d_N(u-n+\tfrac12),\qquad \text{ if }N=2n,\\
d_1(u+n)d_2(u+n-1)\cdots  d_N(u-n),\qquad \hskip 1.45cm \text{ if }N=2n+1.
\end{cases}
\eeq
Recall $c(u)$ from \eqref{cudef} and note that $c(u)=1$ in $\Y_N$; see \eqref{unimat} and Proposition \ref{Binv}. Then it follows from Lemma \ref{efaiii} that
\beq\label{cu-sym}
\mathcal C(u)\mathcal C(-u)=1.
\eeq
Define the elements $\mathcal C_i\in\Y_N$ by
\[
\mathcal C(u)=1+\sum_{i\gge 1}\mathcal C_iu^{-i}.
\]
Denote by $\mathscr{ZY}_N$ the center of the twisted Yangian $\Y_N$. 
\begin{thm}\label{thm:center}
We have the following statements.
\begin{enumerate}
    \item The coefficients of the elements $\mathcal C(u)$ are central in $\Y_N$.
    \item The elements $\mathcal C_{2r+1}$ for $r\in\bN$ are algebraic free generators of the center $\mathscr{ZY}_N$ of $\Y_N$.
    \item We have an algebra isomorphism $\Y_N\cong \mathscr{ZY}_N\otimes \SY_N$. Moreover, the center of $\SY_N$ is trivial.
    \item We have $\SY_N=\rY(\mathfrak{sl}_N)\cap \Y_N$.
\end{enumerate}
\end{thm}

In particular, (4) implies that our Definition~\ref{defSY} for the special twisted Yangian $\SY_N$ is equivalent to the one in \cite[Def. 2.9.1]{Mol07}.

\begin{proof}
(1) By Theorem \ref{main2} or Lemma \ref{lemnew1}, we have $[d_i(u),d_j(v)]=0$ for $1\lle i,j\lle N$. Thus it suffices, by Lemmas \ref{efaiii} and \ref{lem:de-gen}, to verify that 
$$
[\mathcal C(u),e_i(v)]=[\mathcal C(u),f_i(v)]=0,\qquad 1\lle i\lle n:=\lfloor \tfrac{N}{2}\rfloor.
$$ 
If $i<n$, then by Proposition \ref{propA} below and \cite[Thm. 8.6]{BK05} we have $d_1(u)\cdots d_n(u-n+1)$ commutes with $e_i(v)$ and $f_i(v)$ if $1\lle i<n$. Thus it follows from Lemma \ref{lemnew1} that $[\mathcal C(u),e_i(v)]=[\mathcal C(u),f_i(v)]=0$ for $1\lle i<n$. 

Hence it remains to verify that $[\mathcal C(u),e_n(v)]=[\mathcal C(u),f_n(v)]=0$. Again by Lemma \ref{lemnew1}, it reduces to verify the statement for the cases $N=2$ and $N=3$ which have been done in Lemmas \ref{centralb2}, \ref{centralb3}.

(2) It follows from \eqref{cu-sym} that $\mathcal C_{2r}$ can be expressed by $\mathcal C_i$ for $i<2r$. Thus it is easy to see by induction that all $\mathcal C_i$ can be expressed in terms of $\mathcal C_{2r+1}$ for $r\in\bN$. It suffices to prove the statement in the associated graded algebra $\gr\,\Y_N$. Let $\overline{\mc C}_i$ be the image of $\mc C_i$ in the $(i-1)$-st component of $\gr\,\Y_N$.

Recall from Section \ref{sec:basics} that $\gr\,\Y_N\cong \mathrm{U}(\gl_N[z]^\vartheta)$. By \eqref{isogr}, one easily sees that
\[
\overline{\mc C}_{2r+1}=2\mathcal I\,z^{2r},\qquad \mathcal I=e_{11}+\cdots+e_{NN}.
\]
To complete the proof, it suffices to show that the center of the algebra $\mathrm{U}(\gl_N[z]^\vartheta)$ is generated by the elements $\mathcal I\,z^{2r}$ with $r\in\bN$. Note that
\beq\label{grdec}
\mathrm{U}(\gl_N[z]^\vartheta)=\bC[\mathcal I\,z^{2r}]_{r\gge 0}\otimes \mathrm{U}(\mathfrak{sl}_N[z]^\vartheta).
\eeq
The rest is similar to the end of the proof of \cite[Thm. 3.4]{MR02}.

(3) As e.g. \cite[Lem. 3.11]{KLWZ23a}, one verifies that $\Y_N=\mathscr{ZY}_N\cdot \SY_N$. Thus it suffices to show that $\mathscr{ZY}_N\cap\SY_N=\{1\}$. Again, it reduces to the associated graded level which follows easily as the image of $\SY_N$ in $\gr\,\Y_N$ equals $\mathrm{U}(\mathfrak{sl}_N[z]^\vartheta)$; see also \eqref{grdec}.

(4) It is clear by the definition of $\rY(\mathfrak{sl}_N)$ and the Gauss decomposition that $\SY_N\subset \rY(\mathfrak{sl}_N)\cap \Y_N$. Thus, to show the equality, it suffices to note that we have the corresponding equality in $\gr\,\rY(\gl_N)$ (recall that the filtration on $\Y_N$ is induced from the one on $\rY(\gl_N)$).
\end{proof}

\begin{rem}
It would be interesting to investigate the relation between the Sklyanin determinant of $S(u)$ and the central series $\mathcal C(u)$; see \cite[Thm. 3.4]{MR02}.
\end{rem}

\section{Relations between Gaussian generators}\label{sec:lowrk}
In this section, we work out the relations of type A and the relations between Gaussian generators when $N=2,3,4,5$. For low rank situation, there will be two main cases, i.e. with a fixed point ($N=2,4$) for the the  Dynkin diagram automorphism $\tau$ and without a fixed point ($N=3,5$).
\subsection{Relations of type A}\label{secA}
Suppose $N\gge 2m \gge  4$. By \eqref{bcom}, there is a homomorphism from
\[
\rY(\gl_{m})\to \Y_N,\quad t_{ij}(u)\mapsto s_{ij}(u),\quad 1\lle i,j\lle m.
\]
Therefore, the relations among $d_i(u),e_{j}(v),f_{k}(w)$ for $1\lle i\lle m$ and $1\lle j,k<m$ are the same as  those in $\rY(\gl_m)$. These relations are given in \cite[Thm. 5.2 and its proof]{BK05} which we shall list below. It is not hard to see from the PBW theorem that this homomorphism is indeed an embedding. However, we do not need this fact.

Let $n=\lfloor\tfrac{N}2\rfloor$. For $1\lle i<n$, set
\[
e_i^\circ(u)=\sum_{r\gge 2}e_i^{(r)}u^{-r},\qquad f_i^\circ(u)=\sum_{r\gge 2}f_i^{(r)}u^{-r},\qquad \zeta_i(u)=\tl d_i(u)d_{i+1}(u).
\]

\begin{prop}\label{propA}
The following relations hold in $\scrX_N$, with the conditions on the indices $1\lle i,j<n$ and $1\lle k,l\lle n$,
\begin{align*}
&[d_k(u),d_k(v)]=0,\\
&[e_i(u),f_j(v)]=\delta_{ij}\frac{\zeta_i(u)-\zeta_i(v)}{u-v},\\
&[d_k(u),e_j(v)]=(\delta_{k,j+1}-\delta_{kj})\frac{d_k(u)(e_j(u)-e_j(v))}{u-v},\\
&[d_k(u),f_j(v)]=(\delta_{kj}-\delta_{k,j+1})\frac{(f_j(u)-f_j(v))d_k(u)}{u-v},\\
&[e_i(u),e_i(v)]=\frac{(e_i(u)-e_i(v))^2}{u-v},\\
&[f_i(u),f_i(v)]=-\frac{(f_i(u)-f_i(v))^2}{u-v},\\
&[e_i(u),e_j(v)]=[f_i(u),f_j(v)]=0,\qquad \text{ if }c_{ij}=0.
\end{align*}
Moreover, we have, for $1\lle i\lle n-2$,
\begin{align*}
u[e_i^\circ(u),e_{i+1}(v)]-v[e_i(u),e_{i+1}^\circ(v)]&=e_i(u)e_{i+1}(v),\\
u[f_i^\circ(u),f_{i+1}(v)]-v[f_i(u),f_{i+1}^\circ(v)]&=-f_{i+1}(v)f_{i}(u),
\end{align*}
and the Serre relations, for $1\lle i,j<n$ with $|i-j|=1$,
\begin{align*}
\big[e_i(u),[e_i(v),e_j(w)]\big]+\big[e_i(v),[e_i(u),e_j(w)]\big]&=0,\\
\big[f_i(u),[f_i(v),f_j(w)]\big]+\big[f_i(v),[f_i(u),f_j(w)]\big]&=0.
\end{align*}
\end{prop}

Recall Lemma \ref{efaiii} and the definition of $h_i(u),b_i(u)$ from \eqref{beven}--\eqref{hodd}. It is clear that the commutators relations \eqref{qsconj0}-\eqref{qsconj4} between the generating series $h_i(u),b_i(u)$ for $i$ far away from $\lfloor\tfrac{N}{2}\rfloor$ can be deduced from the relations listed in Proposition \ref{propA} (exactly as $\rY(\mathfrak{sl}_m)$); see also Lemma \ref{typeArel} and Lemma \ref{lemnew1}. Thus, we are left with verifying the relations for $i$ close to $\lfloor\tfrac{N}{2}\rfloor$. By Lemma \ref{shiftlem}, it suffices to do that for the case when $N$ is small, namely $N=2,3,4,5$. Note that the case $N=4,5$ is mainly for the Serre relations.

\subsection{Relations in $\scrX_2$}
It has been known since \cite{MR02} that this case is essentially the Olshanski's twisted Yangian for $\mathfrak o_2$ \cite{Ol92}.
\begin{lem}\label{b2ser}
We have the following relations in $\scrX_2$,
\begin{align}
[d_i(u),d_j(v)]&=0,\label{b2dd}\\
e(u)&=-f(-u),\label{b2e=f}\\
\tl d_1(u)d_2(u)&=\tl d_1(-u)d_2(-u),\label{b2d=d}\\
[d_1(u),f(v)]&=\frac{1}{u-v}(f(u)-f(v))d_1(u)+\frac{1}{u+v}d_1(u)(e(u)+f(v)),\label{b2d1f}\\
[d_2(u),f(v)]&=\frac{1}{u-v}(f(v)-f(u))d_2(u)-\frac{1}{u+v}d_2(u)(e(u)+f(v)),\label{b2d2f}\\
[f(u),f(v)]&=-\frac{1}{u-v}(f(u)-f(v))^2+\frac{1}{u+v}\big(\tl d_1(u)d_2(u)-\tl d_1(v)d_2(v)\big).\label{b2ff}
\end{align}
\end{lem}
\begin{proof}
Equations \eqref{b2dd}--\eqref{b2d=d} follow directly from Lemma \ref{efaiii} and Lemma \ref{lemnew1}.

\mybox{Equations \eqref{b2d1f}--\eqref{b2d2f}}. Setting $i=j=k=1$ and $l=2$ in \eqref{bcom} and using \eqref{b2dd}, we have
\beq\label{b2pf1}
(u^2-v^2)[d_1(u),e(v)]=(u+v)d_1(u)(e(v)-e(u))-(u-v)(e(v)+f(u))d_1(u).
\eeq
Thus \eqref{b2d1f} follows from \eqref{b2e=f} and \eqref{b2pf1}. By Lemma \ref{efaiii}, we have $d_2(u)=c(u)\tl d_1(-u)$. Note that $c(u)$ is central. Therefore \eqref{b2d2f} follow from \eqref{b2e=f} and \eqref{b2d1f}.

\mybox{Equation \eqref{b2ff}}. Setting $i=k=1$ and $j=l=2$, we have
\[
(u+v)[s_{12}(u),s_{12}(v)]=[s_{12}(u),s_{12}(v)]+s_{11}(u)s_{22}(v)-s_{11}(v)s_{22}(u).
\]
Rewriting it in terms of Gaussian generating series, we obtain
\beq\label{b2pf3}
\begin{split}
&\, d_1(u)e(u)d_1(v)e(v)-d_1(v)e(v)d_1(u)e(u)\\
=&\,\frac{1}{u+v}\Big(d_1(u)e(u)d_1(v)e(v)-d_1(v)e(v)d_1(u)e(u)+d_1(u)d_2(v)\\&\qquad \qquad +d_1(u)f(v)d_1(v)e(v)-d_1(v)d_2(u)-d_1(v)f(u)d_1(u)e(u)\Big)
\end{split}
\eeq
By \eqref{b2pf1}, we have
\beq\label{b2pf4}
e(v)d_1(u)=d_1(u)e(v)-\frac{1}{u-v}d_1(u)(e(v)-e(u))+\frac{1}{u+v}(e(v)+f(u))d_1(u).
\eeq
Using \eqref{b2pf4} to commute $e(u)d_1(v)$ and $e(v)d_1(u)$ in \eqref{b2pf3}, we find that the left-hand side of \eqref{b2pf3} is transformed to
\begin{align*}
&d_1(u)d_1(v)[e(u),e(v)]-\frac{1}{u-v}d_1(u)d_1(v)(e(u)-e(v))^2+\\ &\frac{1}{u+v}\Big(d_1(u)e(u)d_1(v)e(v)+d_1(u)f(v)d_1(v)e(v)-d_1(v)e(v)d_1(u)e(u)-d_1(v)f(u)d_1(u)e(u)\Big).
\end{align*}
Thus it follows from \eqref{b2pf3} that
\[
[e(u),e(v)]=\frac{1}{u-v}(e(u)-e(v))^2-\frac{1}{u+v}\big(\tl d_1(u)d_2(u)-\tl d_1(v)d_2(v)\big).
\]
By \eqref{b2e=f}, we obtain \eqref{b2ff}.
\end{proof}

It is convenient to use the following notation. We write
\beq\label{simdef}
A(u,v)\simeq B(u,v)
\eeq
if $A(u,v)$ and $B(u,v)$ have the same coefficients of $u^{-r-1}v^{-s-1}$ for $r,s\in\bN$. Later, we will also use the same notation for the case $r\in \bZ$ and $s\in\bN$. When the case $r\in\bZ$ is used, we will clarify it further.

Recall $b(u)=\sqrt{-1}f(u)$ and $h(u)=\tl d_1(u)d_2(u)$ from \eqref{beven} and \eqref{heven}, respectively. Here we drop the subscript $i$ as the rank is one. 

\begin{lem}\label{b2s}
We have the relations in $\scrX_2$ in terms of generating series,
\begin{align*}
&[h(u),h(v)]=0,\qquad h(u)=h(-u),\\
&[b(u),b(v)]=-\frac{1}{u-v}(b(u)-b(v))^2-\frac{1}{u+v}(h(u)-h(v)),\\\
&[h(u),b(v)]\simeq\frac{1}{u^2-v^2}\big((2v+1)h(u)b(v)+(2v-1)b(v)h(u)\big).
\end{align*}
\end{lem}
\begin{proof}
The proof is parallel to that of \cite[Lem. 4.2]{KLWZ23a}.
\end{proof}

\begin{prop}\label{X=2}
For $r,s\in\bN$, we have
\begin{align*}
[h_r,h_s] &=0,\\
[b_{r+1},b_s]-[b_r,b_{s+1}] &=b_rb_s+b_sb_r-2(-1)^rh_{r+s+1},\\
[h_{r+2},b_s]-[h_r,b_{s+2}] &=2(b_{s+1}h_r+h_rb_{s+1})+[h_r,b_s].
\end{align*}
\end{prop}

\begin{proof}
The proof is parallel to that of  \cite[Prop. 4.3]{KLWZ23a} by taking the coefficients of $u^{-r-1}v^{-s-1}$ for $r,s\in\bN$ from the relations (expanded in the region $|u|\gg |v|$) in Lemma \ref{b2s}. Note that we can take $r\in \bZ$ in the third relation as the power of $v$ in the terms we dropped are nonnegative; see the proof of Proposition \ref{X=3prop} below for more details.
\end{proof}
\begin{rem}
If we set $b(u)=f(u)$, then we have
\[
[b(u),b(v)]=-\frac{1}{u-v}(b(u)-b(v))^2+\frac{1}{u+v}(h(u)-h(v)).
\]
The purpose for the extra $\sqrt{-1}$ in \eqref{beven} and \eqref{bodd} is to change $+$ in the above equation to $-$ so that it will match with the formulas in split type \cite{KLWZ23a,KLWZ23b} and Definition \ref{deftY}.
\end{rem}

\begin{lem}\label{centralb2}
The coefficients of $d_1(u)d_2(u-1)$ are central elements in $\scrX_2$.
\end{lem}
\begin{proof}
Since $[d_i(u),d_j(v)]=0$ and $e(u)=-f(-u)$ by \eqref{b2dd} and \eqref{b2e=f}, respectively, it suffices to prove that $[d_1(u)d_2(u-1),f(v)]=0$. By ignoring the terms like $f(u)d_1(u)$ and $d_1(u)e(u)$ in \eqref{b2d1f} and \eqref{b2d2f} (as these series do not contribute if we expand them in the region $|u|\gg |v|$), we have
\begin{align*}
d_1(u)f(v)\simeq \frac{(u+v)(u-v-1)}{(u+v-1)(u-v)}f(v)d_1(u),\\
d_2(u)f(v)\simeq \frac{(u+v)(u-v+1)}{(u+v+1)(u-v)}f(v)d_2(u).
\end{align*}
Thus $d_1(u)d_2(u-1)f(v)\simeq f(v)d_1(u)d_2(u-1)$ and the statement follows.
\end{proof}

\subsection{Relations in $\scrX_4$}\label{secN=4}
Let us consider the case $N=4$. Thanks to Lemma \ref{shiftlem}, we have the shift homomorphism
\[
 \psi_1:  \scrX_2\to \scrX_4,\quad d_{i}(u)\to d_{i+1}(u),\quad e(u)\mapsto e_{2}(u),\quad f(u)\mapsto f_{2}(u).
\]
Thus, by Lemmas \ref{efaiii}, \ref{b2ser} and Proposition \ref{propA}, we immediately have the following relations
\begin{align}
&[d_i(u),d_j(v)]=[d_1(u),e_2(v)]=[d_1(u),f_2(v)]=0,\label{ddb4}\\
&(u-v)[e_1(u),f_1(v)]=\tl d_1(u)d_2(u)-\tl d_1(v)d_2(v),\label{e1f1b4}\\
&(u-v)[e_1(u),e_1(v)]=(e_1(u)-e_1(v))^2,\label{e1e1b4}\\
&(u-v)[d_1(u),e_1(v)]=d_1(u)(e_1(v)-e_1(u)),\label{d1e1b4}\\
&(u-v)[d_2(u),e_1(v)]=d_2(u)(e_1(u)-e_1(v)),\label{d2e1b4}\\
&(u-v)[d_1(u),f_1(v)]=(f_1(u)-f_1(v))d_1(u),\label{d1f1b4}\\
&[e_2(u),e_2(v)]=\frac{1}{u-v}\big(e_2(u)-e_2(v)\big)^2-\frac{1}{u+v}\big(\tl d_2(u)d_3(u)-\tl d_2(v)d_3(v)\big).\label{e2e2b4}
\end{align}

\begin{lem}\label{be1e2lem}
We have
\begin{align}
(u-v)[e_1(u),e_2(v)]&=e_1(u)e_2(v)-e_1(v)e_2(v)-e_{13}(u)+e_{13}(v),\label{be1e2}\\
(u-v)[f_1(u),f_2(v)]&=f_2(v)f_1(v)-f_2(v)f_1(u)+f_{31}(u)-f_{31}(v).\label{bf1f2}
\end{align}
\end{lem}
\begin{proof}
By Lemma \ref{efaiii}, it suffices to show \eqref{be1e2}. By \eqref{bcom}, we have 
\[
(u-v)[s_{12}(u),s_{23}(v)]=s_{22}(u)s_{13}(v)-s_{22}(v)s_{13}(u).
\]
In terms of Gaussian generators, we have
\beq\label{b4pf1}
\begin{split}
(u-v)&[d_1(u)e_1(u),d_2(v)e_2(v)+f_1(v)d_1(v)e_{13}(v)]\\=\, &\big(d_2(u)+f_1(u)d_1(u)e_1(u)\big)d_1(v)e_{13}(v)-\big(d_2(v)+f_1(v)d_1(v)e_1(v)\big)d_1(u)e_{13}(u).
\end{split}
\eeq
We shall transform the left-hand side of \eqref{b4pf1}. Expanding the commutator, the left-hand side of \eqref{b4pf1} becomes
\begin{align*}
(u-v)\Big(d_1(u)& e_1(u)d_2(v) e_2(v)-d_2(v) e_2(v)d_1(u) e_1(u)\\
+&\, d_1(u) e_1(u)f_1(v) d_1(v)e_{13}(v)-f_1(v)d_1(v)e_{13}(v)d_1(u)e_1(u)\Big).
\end{align*}
Permuting the products $e_1(u)d_2(v)$, $e_2(v)d_1(u)$, and $e_1(u)f_1(v)$ using \eqref{d2e1b4}, \eqref{ddb4}, \eqref{e1f1b4}, respectively,  we have
\begin{align*}
&\,d_1(u) \big((u-v)d_2(v)e_1(u)-d_2(v)(e_1(u)-e_1(v))\big) e_2(v)-(u-v)d_2(v) d_1(u)e_2(v) e_1(u)\\
+&\,d_1(u) \big((u-v)f_1(v)e_1(u)+\tl d_1(u)d_2(u)-\tl d_1(v)d_2(v)\big) d_1(v)e_{13}(v)\\
&\hskip7.9cm-(u-v)f_1(v)d_1(v)e_{13}(v)d_1(u)e_1(u),
\end{align*}
which simplifies further to
\begin{align*}
d_1(u)d_2(v)\big((u-v)[e_1(u),e_2(v)]-(e_1(u)-e_1(v))e_2(v)-e_{13}(v)\big)+d_1(v)d_2(u)e_{13}(v)\\
+(u-v)\big(d_1(u)f_1(v)e_1(u)d_1(v)e_{13}(v)-f_1(v)d_1(v)e_{13}(v)d_1(u)e_1(u)\big).
\end{align*}
Therefore, it suffices to show that
\begin{align*}
(u-v)\big(& d_1(u)f_1(v)e_1(u)d_1(v)e_{13}(v)-f_1(v)d_1(v)e_{13}(v)d_1(u)e_1(u)\big)\\
=&\,f_1(u)d_1(u)e_1(u) d_1(v)e_{13}(v)- f_1(v)d_1(v)e_1(v) d_1(u)e_{13}(u).
\end{align*}
Applying \eqref{d1f1b4}, i.e. $(u-v)d_1(u)f_1(v)-f_1(u)d_1(u)=(u-v)f_1(v)d_1(u)-f_1(v)d_1(u)$, it reduces to show
\beq\label{b4pf2}
(u-v)[d_1(u)e_1(u),d_1(v)e_{13}(v)]=d_1(u)e_1(u)d_1(v)e_{13}(v)-d_1(v)e_1(v)d_1(u)e_{13}(u),
\eeq
which is equivalent to $(u-v)[s_{12}(u),s_{13}(v)]=s_{12}(u)s_{13}(v)-s_{12}(v)s_{13}(u)$ and follows from \eqref{bcom}.
\end{proof}

\begin{lem}\label{lemhelp1}
We have
\beq\label{e1e13}
[e_1(u),e_{13}(v)-e_1(v)e_2(v)]=-[e_1(u),e_2(v)]e_1(u).
\eeq
\end{lem}
\begin{proof}
It follows from \eqref{bcom} that
\[
(u-v)[s_{11}(u),s_{13}(v)]=s_{11}(u)s_{13}(v)-s_{11}(v)s_{13}(u),
\]
which implies that
\beq\label{b4pf3}
(u-v)[d_1(u),e_{13}(v)]=d_1(u)(e_{13}(v)-e_{13}(u)).
\eeq
We have
\beq\label{b4pf4}
\begin{split}
d_1(u)e_1(u)d_1(v)e_{13}(v)-&\, d_1(v)e_1(v)d_1(u)e_{13}(u)\\
\stackrel{\eqref{d1e1b4}}{=}&\, d_1(u)\big(d_1(v)e_1(u)+\tfrac{1}{u-v}d_1(v)(e_1(u)-e_1(v))\big)e_{13}(v)\\
-&\,d_1(v)\big(d_1(u)e_1(v)+\tfrac{1}{u-v}d_1(u)(e_1(u)-e_1(v))\big)e_{13}(u).
\end{split}
\eeq
On the other hand, we also have
\beq\label{b4pf5}
\begin{split}
d_1(u)e_1(u)& d_1(v)e_{13}(v)- d_1(v)e_1(v)d_1(u)e_{13}(u)\\
&\, \stackrel{\eqref{b4pf2}}{=}(u-v)\big(d_1(u)e_1(u)d_1(v)e_{13}(v)-d_1(v)e_{13}(v)d_1(u)e_{1}(u)\big)\\
&\, \overset{\eqref{d1e1b4}}{\underset{\eqref{b4pf3}}{=}}d_1(u)\big((u-v)d_1(v)e_1(u)+d_1(v)(e_1(u)-e_1(v))\big)e_{13}(v)\\
& \qquad\qquad \, -d_1(v)\big((u-v)d_1(u)e_{13}(v)-d_1(u)(e_{13}(v)-e_{13}(u))\big)e_1(u).
\end{split}
\eeq
Combining \eqref{b4pf4} and \eqref{b4pf5}, we obtain
\beq\label{b4pf6}
\begin{split}
(u-v)[e_1(u),e_{13}(v)]=&\,\frac{1}{u-v}\big(e_1(u)-e_1(v)\big)\big(e_{13}(v)-e_{13}(u)\big)\\&+\,e_1(v)\big(e_{13}(v)-e_{13}(u)\big)-\big(e_{13}(v)-e_{13}(u)\big)e_1(u).
\end{split}
\eeq
To prove \eqref{e1e13}, it suffices to show
\begin{align*}
(u-v)&\,[e_1(u),e_{13}(v)]\\
=~~&(u-v)\big([e_1(u),e_1(v)]e_2(v)+e_1(v)[e_1(u),e_2(v)]-[e_1(u),e_2(v)]e_1(u)\big)\\
\overset{\eqref{e1e1b4}}{\underset{\eqref{be1e2}}{=}}&\, e_1(u)[e_1(u),e_2(v)]-[e_1(u),e_1(v)e_2(v)]\\
&+e_1(v)\big(e_{13}(v)-e_{13}(u)\big)-\big(e_{13}(v)-e_{13}(u)\big)e_1(u).
\end{align*}
Therefore, by \eqref{b4pf6}, it reduces to show
\beq\label{b4pf7}
e_1(u)[e_1(u),e_2(v)]-[e_1(u),e_1(v)e_2(v)]=\frac{1}{u-v}\big(e_1(u)-e_1(v)\big)\big(e_{13}(v)-e_{13}(u)\big).
\eeq
Rewrite $(u-v)\big(e_1(u)[e_1(u),e_2(v)]-[e_1(u),e_1(v)e_2(v)]\big)$ as
\[
(u-v)\big(e_1(u)[e_1(u),e_2(v)]-[e_1(u),e_1(v)]e_2(v)-e_1(v)[e_1(u),e_2(v)]\big),
\]
then \eqref{b4pf7} follows from \eqref{e1e1b4} and \eqref{be1e2}.
\end{proof}

\begin{lem}\label{sec:A-serre}
We have the following Serre relation,
\[
\big[e_1(u),[e_1(v),e_2(w)]\big]+\big[e_1(v),[e_1(u),e_2(w)]\big]=0.
\]
\end{lem}
\begin{proof}
The proof is parallel to that of \cite[Lem. 5.7 (ii)]{BK05} as the commutator relations used in the proofs are the same, see \eqref{e1e1b4}, \eqref{be1e2}, \eqref{e1e13}; cf. \cite[Lem. 5.5]{BK05}. Note that again we can prove that $\big[e_1(u),[e_1(u),e_2(v)]\big]=0$.
\end{proof}


Next, we discuss the Serre relation involving
\[
\big[e_2(u),[e_2(v),e_1(w)]\big]+\big[e_2(v),[e_2(u),e_1(w)]\big].
\]
To this end, we first need the following lemmas. 
\begin{lem}\label{help1}
We have
\begin{align}
&[e_{13}(u),e_2^{(1)}]=e_1(u),\label{b4pf11}\\
&[e_1^{(1)},\tl d_2(v)]=-e_1(v)\tl d_2(v),\quad [e_1^{(1)}, d_3(v)]=f_3(v) d_3(v),\label{b4pf10}\\
&[e_1(u),e_2^{(1)}]=e_{13}(u),\quad [e_1^{(1)},e_2(v)]=e_{13}(v)-e_1(v)e_2(v),\label{b4pf8}\\
&[e_2^{(1)},e_2(v)]=e_2(v)e_2(v)-1+\tl d_2(v)d_3(v).    \label{b4pf9}
\end{align}
\end{lem}
\begin{proof}
By Lemma \ref{efaiii}, we have $d_{3}(v)=c(u)\tl d_{2}(-v)$ and $f_3(v)=-e_1(-v)$, the relation \eqref{b4pf10} is straightforward from \eqref{d2e1b4} while the relations \eqref{b4pf8}, \eqref{b4pf9} follow from \eqref{be1e2}, \eqref{e2e2b4}, respectively. 

Then we prove \eqref{b4pf11}. By \eqref{bcom}, we have
\[
(u-v)[s_{13}(u),s_{32}(v)]=s_{33}(u)s_{12}(v)-s_{33}(v)s_{12}(u)
\]
which implies $[s_{13}(u),s_{32}^{(1)}]=s_{12}(u)$. In terms of Gaussian generators, it transforms to
\[
[d_1(u)e_{13}(u),f_{2}^{(1)}]=d_1(u)e_1(u).
\]
By \eqref{ddb4}, we have $[e_{13}(u),f_{2}^{(1)}]=e_1(u)$. Then \eqref{b4pf11} follows from $f_{2}^{(1)}=e_{2}^{(1)}$ as $e_2(u)=-f_2(-u)$ by Lemma \ref{efaiii}.
\end{proof}

\begin{lem}\label{be13e2lem}
We have
\[
[e_{13}(u),e_2(v)]=e_2(v)[e_1(u),e_2(v)]+\frac{1}{u+v}\tl d_2(v)\big(f_3(v)+e_1(u)\big)d_3(v).
\]
\end{lem}
\begin{proof}
By Lemma \ref{help1}, we have
\begin{align*}
[e_{13}^{(1)},e_2(v)]=&\,\big[[e_1^{(1)},e_2^{(1)}],e_2(v)\big]=\big[[e_1^{(1)},e_2(v)],e_2^{(1)}\big]+[e_1^{(1)},\big[e_2^{(1)},e_2(v)]\big]\\
=&\,[e_{13}(v)-e_1(v)e_2(v),e_2^{(1)}]+[e_1^{(1)},e_2(v)e_2(v)+\tl d_2(v)d_3(v)]\\
=&\, e_1(v)-e_{13}(v)e_2(v)+e_1(v)\big(e_2(v)e_2(v)-1+\tl d_2(v)d_3(v)\big)\\
&\, +\big(e_{13}(v)-e_1(v)e_2(v)\big)e_2(v)+e_2(v)[e_1^{(1)},e_2(v)]+[e_1^{(1)},\tl d_2(v)d_3(v)]\\
=&\, e_2(v)[e_1^{(1)},e_2(v)]+ e_1(v) \tl d_2(v)d_3(v)+[e_1^{(1)},\tl d_2(v)]d_3(v)+\tl d_2(v)[e_1^{(1)},d_3(v)]\\
=&\, e_2(v)[e_1^{(1)},e_2(v)] + \tl d_2(v)f_3(v)d_3(v).
\end{align*}
Finally, by \eqref{d1e1b4} and \eqref{b4pf3}, we have
\beq\label{b4pf12}
[d_1(u),e_1^{(1)}]=d_1(u)e_1(u),\qquad [d_1(u),e_{13}^{(1)}]=d_1(u)e_{13}(u).
\eeq
Using \eqref{ddb4}, \eqref{d1e1b4}, \eqref{b4pf12} and $f_3(v)=-e_1(-v)$ (by Lemma \ref{efaiii}), the lemma follows from the equality above by taking the commutator relation with $d_1(u)$.
\end{proof}

\begin{lem}\label{be1e2e2lem}
We have
\[
\big[[e_1(u),e_2(v)],e_2(v)\big]=\frac{\tl d_2(v)}{u-v-1}\Big(\frac{f_3(v)+e_1(v+1)}{2v+1}-\frac{f_3(v)+e_1(u)}{u+v}\Big)d_3(v).
\]
\end{lem}
\begin{proof}
The proof is similar to that of \cite[Lemma 4.8]{KLWZ23a} using Lemmas \ref{be1e2lem} and \ref{be13e2lem}.
\end{proof}

It is convenient to set
\begin{align*}
&\beta_1(u,v)=\frac{1}{u+v}\tl d_2(v)\big(f_3(v)+e_1(u)\big)d_3(v),\\
&\beta_2(u,v)=-\frac{1}{u+v}\big(\tl d_2(u)d_3(u)-\tl d_2(v)d_3(v)\big),\\
&\beta_3(u,v)=\frac{\tl d_2(v)}{u-v-1}\Big(\frac{f_3(v)+e_1(v+1)}{2v+1}-\frac{f_3(v)+e_1(u)}{u+v}\Big)d_3(v).
\end{align*}
Then it follows from Lemma \ref{be13e2lem}, \eqref{e2e2b4}, and Lemma \ref{be1e2e2lem}, respectively, that
\begin{align}
[e_{13}(u),e_2(v)]&=e_2(v)[e_1(u),e_2(v)]+\beta_1(u,v),\label{beta1b}\\
[e_2(u),e_2(v)]&=\frac{1}{u-v}(e_2(u)-e(v))^2+\beta_2(u,v),\label{beta2b}\\
\big[[e_1(u),e_2(v)],e_2(v)\big]&=\beta_3(u,v).\label{beta3b}
\end{align}
Comparing these relations with their counterparts in \cite[\S5]{BK05}, one finds that the series $\beta_i(u,v)$ are extra terms appearing in the relations of twisted Yangians. By the same strategy of \cite[Lem. 4.9]{KLWZ23a}, we obtain the following.

\begin{lem}\label{serreai}
We have
\[
\begin{split}
\big[[e_1(u),e_2(v)],e_2(w)\big]+\{v\leftrightarrow w\}= \frac{1}{u-v}\Big( \beta_1(v,w)-&\,\beta_1(u,w)+(e_1(u)-e_1(v))\beta_2(v,w)\\
&+\beta_3(u,w)-\beta_3(v,w)\Big)+\{v\leftrightarrow w\}.
\end{split}
\]
\end{lem} 
Here $\{v\leftrightarrow w\}$ denotes a summand obtained from the prior one with $v, w$ switched.

Recall $b_i(u)$ and $h_i(u)$ from \eqref{beven} and \eqref{heven}, respectively.

\begin{prop}
We have the following Serre relations 
\begin{align*}
&\big[b_1(u),[b_1(v),b_2(w)]\big]+\big[b_1(v),[b_1(u),b_2(w)]\big]=0,\\
&\big[b_3(u),[b_3(v),b_2(w)]\big]+\big[b_3(v),[b_3(u),b_2(w)]\big]=0,\\
\Sym_{k_1,k_2}&\big[b_{2,k_1},[b_{2,k_2},b_{j,r}]\big]=(-1)^{k_1}\sum_{p\gge 0}2^{-2p}\big([h_{2,k_1+k_2-2p-1},b_{j,r+1}]-\{h_{2,k_1+k_2-2p-1},b_{j,r}\}\big),
\end{align*}
where $j=1,3$.
\end{prop}
\begin{proof}
The first equality follows from Lemma \ref{sec:A-serre} by applying the anti-automorphism $\eta$ (recall $\eta$ is defined in Lemma~\ref{lem:eta}), while the second one follows from Lemma \ref{efaiii} and Lemma \ref{sec:A-serre}. The proof for the third equality is parallel to the counterpart in \cite[Thm. 4.10]{KLWZ23a} by using Lemma \ref{serreai} and the anti-automorphism $\eta$.
\end{proof}

We have established the desired Serre relations in $\scrX_4$. Let us consider the other relations.

\begin{prop}\label{X=4prop}
For $i,j\in\{1,2,3\}$ and $r,s\in\bN$, we have 
\begin{align}
&[h_{i,r},h_{j,s}]=0,\qquad h_{i,r}=(-1)^{r+1}h_{\tau i,r},\label{hhb4}\\
&[\X_{1,r},\X_{3,s}]=  (-1)^{r}h_{3,r+s}, \qquad [\X_{3,r},\X_{1,s}]=  (-1)^{r}h_{1,r+s},   
\label{bbb4-}
\\
&[b_{i,r+1},b_{j,s}]  - [b_{i,r },b_{j,s+1 }]  =\frac{c_{ij}}{2}\{b_{i,r },b_{j,s}\}-2\delta_{i,\tau j} (-1)^{r}h_{\tau i,r+s+1 },\label{bbb4}\\
&[h_{i,r+2},b_{j,s}] - [h_{i,r},b_{j,s+2}]=\frac{c_{ij}-c_{i,\tau j}}{2}\{h_{i,r+1},b_{j,s}\}\notag\\
&\qquad\qquad\qquad \qquad \qquad \ \  +\frac{c_{ij}+c_{i,\tau j}}{2} \{h_{i,r},b_{j,s+1}\}+\frac{c_{ij}c_{\tau i,j}}{4} [h_{i,r}, b_{j,s}].\label{hbb4}
\end{align}
Here we allow $r\in\bZ$ in the relation \eqref{hbb4} with $h_{i,-1}=1$ and $h_{i,r}=0$ for $r<-1$.
\end{prop}
\begin{proof}
We shall frequently use \eqref{i=tauih}--\eqref{etahb} obtained from Lemma \ref{efaiii}. 
The relation \eqref{hhb4} follows from Lemma \ref{lemnew1}. The relation \eqref{bbb4-} is immediate from \eqref{e1f1b4} and Lemma \ref{efaiii}. The relation \eqref{bbb4-} clearly implies the relation \eqref{bbb4} for $(i,j)=(1,3),(3,1)$. The relation \eqref{bbb4} for $(i,j)=(1,2),(2,3)$ is a corollary of Lemma \ref{be1e2lem} with \eqref{fi=tauie}. The relation \eqref{bbb4} for $(i,j)=(3,3)$ follows from \eqref{e1e1b4} while the relation \eqref{bbb4} for $(i,j)=(2,2)$ follows by applying the shift homomorphism $\psi_1$ to the corresponding rank 1 relation in Proposition \ref{X=2}. Applying the anti-automorphism $\eta$, we obtain the relation \eqref{bbb4} for $(i,j)=(1,1)$ as $\eta$ sends $b_1(u)$ to $-b_3(-u)$ by \eqref{etahb}.

By the anti-automorphism $\eta$, $h_1(u)=h_3(-u)$ and \eqref{etahb} (or Lemma \ref{efaiii}), it suffices to consider the relation \eqref{hbb4} for $(i,j)$ such that $i,j\lle 2$. The relation \eqref{hbb4} for $(i,j)=(1,1)$ is the same as  Yangian of type A as observed in Section \ref{secA}, while this relation for $(i,j)=(2,2)$ follows from Proposition \ref{X=2} by applying the shift homomorphism $\psi_1$ to the corresponding relation in $\scrX_2$. 
 
We show the relation \eqref{hbb4} for $(i,j)=(1,2)$. Applying the shift homomorphism $\psi_1$ to the relation \eqref{b2d1f} in $\scrX_2$, we obtain
\[
[d_2(u),f_2(v)]=\frac{1}{u-v}\big(f_2(u)-f_2(v)\big)d_2(u)+\frac{1}{u+v}d_2(u)\big(e_2(u)+f_2(v)\big).
\]
Note that $[d_1(u),f_2(v)]=0$ and then $[\tl{d}_1(u),f_2(v)]=0$. By definition, $h_1(u)=\tl{d}_1(u+\frac{1}{2}) d_2(u+\frac{1}{2})$ and $b_2(v)=\sqrt{-1}f_2(v)$. Then we have
\begin{align*}
\big(u-v+\tfrac{1}{2}\big)&\big(u+v+\tfrac{1}{2}\big) [h_1(u),b_2(v)]\\
=\, &\big(u+v+\tfrac{1}{2}\big) \big(b_2(u)-b_2(v)\big)h_1(u) + \big(u-v+\tfrac{1}{2}\big)h_1(u)\big(\sqrt{-1}e_2(u)+b_2(v)\big).
\end{align*}
The desired statement follows by taking the coefficients of $u^{-r-1}v^{-s-1}$ for $r\in\Z,s\in \bN$ in this identity. Here the terms $b_2(u)h_1(u)$ and $h_1(u)e_2(u)$ in the RHS do not contribute to the coefficients of $u^{-r-1}v^{-s-1}$ for $r\in\Z,s\in\bN$. 

Finally, it remains to consider the relation \eqref{hbb4} for $(i,j)=(2,1)$. Applying the anti-automorphism $\eta$ to \eqref{d2e1b4}, we have
\[
[\tl d_2(u),f_1(v)]=\frac{1}{u-v}\tl d_2(u)\big(f_1(u)-f_1(v)\big).
\]
By Lemma \ref{efaiii}, we further have
\beq\label{mm1}
[d_3(u),f_1(v)]=\frac{1}{u+v}d_3(u)\big(e_2(u)+f_1(v)\big).
\eeq
Therefore, we find
\beq\label{mm3}
\begin{split}
(u^2-v^2)&[\tl d_2(u)d_3(u),f_1(v)]\\
&=(u+v)\tl d_2(u)(f_1(u)-f_1(v))d_3(u)+(u-v)\tl d_2(u)d_3(u)(e_3(u)+f_1(v))\\
&\simeq -(u+v)\tl d_2(u)f_1(v)d_3(u)+(u-v)\tl d_2(u)d_3(u)f_1(v).
\end{split}
\eeq
By \eqref{mm1}, we have 
\beq\label{mm2}
(u+v)f_1(v)d_3(u)\simeq (u+v-1)d_3(u)f_1(v).
\eeq
Plugging \eqref{mm2} into \eqref{mm3}, we obtain
\[
(u^2-v^2)\big[\tl d_2(u)d_3(u),f_1(v)\big]\simeq (-2v+1)\tl d_2(u)d_3(u)f_1(v).
\]
By definition, we have $h_2(u)=\tl d_2(u)d_3(u)$ and $b_1(v)=\sqrt{-1}f_1(v+\frac{1}{2})$, and then the relation \eqref{hbb4} for $(i,j)=(2,1)$ follows by taking the components of this relation.
\end{proof}

\begin{rem}
Similarly, we can also compute the relations involving $h_0(u)$ or $h_4(u)$ that are required to complete the proof of Theorem \ref{main3}. These relations are much simpler to obtain as $h_0(u)$ is $d_1(u+1)$. So we omit the details.
\end{rem}

\subsection{Relations in $\scrX_3$}

We start with listing relations between Gaussian generators $d_1(u)$, $d_2(u)$, $e_1(u)$, $f_1(u)$; cf. Lemma \ref{efaiii}.
\begin{lem}
We have 
\begin{align}
[d_i(u),d_j(v)]&=0,\label{ddb3}\\
[d_1(u),e_1(v)]&=\frac{1}{u-v}d_1(u)(e_1(v)-e_1(u)),\label{d1e1b3}\\
[d_1(u),f_1(v)]&=\frac{1}{u-v}(f_1(u)-f_1(v))d_1(u),\label{d1f1b3}\\
[e_1(u),e_1(v)]&=\frac{1}{u-v}(e_1(u)-e_1(v))^2,\label{e1e1b3}\\
[f_1(u),f_1(v)]&=-\frac{1}{u-v}(f_1(u)-f_1(v))^2,\label{f1f1b3}\\
[e_1(u),f_1(v)]&=\frac{1}{u-v}\big(\tl d_1(u)d_2(u)-\tl d_1(v)d_2(v)\big)+\frac{1}{u+v}\big(e_{13}(u)+e_1(u)f_1(v)+f_{31}(v)\big),\label{e1f1b3}\\
[d_2(u),e_1(v)]&=\frac{1}{u-v}d_2(u)(e_1(u)-e_1(v))-\frac{1}{u+v}(e_1(v)+f_2(u))d_2(u),\label{d2e1b3}\\
[d_2(u),f_1(v)]&=\frac{1}{u-v}(f_1(v)-f_1(u))d_2(u)+\frac{1}{u+v}d_2(u)(f_1(v)+e_2(u)).\label{d2f1b3}
\end{align}
\end{lem}
\begin{proof}
We verify the essential relations \eqref{ddb3}, \eqref{e1e1b3}, \eqref{e1f1b3}, \eqref{d2e1b3} as the other relations follow from the essential ones by taking the anti-automorphism $\eta$; see Lemma \ref{lem:eta}.

The relation \eqref{ddb3} follows from Lemma \ref{lemnew1}. Then \eqref{d1e1b3} follows from \eqref{bcom} with $i=j=k=1$, $l=2$ and $[d_1(u),d_1(v)]=0$. Applying the anti-automorphism $\eta$ to \eqref{d1e1b3}, we obtain \eqref{d1f1b3}.

\mybox{Equation \eqref{e1e1b3}}. By \eqref{bcom} with $i=k=1$ and $j=l=2$, we have $[s_{12}(u),s_{12}(v)]=0$ which implies that
\beq\label{b3pf2}
d_1(u) e_1(u)d_1(v) e_1(v)=d_1(v) e_1(v)d_1(u) e_1(u).
\eeq
Using \eqref{d1e1b3} to commute $d_1(u)$ and $e_1(v)$, we have
\begin{align*}
d_1(u)\Big(d_1(v)e_1(u)+&\frac{1}{u-v}d_1(v)(e_1(u)-e_1(v))\Big)e_1(v)\\
&=d_1(v)\Big(d_1(u)e_1(v)+\frac{1}{u-v}d_1(u)(e_1(u)-e_1(v))\Big)e_1(u).
\end{align*}
Canceling $d_1(u)d_1(v)$, we obtain \eqref{e1e1b3}.

\mybox{Equation \eqref{e1f1b3}}. We first claim that
\beq\label{b3pf1}
\begin{split}
(u-v)\big([d_1(u)e_1(u),f_1(v)d_1(v)]&-d_1(u)[e_1(u),f_1(v)]d_1(v)\big)\\
&=f_1(u)d_1(u)e_1(u)d_1(v)-f_1(v)d_1(v)e_1(v)d_1(u).
\end{split}
\eeq
This is equivalent to
\begin{align*}
\big((u-v)d_1(u)f_1(v)&-f_1(u)d_1(u)\big)e_1(u)d_1(v)\\
&=f_1(v)d_1(v)\big((u-v)d_1(u)e_1(u)-e_1(v)d_1(u)\big).
\end{align*}
Note that by \eqref{d1f1b3}, we have
\[
(u-v)d_1(u)f_1(v)-f_1(u)d_1(u)=f_1(v)\big((u-v)d_1(u)-d_1(u)\big).
\]
Hence, to prove \eqref{b3pf1}, it reduces to show
\[
\big((u-v)d_1(u)-d_1(u)\big)e_1(u)d_1(v)=d_1(v)\big((u-v)d_1(u)e_1(u)-e_1(v)d_1(u)\big),
\]
that is
\[
(u-v)d_1(u)[d_1(v),e_1(u)]=d_1(v)e_1(v)d_1(u)-d_1(u)e_1(u)d_1(v).
\]
Applying \eqref{d1e1b3} to the left-hand side, it transforms to
\[
d_1(v)[d_1(u),e_1(v)]=d_1(u)[d_1(v),e_1(u)]
\]
which follows directly from \eqref{d1e1b3} by applying it to both sides.

Let us come back to \eqref{e1f1b3}. By \eqref{bcom} with $i=l=1$ and $j=k=2$ in terms of Gaussian generators, we have
\begin{align*}
(u^2-v^2)&[d_1(u)e_1(u),f_1(v)d_1(v)]\\
=\, &(u+v)\big(d_2(u)d_1(v)+f_1(u)d_1(u)e_1(u)d_1(v)-d_2(v)d_1(u)-f_1(v)d_1(v)e_1(v)d_1(u)\big)\\
&\qquad \qquad +(u-v)\big(d_1(u)e_{13}(u)d_1(v)+d_1(u)e_1(u)f_1(v)d_1(v)+d_1(u)f_{31}(v)d_1(v)\big).
\end{align*}
Now using \eqref{b3pf1} for $[d_1(u)e_1(u),f_1(v)d_1(v)]$ and multiplying $\tl d_1(u)$, $\tl d_1(v)$ from the left and the right, respectively, one finds \eqref{e1f1b3}. 

\mybox{Equation \eqref{d2e1b3}}. Taking the coefficients of $u$ in \eqref{bcom} with $i=1,j=k=l=2$ in terms of Gaussian generators, we find that
\beq\label{b3pf3}
[e_1^{(1)},d_2(v)+f_1(v)d_1(v)e_1(v)]=d_1(v)e_1(v)+f_2(v)d_2(v)+f_{31}(v)d_1(v)e_1(v).
\eeq
It follows from \eqref{b3pf2} that 
\beq\label{b3pf4}
[e_1^{(1)},d_1(v)e_1(v)]=0.
\eeq
Note also that, by \eqref{e1f1b3}, we have
\beq\label{b3pf5}
[e_1^{(1)},f_1(v)]=1-\tl d_1(v)d_2(v)+f_{31}(v).
\eeq
Combining \eqref{b3pf3}, \eqref{b3pf4}, and \eqref{b3pf5}, we conclude that
\beq\label{b3pf6}
[e_1^{(1)},d_2(v)]=d_2(v)e_1(v)+f_2(v)d_2(v).
\eeq
On the other hand, we have $[d_1(u),e_1^{(1)}]=d_1(u)e_1(u)$ by \eqref{d1e1b3} and $f_2(v)=-e_1(-v)$ by Lemma~ \ref{efaiii}. Therefore, we find that
\begin{align*}
d_1(u)[e_1(u),d_2(v)]&\stackrel{\eqref{ddb3}}{=}[d_1(u)e_1(u),d_2(v)]=[d_1(u),[e_1^{(1)},d_2(v)]]\\
&\stackrel{\eqref{b3pf6}}{=}[d_1(u),d_2(v)e_1(v)+f_2(v)d_2(v)]\\
&\stackrel{\eqref{ddb3}}{=}d_2(v)[d_1(u),e_1(v)]+[d_1(u),f_2(v)]d_2(v)\\
&\stackrel{\eqref{d1e1b3}}{=}\frac{1}{u-v}d_2(v)d_1(u)\big(e_1(v)-e_1(u)\big)+\frac{1}{u+v}d_1(u)\big(e_1(u)+f_2(v)\big)d_2(v),
\end{align*}
completing the proof of \eqref{d2e1b3}.
\end{proof}

\begin{lem}\label{centralb3}
The coefficients of $d_1(u+1)d_2(u)d_3(u-1)$ are central elements in $\scrX_3$.
\end{lem}
\begin{proof}
The proof is similar to that of Lemma \ref{centralb2}.
\end{proof}


Recall $b_{i,r}$ and $h_{i,r}$ for $i\in\{1,2\}$,  $r\in\bN$ from \eqref{bodd}, \eqref{hodd}, and \eqref{bhcom}.
\begin{prop}\label{X=3prop}
For $i,j\in\{1,2\}$ and $r,s\in\bN$, we have 
\begin{align}
&[h_{i,r},h_{j,s}]=0,\qquad h_{i,r}=(-1)^{r+1}h_{\tau i,r},\label{hhb3}\\
&[b_{i,r+1},b_{j,s}]  - [b_{i,r },b_{j,s+1 }]  =\frac{c_{ij}}{2}\{b_{i,r },b_{j,s}\}-2\delta_{i,\tau j} (-1)^{r}h_{\tau i,r+s+1 },\label{bbb3}\\
&[h_{i,r+2},b_{j,s}] - [h_{i,r},b_{j,s+2}]=\frac{c_{ij}-c_{i,\tau j}}{2}\{h_{i,r+1},b_{j,s}\}\notag\\
&\qquad\qquad\qquad \qquad \qquad \ \  +\frac{c_{ij}+c_{i,\tau j}}{2} \{h_{i,r},b_{j,s+1}\}+\frac{c_{ij}c_{\tau i,j}}{4} [h_{i,r}, b_{j,s}].\label{hbb3}
\end{align}
Here we allow $r\in\bZ$ in the relation \eqref{hbb3} with $h_{i,-1}=1$ and $h_{i,r}=0$ for $r<-1$.
\end{prop}
\begin{proof}
The relation \eqref{hhb3} is clear from \eqref{i=tauih}, \eqref{ddb3}, and the definition of $h_{i,r}$.

The relation \eqref{bbb3} for $i=j$ is obvious from \eqref{e1e1b3} and \eqref{f1f1b3}. Here for \eqref{e1e1b3}, we apply \eqref{fi=tauie} to get the relation \eqref{bbb3} for $i=j=2$. Then we consider the case $(i,j)=(1,2)$. It follows from \eqref{e1f1b3} and $e_1(u)=-f_2(-u)$ (by Lemma \ref{efaiii}) that
\begin{align*}
(u-v)&[f_1(u),f_2(v)]\\
&=\frac{u-v}{u+v}\big(\tl d_1(u)d_2(u)-\tl d_1(-v)d_2(-v)\big)+e_{13}(-v)-f_2(v)f_1(u)+f_{31}(u).
\end{align*}
Using $\tl d_2(v)d_3(v)=\tl d_1(-v)d_2(-v)$ from Lemma \ref{efaiii} and rewriting it in terms of $b_1(u)$ and $b_2(v)$, we have
\beq\label{mm-00}
\begin{split}
\big(u-v+\tfrac12\big)[b_1(u),b_2(v)]
=\frac{u-v+\frac12}{u+v}\Big(\tl d_2(v-\tfrac14)d_3(v-\tfrac14)-\tl d_1(u+\tfrac14)d_2(u+\tfrac14)\Big)\\-e_{13}(-v+\tfrac14)-b_2(v)b_1(u)-f_{31}(u+\tfrac14).
\end{split}
\eeq
Expanding the RHS in the region $|u|\gg |v|$, then the term
\beq\label{996}
\frac{u-v+\frac12}{u+v} \tl d_1(u+\tfrac14)d_2(u+\tfrac14)=\Big(1-2\big(v-\tfrac14\big)\sum_{k\gge 0}(-1)^k\frac{v^k}{u^{k+1}}\Big)\tl d_1(u+\tfrac14)d_2(u+\tfrac14)
\eeq
does not contribute to the coefficients of $u^{-r-1}v^{-s-1}$ with $r\in\bZ$ and $s\in \bN$. For the relation \eqref{bbb3}, we shall only consider the coefficients of $u^{-r-1}v^{-s-1}$ with $r\in\bN$ and $s\in \bN$. Hence we can drop terms like $\tl d_1(u+\tfrac14)d_2(u+\tfrac14)$, $f_{31}(u+\tfrac14)$, $e_{13}(-v+\tfrac14)$ in \eqref{mm-00} and this gives rise to
\[
\big(u-v+\tfrac12\big)[b_1(u),b_2(v)]\simeq \frac{u-v+\frac12}{u+v} \tl d_2(v-\tfrac14)d_3(v-\tfrac14)-b_2(v)b_1(u).
\]
Thus we have
\[
(u-v)[b_1(u),b_2(v)]\simeq\Big(1-\frac{2( v-\tfrac14)}{u+v}\Big)\tl d_2(v-\tfrac14)d_3(v-\tfrac14)-\frac{1}{2}\big\{b_{1}(u),b_2(v)\big\},
\]
which implies further
\[
(u-v)[b_1(u),b_2(v)]\simeq -\frac{1}{2}\big\{b_1(u),b_2(v)\big\}-\frac{2v}{u+v} h_2(v).
\]
By taking the coefficients of $u^{-r-1}v^{-s-1}$ for $r,s\in\bN$, we obtain the relation \eqref{bbb3} for $(i,j)=(1,2)$.

We consider the relation \eqref{hbb3}. Note that $h_1(u)=h_2(-u)$ by \eqref{i=tauih} and the anti-automorphism $\eta$ sends $b_1(u)$ to $-b_2(-u)$; see \eqref{etahb}. It suffices to consider the case $(i,j)=(1,1)$. It follows from \eqref{d1f1b3} and \eqref{d2f1b3} that
\beq\label{1helper}
\begin{split}
[\tl d_1(u)&d_2(u),f_1(v)]\\&=\frac{2}{u-v}\tl d_1(u)(f_1(v)-f_1(u))d_2(u)+\frac{1}{u+v}\tl d_1(u)d_2(u)(f_1(v)+e_2(u)).    
\end{split}
\eeq
Thus, we have
\beq\label{2hper}
[\tl d_1(u)d_2(u),f_1(v)]\simeq\frac{2}{u-v}\tl d_1(u)f_1(v)d_2(u)+\frac{1}{u+v}\tl d_1(u)d_2(u)f_1(v),
\eeq
where we dropped terms like $\tl d_1(u)f_1(u)d_2(u)$. By \eqref{d1f1b3}, we find
\[
\frac{1}{u-v}\tl d_1(u)f_1(v)\simeq \frac{1}{u-v-1}f_1(v)\tl d_1(u).
\]
Plugging it into \eqref{2hper}, clearing the denominator, substituting $u\to u+\frac14$ and $v\to v+\frac14$ and multiplying both sides by $\sqrt{-1}\big(1+\frac{1}{4u}\big)$, we obtain
\begin{align*}
\big(u+v+\tfrac{1}{2}\big)&\big(u-v-1\big)\big[h_1(u),b_1(v)\big]\\
\simeq\,&2 \big(u+v+\tfrac{1}{2}\big)b_1(v)h_1(u)+ \big(u-v-1\big)h_1(u)b_1(v),
\end{align*}
which gives rise to
\begin{align*}
\big(u^2-v^2\big)\big[h_1(u),b_1(v)\big]
\simeq \frac{3u+v}{2} \big\{h_1(u),b_1(v)\big\}-\frac{1}{2}[h_1(u),b_1(v)].                               
\end{align*}
Taking the coefficients of $u^{-r-1}v^{-s-1}$  with $r\in\bZ,s\in\bN$, one finds the relation \eqref{hbb3} for $(i,j)=(1,1)$. Here we can make $r\in\Z$ as the terms we dropped give rise to $u^{k}v^{l}$ for $l\gge 0$ if we expand the denominator in the region $|u|\gg |v|$ as exhibited in \eqref{996}.
\end{proof}


We need one more lemma for the Serre relations among generators of degree zero, which will be sufficient to deduce more general Serre relations; see Proposition \ref{prop-serre}.
\begin{prop}\label{zeroserre}
We have
\[
\big[b_{1,0},[b_{1,0},b_{2,0}]\big]=4b_{1,0},\quad \big[b_{2,0},[b_{2,0},b_{1,0}]\big]=4b_{2,0}.
\]
\end{prop}
\begin{proof}
We only prove the first relation. The other one is similar or can be obtained by taking the anti-automorphism $\eta$ from Lemma \ref{lem:eta}.

It follows from  \eqref{bodd}, \eqref{e1f1b3}, and Lemma \ref{efaiii} that
\[
[b_{1,0},b_{2,0}]=-[f_1^{(1)},f_2^{(1)}]=-[f_1^{(1)},e_1^{(1)}]=-h_{1,0}+\frac{1}4+f_{31}^{(1)}.
\]
On the other hand, setting $i=2,j=k=3,l=1$ in \eqref{bcom}, we find
\[
[s_{23}(u),s_{31}(v)]=\frac{1}{u-v}\big(s_{33}(u)s_{21}(v)-s_{33}(v)s_{21}(u)\big),
\]
which implies further
\begin{align*}
[b_{1,0},f_{31}^{(1)}]&=\sqrt{-1}\,[f_{1}^{(1)},f_{31}^{(1)}]=\sqrt{-1}\,[e_{2}^{(1)},f_{31}^{(1)}]\\&=\sqrt{-1}\,[s_{23}^{(1)},s_{31}^{(1)}]=\sqrt{-1}\,s_{21}^{(1)} =\sqrt{-1}\,f_1^{(1)}=b_{1,0}.
\end{align*}
Since by \eqref{hbb3} that $[h_{1,0},b_{1,0}]=3b_{1,0}$, we deduce from above that $\big[b_{1,0},[b_{1,0},b_{2,0}]\big]=4b_{1,0}$.
\end{proof}

\begin{cor}\label{X=3cor}
We have for $i\neq j, i,j\in \{1,2\}$,
\begin{align*}
\mathrm{Sym}_{k_1,k_2}\big[b_{i,k_1},[b_{i,k_2},b_{j,r}] \big] =\frac{4}{3}\,\mathrm{Sym}_{k_1,k_2}(-1)^{k_1}\sum_{p=0}^{k_1+r}3^{-p}[b_{i,k_2+p},h_{\tau i,k_1+r-p}].
\end{align*}
\end{cor}

\begin{proof}
This identity follows from Propositions~\ref{X=3prop}-\ref{zeroserre} using Proposition~\ref{prop-serre}.
\end{proof}

\subsection{Relations in $\scrX_5$} The calculations for \eqref{qsconj0}-\eqref{qsconj9} in $\scrX_5$ are very similar to the previous subsections, while the relation \eqref{qsconj10} in $\scrX_5$ follows by applying the shift homomorphism $\psi_1$ (see Proposition~\ref{prop:red}, Lemma \ref{shiftlem}, and Corollary~\ref{X=3cor}); also note that \eqref{qsconj8} does not show up in $\scrX_5$. Here we only provide details for the verification of the relation \eqref{qsconj9}.


\begin{lem}\label{mmm}
We have the following relations in $\scrX_5$,
\begin{align*}
\big[e_1(u),[e_1(v),e_2(w)]\big]+\big[e_1(v),[e_1(u),e_2(w)]\big]=0,\\
\big[e_2(u),[e_2(v),e_1(w)]\big]+\big[e_2(v),[e_2(u),e_1(w)]\big]=0.
\end{align*}
\end{lem}
\begin{proof}
Again, one shows the same relations as in Lemmas \ref{be1e2lem} and \ref{lemhelp1}. Then the verification of the first Serre relation follows from exactly the same arguments as in Lemmas \ref{be1e2lem}--\ref{sec:A-serre}.

For the second Serre relation, we have the following identities; cf. Lemma \ref{help1},
\begin{align}
&[e_{13}(u),e_2^{(1)}]=0,\label{b4pf11-}\\
&[e_1^{(1)},\tl d_2(v)]=-e_1(v)\tl d_2(v), \label{b4pf10-}\\
&[e_1(u),e_2^{(1)}]=e_{13}(u),\quad [e_1^{(1)},e_2(v)]=e_{13}(v)-e_1(v)e_2(v),\label{b4pf8-}\\
&[e_2^{(1)},e_2(v)]=e_2(v)e_2(v).\label{b4pf9-}
\end{align}
Here \eqref{b4pf11-} follows from
\[
(u-v)[s_{13}(u),s_{23}(v)]=s_{23}(u)s_{13}(v)-s_{23}(v)s_{13}(u).
\]
The relations \eqref{b4pf10-} and \eqref{b4pf8-} are proved the same as the corresponding relations in \eqref{b4pf10} and \eqref{b4pf8}. The relation \eqref{b4pf9-} follows from the relation $(u-v)[e_2(u),e_2(v)]=(e_2(u)-e_2(v))^2$ (which follows from \eqref{e1e1b3} and Lemma \ref{shiftlem}).

Repeating the argument as in Lemma \ref{be13e2lem}, one obtains $[e_{13}(u),e_2(v)]=e_2(v)[e_1(u),e_2(v)]$ which is exactly the relation for non-twisted Yangian of type A; cf. \cite[Lem. 5.5 (iii)]{BK05}. The rest is done similarly as in \cite[Lem. 5.7 (i)]{BK05}.
\end{proof}
\begin{rem}
The lemma above can be proved similarly to Proposition \ref{prop:quasiqua} as follows. The series $e_1(u)$ is expressed in terms of $s_{11}(u)$ and $s_{12}(u)$ while $e_2(u)$ is expressed in terms of $s_{11}(u)$, $s_{13}(u)$, $s_{21}(u)$ and $s_{23}(u)$. Note that the commutator relations between these series $s_{ab}(u)$ are the same as in $\rY(\gl_N)$ (with $s_{ab}(u)$ replaced by $t_{ab}(u)$); see \eqref{Trel} and \eqref{bcom}. Thus these Serre relations hold as the same Serre relations hold in $\rY(\gl_m)$.
\end{rem}

\begin{cor}
We have the following Serre relations in $\scrX_5$,
\begin{align*}
\big[b_i(u),[b_i(v),b_j(w)]\big]+\big[b_i(v),[b_i(u),b_j(w)]\big]=0,
\end{align*}
for the pairs $(i,j)=(1,2),(2,1), (3,4),(4,3)$.
\end{cor}
\begin{proof}
Follows from Lemma \ref{mmm} by applying the anti-automorphism $\eta$ and Lemma \ref{efaiii}.
\end{proof}


\end{document}